\documentclass[reqno,12pt]{article}
\usepackage{enumerate,amsfonts,mathrsfs,amscd,amssymb,amsthm,amsmath}
\usepackage{bm,graphicx,psfrag,subfigure,xypic,xcolor}

\newtheorem{theorem}{Theorem}[section]
\newtheorem{lem}[theorem]{Lemma}
\newtheorem{defn}[theorem]{Definition}
\newtheorem{thm}[theorem]{Theorem}
\newtheorem{prop}[theorem]{Proposition}
\newtheorem{cor}[theorem]{Corollary}

\theoremstyle{definition}

\theoremstyle{remark}

\numberwithin{equation}{section}

\usepackage{bm}

\def\rk{{\rm rk}}

\def\Tor{{\rm Tor}}

\def\Tens{{\rm Tens}}

\def\setminus{\smallsetminus}

\def\supp{{\rm{supp\,}}}
\def\deg{\mbox{\rm deg\,}}

\begin{document}

\title{On the foundations of signed graphs I: chain groups, frame matroid, and bivariate flow polynomial\thanks{Research is supported by RGC Competitive Earmarked Research Grant 16308821.}}

\author{Beifang Chen\\
{\normalsize Department of Mathematics}\\
{\normalsize Hong Kong University of Science and Technology}\\
{\normalsize Clear Water Bay, Hong Kong}\\
{\normalsize\tt mabfchen@ust.hk}
}






\maketitle

\begin{abstract}
This paper studies signed graphs with possible outer-edges. We introduce and investigate the chain group, the boundary operator, the co-boundary operator, the flow group, the tension group, the homology group, the cohomology group, with coefficients in an abelian group. We also introduce and investigate the bivariate flow polynomial for signed graphs. The guiding principle is the correspondence between representable matroids over $\Bbb R$ on a ground set $E$ of edges and the subspaces of the vector space of real-valued chains on the same ground set. The frame matroid of signed graph emerges naturally by defining circuits as minimal supports of nonzero flows, rather than listing circuit patterns abruptly. Likewise, bonds, or co-circuits, can be obtained as minimal supports of nonzero tensions. In addition to standardizing the concepts and their meanings of signed graphs, we update the following results: (1) Characterization of cuts and bonds of signed graph with outer-edges. (2) The structures of flow group, boundary group, and homology group, with coefficients in an arbitrary abelian group. (3) Introduction of  bivariate flow polynomial for signed graph, revealing the mystery of inexistence of univariate flow polynomial of signed
graph in the literature.
\end{abstract}



\section{Introduction}

Signed graphs were introduced by Harary~\cite{Harary} to establish an abstract balance theory for social relations. The systematic study of signed graphs was begun by Zaslavsky~\cite{Zaslavsky-signed-graphs-DAM,
Zaslavsky-signed-graph-coloring-DM, Zaslavsky-orientation-EJC},
who developed fundamental concepts such as circuits, bonds, frame
matroids, colorings, flows, Laplacian matrices and a weighted counting formula on the number of bases. It was emphasized in
\cite[p.\,53]{Zaslavsky-signed-graphs-DAM} that the central observation
was the existence of the signed-graphic matroid, known as {\em frame matroid}, whose circuits were introduced by listing their patterns abruptly without explanation and motivation. We wonder whether there exists a guiding principle to naturally obtain the frame matroid for a signed graph. Rather than checking the circuit axioms for listed circuit patterns in the literature case by case, could it be possible to directly define the frame matroid for a signed graph and deduce the circuit patterns naturally?

The answer is affirmative. Actually, Bouchet \cite{Bouchet} might have noticed a unified approach from the chain group of Tutte \cite{Tutte-Matroid-Theory} to obtain circuits of a signed graph, but his interest was on the flow numbers of signed graphs, leading to his celebrated $6$-flow conjecture. However, Tutte's chain group is not intrinsic when it is applied to graphs and simplicial complexes, for it depends on the chosen orientation. The observations of Bouchet \cite{Bouchet} (for signed graph without outer-edges) and Chen \cite{Chen-GC} (for signed graph with outer-edges) provide the answer to the question raised above for the circuit part, and can serve as a guiding principle to the reminiscent question for the bond part.

Based the work of Zaslavsky, Chen and Wang~\cite{Chen-Wang-EJC,Chen-Wang-DAM1} set up the algebraic structures such as the flow lattice (space), the tension lattice (space), and the torsion of the incidence matrix for signed graphs without outer-edges (referred to as half-edges in the literature), parallel to that for ordinary graphs. Later, Chen and Wang \cite{Chen-Wang-arXiv} and Chen, Wang and Zaslavsky \cite{Chen-Wang-Zaslavsky-DM} discovered and characterized the structures of the so-called conformal indecomposable integer-valued flows for signed graphs without outer-edges. More recently, Chen \cite{Chen-GC} extended such structures and characterizations of conformal indecomposable integer flows to signed graphs with outer-edges. The by-product of the characterizations is that the supports of elementary flows of a signed graph are exactly the circuits of the given signed graph, defined frequently by listing their circuit patterns case by case in the literature.

Now, it is clear and straightforward to see, once recognized, that the circuits of a signed graph are just the minimal supports of nonzero flows of the given signed graph; and consequently, the bonds should be and actually are the minimal supports of nonzero tensions (dual to flows) of the signed graph. However, this fact seems unknown in the graph theory community, though it should be well-known. Nevertheless, a unified approach to obtain circuits and bonds for signed graphs is essential for foundational purposes and aesthetic perfection.

In a recent paper by Goodall {\em et all} \cite{Goodall et al}, profound formulas on the number of proper colorings, nowhere-zero flows and tensions over finite abelian groups are studied (for introducing Tutte polynomial for signed graphs without outer-edges), relying on the algebraic setup made by Chen and Wang \cite{Chen-Wang-EJC}. However, we are not satisfied with the setup and notations adopted in the early treatment \cite{Chen-Wang-EJC}. The aim of the present paper is to provide a standard algebraic-topology style foundation for signed graphs, serving as a succinct and efficient setup for future use, avoiding the repetition procedure that has happened for decades in almost all papers about signed graphs on flows. Such an intrinsic approach (with no chosen specific orientation) is still unavailable in the literature to the best of our knowledge.

The fundamental ingredient of a signed graph $\Sigma$ is the oppositely bi-directed orientations on its edges, which dates back to Edmonds and Johnson \cite{Edmonds-Johnson}. The bi-directed edges generate a $1$-chain group ${\rm C}_1(\Sigma,{\Bbb A})$, with coefficients in an arbitrary abelian group $\Bbb A$, identifying each oriented edge to the negative of its oppositely oriented edge. The set of vertices generates a $0$-chain group ${\rm C}_0(\Sigma,{\Bbb A})$ with coefficients in $\Bbb A$. The significant role is the boundary operator
\[
\partial: {\rm C}_1(\Sigma,{\Bbb A})\rightarrow {\rm C}_0(\Sigma,{\Bbb A}),
\]
sending each oriented edge to the addition of its two endpoints (boundary points) with positive (resp. negative) sign when the oriented edge points to (resp. away from) the endpoint, extended through linearity. The kernel of $\partial$ defines the flow group ${\rm F}(\Sigma,{\Bbb A})$. Our first purpose is to characterize the group structures of ${\rm F}(\Sigma,{\Bbb A})$, the boundary chain subgroup ${\rm im}\,\partial$, and the homology group ${\rm H}_0(\Sigma,{\Bbb A}):={\rm C}_0(\Sigma,{\Bbb A})/{\rm im}\,\partial$.

As for producing a guiding principle, our idea originates from
Tutte \cite{Tutte-Homotopy Theorem, Tutte-Matroid-Theory}, who introduced (non-intrinsically for a graph by fixing an orientation) elementary chains for chain groups. The collection of supports of elementary chains straightforwardly forms a circuit system of a matroid on the ground set of the chain group. Taking supports of real-valued elementary flows of ${\rm F}(\Sigma,{\Bbb R})$ as circuits, we automatically obtain the frame matroid $\mathcal{M}(\Sigma)$ of signed graph $\Sigma$. It is then natural to characterize the signed-graph circuits into five patterns (see Theorem~\ref{thm:circuits}) listed by Zaslavsky \cite[p.\,54]{Zaslavsky-signed-graphs-DAM} case by case
without motivation and explanation.

The vector space ${\rm C}_1(\Sigma,{\Bbb R})$ is equipped with an inner product in an obvious way. The orthogonal complement of ${\rm F}(\Sigma,{\Bbb R})$ is the tension space ${\rm T}(\Sigma,{\Bbb R})$. The collection of supports of elementary tensions forms another matroid $\mathcal{M}^*(\Sigma)$, which is dual to $\mathcal{M}(\Sigma)$ automatically according to the equivalence of matroid duality and vector space orthogonality; see Proposition~\ref{Prop:Duality-and-othogonality}. Our next purpose is to characterize the circuits of $\mathcal{M}^*(\Sigma)$, that is, the co-circuits of $\mathcal{M}(\Sigma)$, in particular when $\Sigma$ contains possible outer-edges. We must also confirm that the characterization is exactly the bond defined by Zaslavsky \cite[p.\,363]{Zaslavsky-orientation-EJC}, where the detailed structural description was missing in the literature.

Tensions for a graph were initially considered by Berge \cite{Berge1} to describe colorings in terms of edges, and were formally introduced by Kochol \cite{Kochol1}, defined at each arc as the difference of potential functions at the endpoints of the arc. Later, Kochol \cite{Kochol3} modified tension (of a graph) as a chain whose summation over each directed circuit vanishes. This suggests that conceptually, tension is dual to flow. Since there is no inner product in the chain group ${\rm C}_1(\Sigma,{\Bbb A})$, we follow the standard treatment in algebraic topology to introduce the obviously defined bilinear pairing
\[
\langle\,,\rangle:{\rm C}_i(\Sigma,{\Bbb Z})\times{\rm C}_i(\Sigma,{\Bbb A})\rightarrow {\Bbb A}, \quad i=0,1.
\]
The tension group ${\rm T}(\Sigma,{\Bbb A})$ is the subgroup of chains $\bf g$ with coefficients in $\Bbb A$, such that for all integer-valued flows $\bf f$,
\[
\langle{\bf f},{\bf g}\rangle=0.
\]
Adjoint to $\partial$ is the co-boundary operator
\[
\delta:{\rm C}_0(\Sigma,{\Bbb A})\rightarrow{\rm C}_1(\Sigma,{\Bbb A}).
\]
Our next purpose is to characterize the structure of ${\rm T}(\Sigma,{\Bbb A})$, the kernel ${\rm ker}\,\delta$, the co-boundary group
${\rm im\,}\delta$, and the co-homology group ${\rm H}^1(\Sigma,{\Bbb A}):={\rm C}_1(\Sigma,{\Bbb A})/{\rm im\,}\delta$.

Let $\Bbb A$ be a finite abelian group. Its {\em torsion-$2$ subgroup} is the kernel of the homomorphism ${\Bbb A}\rightarrow 2{\Bbb A},a\mapsto2a$, denoted
${\rm Tor}_2({\Bbb A})$. It is well-known that the number of nowhere-zero flows of a graph $G$ over $\Bbb A$, has a polynomial counting pattern $\varphi(G,|{\Bbb A}|)$ in terms of the cardinality $|\Bbb A|$, regardless of the group structures of $\Bbb A$. The polynomial function $\varphi(G,x)$, determined by the polynomial counting pattern, is known as the {\em flow polynomial} of $G$; see Tutte \cite{Tutte-class-of-abelian-group}. For a signed graph $\Sigma$, however, the counting function of the number of nowhere-zero flows of $\Sigma$ over $\Bbb A$ depends not only on $|\Bbb A|$, but also on the order of its torsion-2 subgroup ${\rm Tor}_2(\Bbb A)$. We are motivated to introduce a unique bivariate flow polynomial $\varphi(\Sigma;t,x)$ for $\Sigma$, satisfying
\[
|{\rm F}_{\rm nz}(\Sigma,{\Bbb A})|=\varphi(\Sigma;|{\rm Tor}_2{\Bbb A}|,|{\Bbb A}|),
\]
where ${\rm F}_{\rm nz}(\Sigma,{\Bbb A})$ is the set of nowhere-zero
flows over $\Bbb A$. The polynomial $\varphi(\Sigma;t,x)$ reduces to a univariate polynomial $\varphi(\Sigma,x)$, whenever $\Sigma$ contains no negative circles. Set
\[
\varphi_{\rm odd}(\Sigma,x):=\varphi(\Sigma;1,x), \quad \varphi_{\rm even}(\Sigma,x):=\varphi(\Sigma;2,x).
\]
We see that $\varphi_{\rm odd}(\Sigma,n)$ and $\varphi_{\rm even}(\Sigma,n)$ count the number of nowhere-zero flows of $\Sigma$ over ${\Bbb Z}/n{\Bbb Z}$, with respect to odd integers $n\geq 1$ and to even integers $n\geq 2$. The polynomial $\varphi_{\rm even}(\Sigma,n)$ answers the question of Beck and Zaslavsky \cite[Problem 4.2]{Beck-Zaslavsky-JCTB}.

Due to the length of the present paper, we postpone the treatment of the tension group, the co-boundary group, the co-homology group, the bivariate chromatic polynomial, the bivariate tension polynomial, the multivariate Tutte polynomial and further studies on signed graphs with possible outer-edges to subsequent papers in this series.

\section{Definitions and notations}

Following Zaslavsky \cite{Zaslavsky-signed-graph-coloring-DM}, by a {\em graph} we mean a system $G=(V,E)$, with a finite vertex set $V$ and a finite edge set $E$, such that each edge has exactly two {\em ends}, each is connected to one or none vertex, but at least one of the two ends is connected to a vertex. A {\em loop} is an edge whose two ends are connected to a common vertex. A {\em link} is an edge whose two ends are connected to two distinct vertices. An {\em outer-edge} is an edge having only one end connected to a vertex. The {\em endpoints} of an edge are the vertices to which the ends of the edge are connected. A {\em signed graph} is a graph $G=(V,E)$ together with a {\em sign function} $\sigma:E\rightarrow\{-1,1\}$, denoted $\Sigma=(G,\sigma)$. Edges with positive (negative) sign are called {\em positive (negative) edges}. A signed graph is called {\em compact} if it contains no outer-edges, and {\em noncompact} if it contains outer-edges.

An {\em open line} is a connected 2-regular signed graph containing exactly two outer-edges. A {\em half-line} is a connected signed graph having degree 1 at exactly one vertex (called the {\em endpoint} of the half-line) and degree 2 at all other vertices. Each half-line contains exactly one outer-edge. A {\em line segment} is a connected signed graph having degree 1 at exactly two distinct vertices (called {\em endpoints} of the line segment) and degree 2 at all other vertices. In particular, a loop with its endpoint is a circle, an outer-edge with its endpoint is a half-line, and a link edge with its endpoints is a line segment. By a {\em line} we mean either an open line or a half-line or a line segment.

A {\em circle}\footnote{The name was suggested by Thom Zaslavsky to avoid using cycle which has too many other meanings.} (commonly known as {\em cycle}) is a connected 2-regular (degree 2 everywhere) signed graph containing no outer-edges. The {\em sign} of a circle $C$ is the product of its edge signs, denoted
\[
\sigma(C):=\prod_{e\in C}\sigma(e).
\]
A circle $C$ is called {\em positive} ({\em negative}) if $\sigma(C)=1$ ($\sigma(C)=-1$). A signed graph is called {\em balanced} if it does not contain outer-edges and its all circles are positive.
For signed graphs we immediately have the numerical invariants
\begin{align*}
c(\Sigma)&=\mbox{number of components of $\Sigma$,}\\
b(\Sigma)&=\mbox{number of balanced components of $\Sigma$,}\\
u(\Sigma)&=\mbox{number of unbalanced components of $\Sigma$,}\\
u_c(\Sigma)&=\mbox{number of compact unbalanced components of $\Sigma$.}
\end{align*}
For each $S\subseteq E$ edge subset of $\Sigma$, we define $c(S)$, $b(S)$, $u(S)$, and $u_c(S)$ as the corresponding numbers of the spanning signed subgraph $\Sigma|S:=(V,S,\sigma|_S)$.

By an {\em orientation} or {\em bi-direction} of an edge $e$ of $\Sigma$ we mean an assignment of two arrows, one to each of two ends of $e$, such that the two arrows are in the same direction if $e$ is positive and in opposite directions if $e$ is negative. Each edge $e$ has exactly two orientations, opposite each other; see Figure~\ref{Fig:Orientation}.
\begin{figure}[h]
\centering
\includegraphics[width=15mm]{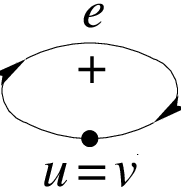}\hspace{11.5mm}
\includegraphics[width=15mm]{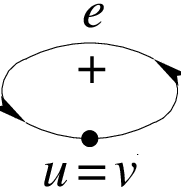}\hspace{11.5mm}
\includegraphics[width=15mm]{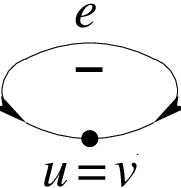}\hspace{11.5mm}
\includegraphics[width=15mm]{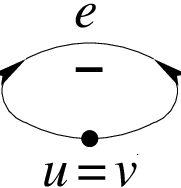}\\ \vspace{3mm}
\includegraphics[width=22mm]{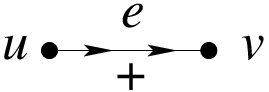}\hspace{5mm}
\includegraphics[width=22mm]{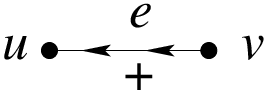}\hspace{5mm}
\includegraphics[width=22mm]{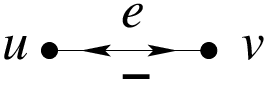}\hspace{5mm}
\includegraphics[width=22mm]{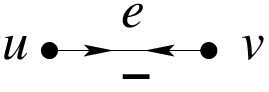}\\ \vspace{3mm}
\includegraphics[width=18.1mm]{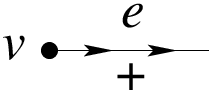}\hspace{8mm}
\includegraphics[width=18.1mm]{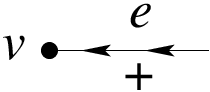}\hspace{8mm}
\includegraphics[width=18.1mm]{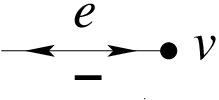}\hspace{8mm}
\includegraphics[width=18.1mm]{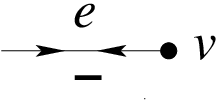}
\caption{Orientations of loop, link and outer-edge}\label{Fig:Orientation}
\end{figure}
More specifically, in the case that $e$ is a link or loop with endpoints $u,v$ (possibly identical), we often write $e=uv$ for convenience, and the two orientations of $e$ are expressed as $\overset{\rightarrow}{u}\overset{\rightarrow}{v},\overset{\leftarrow}{u}\overset{\leftarrow}{v}$ if $e$ is positive, and as $\overset{\leftarrow}{u}\overset{\rightarrow}{v},\overset{\rightarrow}{u}\overset{\leftarrow}{v}$ if $e$ is negative. In the case that $e$ is an outer-edge with only one endpoint $v$, we write $e=vo$ (where $o$ stands for ``outside"), the two orientations of $e$ are expressed as
$\overset{\rightarrow}{v}\overset{\rightarrow}{o},\overset{\leftarrow}{v}
\overset{\leftarrow}{o}$ if $e$ is positive, and as
$\overset{\leftarrow}{v}\overset{\rightarrow}{o}, \overset{\rightarrow}{v}\overset{\leftarrow}{o}$ if $e$ is negative. Each edge with an orientation is frequently called an {\em arc} or {\em oriented edge} or {\em directed edge}. For each edge $e$, we denote by $\vec{e\,}$ the edge $e$ together with an orientation (selected arbitrarily); the same edge $e$ with the other orientation is then denoted by $\vec{e\;}^{-1}$, called the {\em opposite arc} of $\vec{e}$. For each $S$ edge subset of $\Sigma$,
 the collection $\Vec{S}:=\{\vec{e},\vec{e\;}^{-1}\:|\:e\in S\}$ is called the {\em arc set} or {\em oriented edge set} of $S$.
\vspace{1ex}

\noindent
{\bf Remark.} We may view or identify each edge as an embedding in a Euclidean space. More specifically, the imbedding is an injective continuous map $\iota_e$ of $[0,1]$ the closed unit interval to ${\Bbb R}^d$ for each link or loop $e$, and of $[0,1)$ the half-closed and half-open interval to ${\Bbb R}^d$ for each outer-edge $e$. Then $e$ is inherited a topology from the open interval $(0,1)$ of $\Bbb R$ via the imbedding $\iota_e$. The graph $G$, viewed as disjoint union of the vertex set $V$ and its (open) edges, inherits a topology as follows: a subset $\mathcal{U}$ of $G$ is an {\em open set} provided that (i) for each edge $e$ the intersection $\mathcal{U}\cap e$ is open in $e$; (ii) for each vertex $v\in \mathcal{U}\cap V$ and each edge $e$ at $v$ with $v=\iota_e(i)$, where $i\in\{0,1\}$, there exist real numbers $a,b\in(0,1)$ such that $\iota_e(0,a)\subset\mathcal{U}$ if $i=0$, and $\iota_e(b,1)\subset\mathcal{U}$ if $i=1$. Hence the graph $G$ (viewing as a topological space) is compact if it does not contain outer-edges, and is non-compact if it contains outer-edges.

The {\em incidence function} of $\Sigma$ is a function $[\,,]:V\times\vec{E}\rightarrow{\Bbb Z}$, satisfying $[v,\vec{e\;}^{-1}]=-[v,\vec{e\,}]$, determined uniquely by
\begin{equation}\label{eq:incidence-function}
[v, \vec{e\,}]=\left\{\begin{array}{rl}
1 & \mbox{if $\vec{e\,}$ is a non-loop arc pointing to $v$,}\\
2 & \mbox{if $\vec{e\,}$ is a negative loop pointing to $v$,}\\
0 & \mbox{otherwise.}
\end{array}\right.
\end{equation}
An {\em orientation} on an edge subset $S$ of $\Sigma$ is a function $\omega:S\rightarrow\vec{S}$ such that $\omega(e)\in\{\vec{e\,},\vec{e\;}^{-1}\}$. Then $\omega(S)$ is a subset of $\vec{S}$ such that
\[
\omega(S)\cap\omega(S)^{-1}=\varnothing, \quad \omega(S)\cup\omega(S)^{-1}=\vec{S},
\]
where $\omega(S)^{-1}=\{\vec{e}^{\;-1}: \vec{e\,}\in\omega(S)\}$. For simplicity, we often write $\omega=\omega(S)$ and view $\omega$ as a subset of $\vec{S}$. An {\em orientation} of $\Sigma$ is an orientation on the edge set $E$. The {\em incidence matrix} of $\Sigma$, relative to an orientation $\omega$ of $\Sigma$, is a $V\times E$ matrix
\begin{equation}\label{eq:incidence-matrix}
M=\big[m_{ve}\big]_{(v,e)\in V\times E}\;\; \mbox{with}\;\;m_{ve}=[v,\vec{e\,}],\; \vec{e}\in\omega.
\end{equation}

Let $\Sigma$ be a signed graph and $\Bbb A$ be an abelian group throughout. Denote by ${\rm Tor}_2({\Bbb A})$ the torsion-2 subgroup of $\Bbb A$, consisting of members $a\in\Bbb A$ such that $2a=0$. A {\em $1$-chain} of $\Sigma$ over $\Bbb A$ is a function ${\bf c}:\vec{E}\rightarrow\Bbb A$ such that ${\bf c}(\vec{e\;}^{-1})=-{\bf c}(\vec{e\,})$ for each edge $e$ of $\Sigma$. It is convenient to write each 1-chain $\bf c$ as a formal sum
\[
{\bf c}=\sum_{e\in E}{\bf c}(\vec{e\,})\vec{e\,},
\]
where terms with coefficients zero are omitted in practice. For example, given $S$ an edge subset with an orientation $\omega(S)$; each multiset $(S,m)$, where $m:S\rightarrow{\Bbb Z}_{\geq 0}$ is the multiplicity function, defines an integer-valued 1-chain
\[
{\bf I}_{\omega(S,m)}= \sum_{\vec{e\,}\in\omega(S)}m(e)\,\vec{e\,}.
\]
The set ${\rm C}_1(\Sigma,{\Bbb A})$ of $1$-chains of $\Sigma$ over $\Bbb A$ forms an abelian group, called the {\em $1$-chain group} or just {\em chain group} of $\Sigma$. Let ${\rm C}_0(\Sigma,{\Bbb A})$ denote the group of functions of the vertex set $V$ to $\Bbb A$, whose members are known as {\em $0$-chains} of $\Sigma$.

The {\em boundary operator} $\partial:{\rm C}_1(\Sigma,{\Bbb Z})\rightarrow {\rm C}_0(\Sigma,{\Bbb Z})$ is a homomorphism defined by
\begin{equation}
\partial\vec{e\,}
=\left\{\begin{array}{ll}
v-u & \mbox{if $\vec{e\,}=\overset{\rightarrow}{u}\overset{\rightarrow}{v}$,}\\
u+v & \mbox{if $\vec{e\,}=\overset{\leftarrow}{u}\overset{\rightarrow}{v}$,}\\
v & \mbox{if $\vec{e\,}=\overset{\leftarrow}{v}\overset{\,\thicksim}{o}$,}
\end{array}\right.
\end{equation}
where $\overset{\thicksim}{o}$ denotes either $\overset{\rightarrow}{o}$ or $\overset{\leftarrow}{o}$, extended through linearity. The operator $\partial$ can be canonically extended to a homomorphism
\[
\partial:{\rm C}_1(\Sigma,{\Bbb A})\rightarrow {\rm C}_0(\Sigma,{\Bbb A})
\]
by taking tensor of ${\rm C}_1(\Sigma,{\Bbb Z})$ and ${\rm C}_0(\Sigma,{\Bbb Z})$ with $\Bbb A$. More specifically, the boundary $\partial{\bf c}$ of 1-chain ${\bf c}\in{\rm C}_1(\Sigma,{\Bbb A})$ is explicitly given by
\begin{equation}
(\partial{\bf c})(v)=\sum_{e\in E}[v,\vec{e}\,]{\bf c}(\vec{e}\,).
\end{equation}
The {\em cycle group} ${\rm Z}_1(\Sigma,{\Bbb A})$, the {\em boundary group} ${\rm B}_0(\Sigma,{\Bbb A})$, and the {\em $0$th homology group} ${\rm H}_0(\Sigma,{\Bbb A})$ of $\Sigma$ over $\Bbb A$ are defined by
\begin{align}
{\rm Z}_1(\Sigma,{\Bbb A}):&=\ker\partial,\\
{\rm B}_0(\Sigma,{\Bbb A}):&={\rm im}\,\partial,\\
{\rm H}_0(\Sigma,{\Bbb A}):&={\rm C}_0(\Sigma,{\Bbb A})/{\rm B}_0(\Sigma,{\Bbb A}).
\end{align}
Each member of ${\rm Z}_1(\Sigma,{\Bbb A})$ is called a {\em cycle} of $\Sigma$ over $\Bbb A$; and each member ${\bf b}$ of ${\rm B}_0(\Sigma,{\Bbb A})$, written as ${\bf b}=\partial{\bf c}$, is called the {\em boundary} of the 1-chain $\bf c$. According to the chain complex
\[
0\rightarrow{\rm C}_1(\Sigma,{\Bbb A})\stackrel{\partial}{\rightarrow}{\rm C}_0(\Sigma,{\Bbb A})\rightarrow 0,
\]
we have the {\em $1$-homology group} ${\rm H}_1(\Sigma,{\Bbb A})={\rm Z}(\Sigma,{\Bbb A})$. In graph theory, a cycle is usually called a {\em flow}, the cycle group is called the {\em flow group}, and ${\rm Z}_1(\Sigma,{\Bbb A})$ is denoted by ${\rm F}(\Sigma,{\Bbb A})$.

Adjoint to $\partial$ is the {\em co-boundary operator} $\delta:{\rm C}_0(\Sigma,{\Bbb A})\rightarrow{\rm C}_1(\Sigma,{\Bbb A})$, defined by
\begin{align}
\delta{\bf p}(\vec{e}\,) &
= \left\{\begin{array}{ll}
{\bf p}(v)-{\bf p}(u) & \mbox{if\, $\vec{e}\,=\overset{\rightarrow}{u}\overset{\rightarrow}{v}$,}\\
{\bf p}(u)+{\bf p}(v) & \mbox{if\, $\vec{e}\,=\overset{\leftarrow}{u}\overset{\rightarrow}{v}$,}\\
{\bf p}(u) & \mbox{if\, $\vec{e}\,=\overset{\leftarrow}{u}\overset{\,\thicksim}{o}$.}
\end{array}\right.
\end{align}
Of course, it is assumed that $\delta{\bf p}(\vec{e\;}^{-1}) = -\delta{\bf p}(\vec{e}\,)$. Analogous to $\partial$, the co-boundary chain $\delta{\bf p}$ (as a function on $\vec{E}$) can be uniformly written as
\begin{equation}
\delta{\bf p}(\vec{e}\,)= \sum_{v\in V}[v,\vec{e}\:]{\bf p}(v).
\end{equation}
The {\em co-cycle group} ${\rm Z}^0(\Sigma,{\Bbb A})$, the {\em co-boundary group} ${\rm B}^1(\Sigma,{\Bbb A})$, and the {\em co-homology group} ${\rm H}^1(\Sigma,{\Bbb A})$ of $\Sigma$ over $\Bbb A$ are defined by
\begin{align}
{\rm Z}^0(\Sigma,{\Bbb A}):&=\ker\delta,\\
{\rm B}^1(\Sigma,{\Bbb A}):&={\rm im}\,\delta,\\
{\rm H}^1(\Sigma,{\Bbb A}):&={\rm C}^1(\Sigma,{\Bbb A})/{\rm B}^1(\Sigma,{\Bbb A}),
\end{align}
according to the co-chain complex
\[
0\rightarrow{\rm C}^0(\Sigma,{\Bbb A})\stackrel{\delta}\rightarrow {\rm C}^1(\Sigma,{\Bbb A})\rightarrow 0,
\]
where ${\rm C}^i(\Sigma,{\Bbb A}):={\rm C}_i(\Sigma,{\Bbb A})$, $i=0,1$.
When ${\Bbb A}={\Bbb Z}$ or ${\Bbb R}$, the corresponding groups are known as {\em lattices} and vector spaces respectively.

In order to define tension on $\Sigma$ over $\Bbb A$, a notion dual to flow conceptually, we need to make use of certain orthogonality over $\Bbb A$. Since there is no way to equip an inner product on ${\rm C}_1(\Sigma,{\Bbb A})$, it is natural to consider the canonical {\em bilinear pairing}
\begin{equation}
\langle\,,\rangle:{\rm C}_i(\Sigma,{\Bbb Z})\times {\rm C}_i(\Sigma,{\Bbb A})\rightarrow{\Bbb A}, \quad i=0,1,
\end{equation}
defined for ${\bf c}\in{\rm C}_i(\Sigma,{\Bbb Z})$ and ${\bf d}\in{\rm C}_i(\Sigma,{\Bbb A})$ by
\[
\langle{\bf c},{\bf d}\rangle=\left\{\begin{array}{ll}
\sum_{v\in V} {\bf c}(v){\bf d}(v) & \mbox{if $i=0$,}\vspace{2mm}\\
\sum_{e\in E} {\bf c}(\vec{e}\,){\bf d}(\vec{e}\,) & \mbox{if $i=1$.}
\end{array}\right.
\]
Now we are ready to introduce the notion of tension of $\Sigma$, extending the potential difference of Berge \cite{Berge1} and Kochol \cite{Kochol1} for graphs.

\begin{defn}
A {\em tension} of $\Sigma$ over $\Bbb A$ is a $1$-chain ${\bf g}$ such that $\langle{\bf f},{\bf g}\rangle=0$ for all integer-valued flows $\bf f$ of $\Sigma$. The set ${\rm T}(\Sigma,{\Bbb A})$ of tensions of $\Sigma$ forms a group, called the {\em tension group} of $\Sigma$ over $\Bbb A$.
\end{defn}

The following Lemma~\ref{Lem:Potential-is-tension} explains why the co-boundary operator $\delta$ is adjoint to the boundary operator $\partial$.

\begin{lem}\label{Lem:Potential-is-tension}
$\langle\partial{\bf c},{\bf p}\rangle
=\langle{\bf c},\delta{\bf p}\rangle$,
${\rm im}\,\delta\subseteq{\rm T}(\Sigma,{\Bbb A})$.
\end{lem}
\begin{proof}
Recall that $(\partial{\bf c})(v)=\sum_{e\in E}[v,\vec{e}\,]{\bf c}(\vec{e}\,)$ and $(\delta{\bf p})(\vec{e\,})=\sum_{v\in V}{\bf p}(v)[v,\vec{e\,}]$. Then
\begin{align*}
\langle\partial{\bf c},{\bf p}\rangle &=\sum_{v\in V}
(\partial{\bf c})(v){\bf p}(v)\\
&=\sum_{v\in V}\sum_{e\in E}[v,\vec{e}\,]{\bf c}(\vec{e}\,){\bf p}(v)\\
&= \sum_{e\in E} {\bf c}(\vec{e}\,) \sum_{v\in V}[v,\vec{e}\,]{\bf p}(v)\\
&= \sum_{e\in E} {\bf c}(\vec{e}\,)(\delta{\bf p})(\vec{e}\,)\\
&=\langle{\bf c},\delta{\bf p}\rangle.
\end{align*}
If ${\bf c}$ is any integer-valued flow, that is, $\partial{\bf c}=0$, then $\langle{\bf c},\delta{\bf p}\rangle=\langle\partial{\bf c},{\bf p}\rangle=0$. This means that $\delta{\bf p}$ is a tension over $\Bbb A$; consequently, ${\rm im}\,\delta\subseteq{\rm T}(\Sigma,{\Bbb A})$.
\end{proof}

We call each function $\nu:V\rightarrow\{-1,1\}$ a {\em switching} on $\Sigma$. A {\em switching} at a vertex $u$ is the function $\nu$ such that $\nu(u)=-1$ and $\nu(v)=1$ for all $v\neq u$. For each $\nu$ switching, let $\Sigma^\nu$ denote the signed graph $(V,E,\sigma^\nu)$, whose sign function is given by
\[
\sigma^\nu(e)=\left\{\begin{array}{ll}
\nu(u)\sigma(e)\nu(v) & \mbox{if $e=uv$ is a loop or link,} \\
\nu(u)\sigma(e)       & \mbox{if $e=uo$ is an outer-edge.}
\end{array}\right.
\]
Each arc $\vec{e}$ of $\Sigma$, after a switching $\nu$ at $v$, becomes an arc $\vec{e}^{\;\nu}$ of $\Sigma^\nu$. More specifically, if $v$ is an endpoint of $e$ then the arrow of $\vec{e}^{\;\nu}$ at $v$ is opposite to that of $\vec{e}$; if $v$ is not an endpoint of $e$ then the arrow of $\vec{e}^{\;\nu}$ at $v$ is the same as that of $\vec{e}$. In notation of incidence numbers,
\[
[v,\vec{e}^{\:\nu}]=[v,\vec{e}\,]\nu(v).
\]
Each switching $\nu$ induces an isomorphism ${\rm C}_0(\Sigma,{\Bbb A})\rightarrow{\rm C}_0(\Sigma^\nu,{\Bbb A})$, ${\bf p}\mapsto {\bf p}^\nu$, given by
\begin{align*}
{\bf p}^\nu(v)&=\nu(v){\bf p}(v),\quad v\in V;
\end{align*}
and an isomorphism ${\rm C}_1(\Sigma,{\Bbb A})\rightarrow{\rm C}_1(\Sigma^\nu,{\Bbb A})$, ${\bf c}\mapsto {\bf c}^\nu$, where
\begin{align*}
{\bf c}^\nu(\vec{e}^{\,\nu})&={\bf c}(\vec{e}\,),\quad \vec{e}\in\vec{E}.
\end{align*}

\begin{lem}
Switching preserves flows, tensions, boundary chains, and co-boundary chains. The preserving property follows from
\begin{equation}\label{eq:switching preversing}
\partial({\bf c}^\nu)=(\partial{\bf c})^\nu,
\quad \delta({\bf p}^\nu)=(\delta{\bf p})^\nu.
\end{equation}
\end{lem}
\begin{proof}
We first show the commutativity between the switching and the boundary (co-boundary) operator in \eqref{eq:switching preversing}. It is trivial to check the commutativity. However, a formal verification is needed. By definitions of ${\bf c}^\nu$ and $\partial$, and definitions of ${\bf p}^\nu$ and $\delta$, letting arcs $\vec{e}\,\in\Sigma$ and $\vec{e}^{\,\nu}\in\Sigma^\nu$, we have
\begin{align*}
\partial({\bf c}^\nu)(u) &= \sum_{e\in E}[u,\vec{e}{\,^\nu}]{\bf c}^\nu(\vec{e}{\,^\nu})\\
&=\sum_{e\in E}[u,\vec{e}\,]\nu(u){\bf c}(\vec{e}\,)\\
&=\nu(u)(\partial{\bf c})(u)\\
&=(\partial{\bf c})^\nu(u);
\end{align*}
\begin{align*}
\delta({\bf p}^\nu)(\vec{e}^{\,\nu}) &=\sum_{v\in V}[v,\vec{e}^{\,\nu}]{\bf p}^\nu(v)\\
&=\sum_{v\in V}[v,\vec{e}\,]\nu(v)\nu(v){\bf p}(v)\\
&=\delta{\bf p}(\vec{e}\,)\\
&=(\delta{\bf p})^\nu(\vec{e}^{\,\nu}).
\end{align*}
Let $\bf f$ be a flow of $\Sigma$. Then $\partial({\bf f}^\nu)=(\partial{\bf f})^\nu=0$. It means that ${\bf f}^\nu$ is a flow of $\Sigma^\nu$. We see that switching preserves flows.

Let ${\bf g}$ be a tension of $\Sigma$. For each integer-valued flow ${\bf c}$ of $\Sigma^\nu$, the chain ${\bf c}^\nu$ is an integer-valued flow of $\Sigma$, as we have just seen. Then
\[
\langle{\bf c},{\bf g}^\nu\rangle= \sum_{e\in E}{\bf c}(\vec{e}^{\,\nu}){\bf g}^\nu(\vec{e}^{\,\nu}) = \sum_{e\in E}{\bf c}^\nu(\vec{e}\,){\bf g}(\vec{e}\,)=0.
\]
This means that ${\bf g}^\nu$ is a tension of $\Sigma^\nu$. We see that switching preserves tensions.

It is trivial to check that switching preserves boundary chains and co-boundary chains.
\end{proof}

Integer-valued flows are closely related to directed closed walks. By a {\em walk} in $\Sigma$ we mean a sequence $W$ of vertices and edges, arranged alternatingly between vertices and edges, such that each vertex of $W$ is an endpoint of the edge ahead of it and the edge behind it. Reversing the order of $W$ becomes another walk, denoted $W^{-1}$. A walk $W$ is called a {\em bounded walk} if its initial member and the terminal member of $W$ are both vertices, a {\em half-open walk} if its initial and terminal members consist of one vertex and one outer-edge, and {\em open walk}~\footnote{The open walk here, requiring the initial and terminal edges to be outer-edges, is different from the open walk defined by Chen, Wang and Zaslavsky \cite{Chen-Wang-Zaslavsky-DM}.} if its initial and terminal members are both outer-edges. For instance, the following three are bounded walk, half-open walk, and open walk
\[
\begin{array}{ll}
W=v_0e_1v_1e_2\ldots v_{k-1}e_mv_m, &\\
W=v_0e_1v_1e_2\ldots e_mv_ke_{m+1}, & \mbox{where $e_{m+1}$ is an outer-edge,}\\
W=e_1v_1e_2v_2\ldots e_mv_ke_{m+1}, & \mbox{where $e_1$ and $e_{m+1}$ are outer-edges.}
\end{array}
\]
The {\em sign} of a walk $W$ is the product
\[
\sigma(W):=\prod_{i=1}^m\sigma(e_i).
\]
Whenever the edges $e_i$ of $W$ are chosen orientations such that
\begin{align*}
[v_i,\vec{e}_i]=-[v_i,\vec{e}_{i+1}]\quad \mbox{({\em coherent\,})}
\end{align*}
at each {\em internal} (neither initial nor terminal) vertex $v_i$, the walk with the coherent arcs is called a {\em directed walk} and the coherent orientation is called a {\em direction} of $W$, denoted $\omega(W)$. Reversing the sequence order of $W$ but keeping the arcs of $\omega(W)$, the obtained directed walk is denoted by $\omega(W^{-1})$.

There exist exactly two directions on each walk $W$, opposite each other. Each directed walk $\omega(W)$ can be viewed as a multiset of arcs, consisting of its arcs, and defines a 1-chain
\begin{align}
{\bf I}_{\omega(W)}&=\sum_{\vec{e}\in\omega(W)} \vec{e\,}.
\end{align}
It is convenient to simply write ${\bf I}_{\omega(W)}$ as $\omega(W)$.
If $\omega:=\omega(W)$ is a bounded directed walk, we define {\em incidence numbers} of $\omega$ at its {\em endpoints} (the initial and terminal vertices of $W$) as
\begin{equation}
[v_0,\omega]=(\partial\omega)(v_0),\quad [v_k,\omega]=(\partial\omega)(v_k).
\end{equation}
In the case that $W$ is a closed walk, it is straightforward to see that
\[
[v_0,\omega]=\left\{\begin{array}{ll}
0 & \mbox{if $\omega$ is a directed closed positive walk,}\\
\pm2 & \mbox{if $\omega$ is a directed closed negaitive walk.}\\
\end{array}\right.
\]

\begin{lem}
Let $\omega(W)$ be a directed walk. If $\,W$ is an open walk or a bounded closed positive walk, then ${\bf I}_{\omega(W)}$ is an integer-valued flow.
\end{lem}
\begin{proof}
Note that $\partial{\bf I}_{\omega(W)}$ vanishes at each vertex other than its endpoints $v_0$ and $v_k$. Notice that $[v_k,\vec{e}_k]=-\sigma(W)[v_0,\vec{e}_1]$, where $\vec{e}_1,\vec{e}_k\in\omega(W)$. Then
\begin{align*}
\partial{\bf I}_{\omega(W)}&=[v_0,\vec{e}_1]v_0+[v_k,\vec{e}_k]v_k\\ &=[1-\sigma(W)][v_0,\vec{e}_1]v_0\quad (\mbox{if $v_0=v_k$}).
\end{align*}
If $W$ is an open walk, then $W$ has no endpoints; it is clear that ${\bf I}_{\omega(W)}$ is a flow. If $W$ is a bounded closed walk, then ${\bf I}_{\omega(W)}$ is a flow if and only if $W$ is positive.
\end{proof}

\begin{lem}[{Chen~\cite[p.~2217]{Chen-GC}}]
Let ${\bf f}$ be an integer-valued flow of connected $\Sigma$. Then there exists either a directed closed positive walk $\omega(W)$, or a directed open walk $\omega(W)$ with initial outer-arc and terminal outer-arc, such that ${\bf f}={\bf I}_{\omega(W)}$.
\end{lem}

\begin{lem}\label{lem:negative-circle-flow}
Let $C$ be a negative circle with an orientation $\omega(C)$. Then every flow on $C$ with coefficients in $\Bbb A$ has the form $a\cdot{\bf 1}_{\omega(C)}$, where $a\in{\rm Tor}_2({\Bbb A})$ and ${\bf 1}_{\omega(C)}$ is the characteristic $1$-chain of $\omega(C)$, that is, ${\bf 1}_{\omega(C)}(\vec{e}\,)=1$ for $\vec{e}\in\omega(C)$ and ${\bf 1}_{\omega(C)}(\vec{e}\,)=0$ for $e\not\in C$.
\end{lem}
\begin{proof}
Let us arrange the vertices and edges of $C$ as a walk $W=v_0e_1v_1e_2\cdots e_kv_k$, where $v_0=v_k$, and choose a direction $\omega(W)$ of $W$, that is, $[v_i,\vec{e}_i]=-[v_i,\vec{e}_{i+1}]$ for $\vec{e}_i\in\omega(W)$, $1\leq i\leq k-1$. Then $[v_0,\vec{e}_1]=[v_k,\vec{e}_k]$, for $C$ is a negative circle. Given a flow ${\bf c}$ on $C$ valued in $\Bbb A$. By definition of flow,
${\bf c}(\vec{e}_i)={\bf c}(\vec{e}_{i+1})$, $1\leq i\leq k-1$, and
${\bf c}(\vec{e}_1)+{\bf c}(\vec{e}_k)=0$. Set $a={\bf c}(\vec{e}_i)$; then $2a=0$, that is, $a\in{\rm Tor}_2({\Bbb A})$. We obtain ${\bf f}=a\cdot{\bf 1}_{\omega(W)}$.

Now for the orientation $\omega(C)$, we have ${\bf f}(\vec{e\,})=a$ for $\vec{e}\in\omega(C)\cap\omega(W)$, and ${\bf f}(\vec{e}\,)=-a=a$ for $\vec{e}\in\omega(C)\cap\omega^{-1}(W)$. It follows that ${\bf f}=a\cdot{\bf 1}_{\omega(C)}$, where $a\in{\rm Tor}_2({\Bbb A})$. It is trivial to check that the 1-chain $a\cdot{\bf 1}_{\omega(C)}$, with $a\in{\rm Tor}_2({\Bbb A})$, is indeed a flow.
\end{proof}

\section{Frame matroid of signed graph}

Let ${\rm V}$ be a nonzero subspace of the vector space ${\rm C}_1(\Sigma,{\Bbb R})$. The {\em support} of a nonzero chain ${\bf c}$ of ${\rm C}_1(\Sigma,{\Bbb R})$ is the edge subset
\[
\supp{\bf c}=\big\{e\in E: {\bf c}(\vec{e}\,)\neq 0\big\}.
\]
A nonzero chain $\bf c$ of $\rm V$ is called {\em elementary} provided that the support of $\bf c$ is minimal in the sense that, if ${\bf c}'$ is a nonzero chain of ${\rm V}$ such that $\supp{\bf c}'\subseteq\supp{\bf c}$ then
$\supp{\bf c}'=\supp{\bf c}$. It is easy to see that any two elementary chains ${\bf c},{\bf c}'$ having the same support are proportional each other, that is, ${\bf c}'=\lambda{\bf c}$ for a nonzero real number $\lambda$. Let $\mathscr{D}({\rm V})$ denote the collection of supports of nonzero chains of $\rm V$, which is a poset under set inclusion. Elementary chains of $\rm V$ are those chains whose supports are minimal members of the poset $\mathscr{D}({\rm V})$.

If ${\rm V}$ is a rational subspace, that is, ${\rm V}$ has a basis of vectors whose coordinates are rational numbers, let ${\rm L(V)}$ denote the sub-lattice ${\rm V}\cap{\rm C}_1(\Sigma,{\Bbb Z})$. A nonzero vector ${\bf u}$ of ${\rm L(V)}$ is called {\em principal} provided that if ${\bf u}=c\,{\bf v}$ for a vector ${\bf v}\in{\rm L(V)}$ and an integer $c$ then $c=\pm1$. The following Theorem~\ref{Tutte-chain-matroid} is an important observation due to Tutte, who mentioned it twice without a formal proof, for it seems so simple and is supposed to be known (but in  fact not well known) in the literature.

\begin{thm}[Tutte~\cite{Tutte-Homotopy Theorem, Tutte-Matroid-Theory}]\label{Tutte-chain-matroid}
Let $\rm V$ be a nonzero vector subspace of ${\rm C}_1(\Sigma,{\Bbb R})$. Then $\mathscr{C}({\rm V})$, the class of minimal members of the poset
$\mathscr{D}({\rm V})$, forms a circuit system of a matroid $\mathcal{M}({\rm V})$ on the edge set $E$ of $\Sigma$, called the {\em frame matroid} of $\rm V$.
\end{thm}
\begin{proof}
Given two distinct members $C_1,C_2$ of $\mathscr{C}({\rm V})$ with a common edge $e$. Let ${\bf c}_i$ be elementary chains of ${\rm V}$ such that $C_i=\supp{\bf c}_i$, $i=1,2$. Let $\lambda={\bf c}_2(\vec{e}\,)/{\bf c}_1(\vec{e}\,)$. Then $\lambda\neq 0$ and ${\bf c}:={\bf c}_2-\lambda{\bf c}_1$ belongs to $\rm V$ and vanishes on $e$. Consequently, $\supp{\bf c}\subseteq(C_1\cup C_2)\setminus e$. Since $C_1$ and $C_2$ are distinct and minimal as edge subsets, there exist edges $e_1\in C_1\smallsetminus C_2$ and $e_2\in C_2\setminus C_1$. So ${\bf c}$ is nonzero; consequently, $\supp{\bf c}$ is nonempty and $\supp{\bf c}\in\mathscr{D}({\rm V})$. Thus there exists a minimal member $C$ in $\mathscr{D}({\rm V})$ such that $C\subseteq\supp{\bf c}$. We have obtained a member $C$ of $\mathscr{C}({\rm V})$ such that $C\subseteq (C_1\cup C_2)\setminus e$. The circuit axiom of matroid is satisfied.
\end{proof}

The vector space ${\rm C}_1(\Sigma,{\Bbb R})$ is equipped with an obvious inner product
\begin{equation}\label{Inner-product}
\langle{\bf c},{\bf d}\rangle:=\sum_{e\in E}{\bf c}(\vec{e}\,){\bf d}(\vec{e}\,).
\end{equation}
Let ${\rm V}^\perp$ denote the orthogonal complement of the vector space $\rm V$. Choose an orientation $\omega$ of $\Sigma$ and a vector basis ${\bf b}_1,\ldots,{\bf b}_m$ of $\rm V$ to form a base matrix $\bf B$, whose rows are the vectors ${\bf b}_i(\vec{e\,})$, $1\leq i\leq m=\dim{\rm V}$. Then the kernel of $\bf B$ is ${\rm V}^\perp$. The matroid $\mathcal{M}({\rm V})$ has the class $\mathscr{C}({\rm V})$ of circuits and the class $\mathscr{I}({\rm V})$ of independent sets (= columns of $\bf B$ corresponding to members of an independent edge set are linearly independent).

Reduce the matrix $\bf B$ by row operations to the form $[{\bf I}_m\; {\bf D}]$, where ${\bf I}_m$ is the identity matrix. Then the rows of $[-{\bf D}^T\;{\bf I}_n]$ form a basis of ${\rm V}^\perp$, where $n=\dim{\rm V}^\perp$. The matroid  $\mathcal{M}({\rm V}^\perp)$ has the class $\mathscr{C}({\rm V}^\perp)$ of circuits and the class $\mathscr{I}({\rm V}^\perp)$ of independent sets. Notice that the matroid $(E,\mathscr{I}({\rm V}^\perp))$ is dual to the matroid of $(E,\mathscr{I}({\rm V}))$, according to matroid duality. We obtain the following proposition as desired.

\begin{prop}\label{Prop:Duality-and-othogonality}
$\mathcal{M}^*({\rm V})=\mathcal{M}({\rm V}^\perp)$.
\end{prop}

The circuits of the matroid $\mathcal{M}({\rm V}^\perp)$ are the co-circuits of the matroid $\mathcal{M}({\rm V})$, and vice versa. We are only interested in two subspaces of ${\rm C}_1(\Sigma,{\Bbb R})$, the flow space ${\rm F}(\Sigma,{\Bbb R})$ and the tension space ${\rm T}(\Sigma,{\Bbb R})$, which is the orthogonal complement of ${\rm F}(\Sigma,{\Bbb R})$. The circuit system of the frame matroid $\mathcal{M}({\rm F})$, where ${\rm F}={\rm F}(\Sigma,{\Bbb R})$, was characterized partly by Bouchet \cite[Corolarry 2.3]{Bouchet} for signed graph without outer-edges, and fully by Chen \cite{Chen-GC} for signed graph with outer-edges, and was confirmed expectantly to be the same circuit system introduced by Zaslavsky \cite{Zaslavsky-signed-graphs-DAM,
Zaslavsky-orientation-EJC} for the signed-graph $\Sigma$. We are motivated to redefine circuits for signed graphs as follows.

\begin{defn}[Signed-graph circuit and its direction]
A {\em signed-graph circuit} or just {\em circuit} is a connected signed graph $C$ whose edge set is nonempty and minimal, permitting a nowhere-zero flow ${\bf c}$. The orientation $\omega(C)=\{\vec{e}:{\bf c}(\vec{e}\,)>0\}$ is called a {\em direction} of $C$. The {\em frame matroid} $\mathcal{M}(\Sigma)$ of a signed graph $\Sigma$ is defined by the circuits of $\Sigma$.
\end{defn}

The minimality of the edge set of a circuit $C$ implies that nowhere-zero integer flows on $C$ are unique up to a constant multiple. Then there exist exactly two opposite directions on $C$. Equivalently, circuit $C$ is a signed graph whose edge set is nonempty and minimal to permit an orientation $\omega_C$ such that the signed digraph $(C,\omega_C)$ contains neither sink nor source; this is exactly the old definition of signed-graph circuit, adopted previously by Chen~\cite[Definition 2.1]{Chen-GC}. The following characterization of signed-graph circuits was first described by Zaslavsky in words (see \cite[p.~54]{Zaslavsky-signed-graphs-DAM}) and in figures (see \cite[p.~365]{Zaslavsky-orientation-EJC}). The circuit types (C2), (C3), (C5) of Theorem \ref{thm:circuits} was characterized by Bouchet \cite[Corollary 2.3]{Bouchet} for signed graphs without outer-edges by supports of elementary chains. Like Tutte \cite{Tutte-class-of-abelian-group}, Bouchet chose an orientation to define chains, not intrinsically like the standard treatment in algebraic topology. For signed graphs with possible outer-edges, supports of elementary flows are characterized into five types (C1)--(C5) in Theorem \ref{thm:circuits} by Chen~\cite[Theorem~2.2]{Chen-GC}.

\begin{thm}[Characterization of signed-graph circuits, Zaslavsky~\cite{Zaslavsky-signed-graphs-DAM}]\label{thm:circuits}
The circuits of the frame matroid $\mathcal{M}(\Sigma)$ are classified into the following five topological patterns:
\begin{enumerate}[\hspace{5mm}\rm (C1)]
\item
Open line $L$.

\item
Positive circle $C$.

\item
Union of two edges disjoint negative circles $C_1$ and $C_2$, meeting at a unique common vertex, written $C=C_1C_2$ {\rm (two contacted negative circles)}.

\item
Union of a negative circle $C_1$ and a half-line $H$
(called {\em double half-line}), meeting uniquely at the endpoint of $H$, written $C=C_1H$ {\rm (negative circle with half-line)}.

\item
Union of two disjoint negative circles $C_1,C_2$ and
a line segment $P$ (called {\em double line segment}), which meets $C_1$ uniquely at its one endpoint and meets $C_2$ uniquely at the other endpoint, written $C_1PC_2$ {\rm (two negative circles connected by line segment)}.
\end{enumerate}
\begin{figure}[h]
\centering
{\includegraphics[height=12mm]{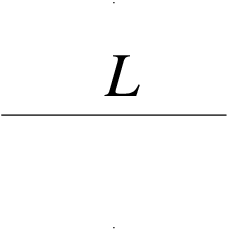}}
\hspace{5mm} {\includegraphics[height=12mm]{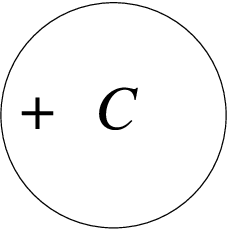}}
\hspace{5mm} {\includegraphics[height=12mm]{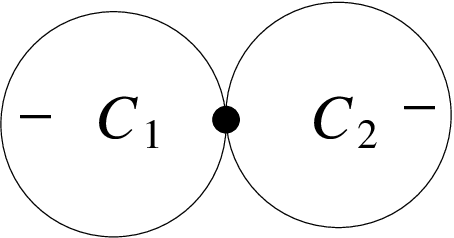}}
\hspace{5mm} {\includegraphics[height=12mm]{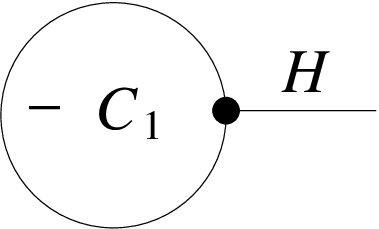}}
\hspace{5mm} {\includegraphics[height=12mm]{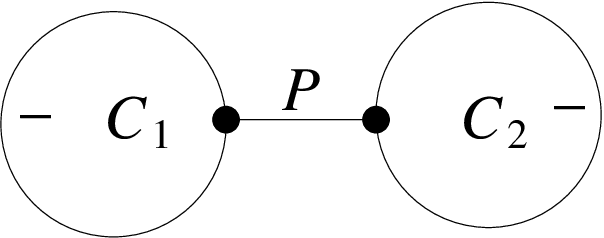}}
\caption{Five topological patterns of signed-graph circuits} \label{Fig:Circuit-patterns}
\end{figure}
\end{thm}

Each circuit $C$ of types (C1)--(C5) can be viewed as an edge multiset $(C,m)$ with multiplicity function $m$, having value 2 at each double edge and value 1 at each single edge of $C$. Given a direction $\omega(C)$ of a circuit $C$; the 1-chain
\begin{align*}
{\bf I}_{\omega(C)}:&= \sum_{\vec{e}\,\in\omega(C)}m(e)\vec{e}
\end{align*}
is an integer-valued principal flow of $\Sigma$, called a {\em circular flow} of $\Sigma$. More specifically, for circuits $C$ of patterns (C1)--(C3), we have ${\bf I}_{\omega(C)}={\bf 1}_{\omega(C)}$, the characteristic chain of $\omega(C)$. For circuits $C$ of patterns (C1)--(C5), we have
\begin{equation}\label{eq:circuit flow}
{\bf I}_{\omega(C)}(\vec{e}\,)=\left\{\begin{array}{ll}
1 & \mbox{if\, $\vec{e}\in\omega(C)$, $e$ is a single edge,}\\
2 & \mbox{if\, $\vec{e}\in\omega(C)$, $e$ is a double edge,}\\
0 & \mbox{otherwise.}
\end{array}\right.
\end{equation}
We shall see that the flow lattice ${\rm F}(\Sigma,{\Bbb Z})$ is generated by circuit flows of $\Sigma$.

The {\em rank} and the {\em dual rank} of the frame matroid $\mathcal{M}(\Sigma)$ or $\Sigma$ are defined, respectively, as
\begin{equation}
r(\Sigma):=|V|-b(\Sigma), \quad r^*(\Sigma):=|E|-r(\Sigma).
\end{equation}
The {\em rank function} $r$ and the {\em cycle-rank function} ${\rm cr}$ of $\mathcal{M}(\Sigma)$ are the functions of the power set $\mathcal{P}(E)$ to ${\Bbb Z}_{\geq 0}$, defined for each $S\in\mathcal{P}(E)$ by
\begin{equation}\label{eq:rank and cycle-rank}
r(S):=r(\Sigma|S),\quad {\rm cr}(S):=r^*(\Sigma|S),
\end{equation}
where $\Sigma|S$ is the spanning signed subgraph of $\Sigma$ induced by the edge subset $S$, whose sign function $\sigma|_S$ is the restriction of $\sigma$ to $S$. More explicitly, each $S\subseteq E$ edge subset has its rank and cycle-rank given by
\begin{align}
r(S)&=|V|-b(S),\\
{\rm cr}(S)&=|S|-r(S)=r^*(E)-r^*(E\smallsetminus S),\label{eq:cycle-rank-formula}
\end{align}
where $r^*$ is the rank function of the matroid $\mathcal{M}^*(\Sigma)$, the dual matroid of $\mathcal{M}(\Sigma)$, defined for each $S\in\mathcal{P}(E)$ as the rank of the sub-matroid $\mathcal{M}^*(\Sigma)|S$, the restriction of $\mathcal{M}^*(\Sigma)$ to $S$.

An edge subset $S$ of $\Sigma$ is called {\em dependent} if $\Sigma|S$ contains circuits; otherwise, $S$ is called {\em independent}. A maximal independent set is called a {\em base} or {\em spanning forest} of $\Sigma$. The rank $r(S)$ equals the number of edges of a maximal independent set contained in $S$, and the cycle-rank ${\rm cr}(S)$ equals the number of edges of a maximal independent set of the matroid $\mathcal{M}^*(\Sigma|S)$ dual to $\mathcal{M}(\Sigma|S)$.\vspace{1ex}

\noindent
{\bf Warning.}  There is a danger of ambiguity in terminology between ``loops" of a signed graph $\Sigma$ and the ``loops" of its frame matroid $\mathcal{M}(\Sigma)$. They are not exactly same in our usage. Loops of $\Sigma$ are single edges whose two endpoints are identical; while loops of $\mathcal{M}(\Sigma)$ are positive loops of $\Sigma$. Negative loops of the signed graph $\Sigma$ are {\em not} loops of the frame matroid $\mathcal{M}(\Sigma)$. The cycle-rank function is {\em not} the rank function of the matroid $\mathcal{M}^*(\Sigma)$ dual to $\mathcal{M}(\Sigma)$.

\section{Characterization of signed-graph cuts and bonds}

The notion of ``cut" for signed graphs was missing until its first characterization by Chen and Wang \cite[Proposition 2.1]{Chen-Wang-EJC} for signed graph without outer-edges. However, at that time, the authors didn't understand yet that circuits come from the minimal supports of nonzero flows, and that bonds come from the minimal supports of nonzero tensions. So naturally, cuts should, and actually do, come from certain tensions.

Let $X$ be a nonempty vertex subset of $\Sigma$. Denote by $X^c$ the complement $V\setminus X$, and by $O_X$ the set of outer-edges with endpoints in $X$. We adopt common graph notations
\begin{align}
\Sigma[X]:&=\mbox{signed subgraph with vertex set $X$, consiting of }\nonumber\\
&\hspace{6mm}\mbox{loops and link edges of $\Sigma$ having endpoins in $X$,}\\
[X,X^c]:&=\mbox{set of link edges between $X$ and $X^c$,}
\end{align}
where $\Sigma[X]$ is called the {\em compact sign subgraph} induced by $X$. We further define {\em non-compact signed subgraphs} (induced by $X$)
\begin{align}
\Sigma[X):&=\Sigma[X]\cup O_X,\\
\Sigma(X):&=\Sigma[X]\cup[X,X^c]\cup O_X.
\end{align}

Let ${\bf 1}_X$ denote the characteristic (or indicator) function of $X$, that is, ${\bf 1}_X(v)=1$ for $v\in X$ and ${\bf 1}_X(v)=0$ for $v\in X^c$. The function ${\bf 1}_X$ can be viewed as a $0$-chain in every signed graph $\Sigma^\nu$, where $\nu$ is any switching on $X$, that is, $\nu(v)=1$ for all $v\in X^c$. In order to specify that ${\bf 1}_X$ is a 0-chain of $\Sigma^\nu$, we may write ${\bf 1}_X$ in the form ${\bf 1}_X(\Sigma^\nu)$.
By Lemma~\ref{Lem:Potential-is-tension}, its co-boundary $\delta{\bf 1}_X(\Sigma^\nu)$ is a tension of $\Sigma^\nu$. Let $\delta X(\Sigma^\nu)$ denote the support of $\delta{\bf 1}_X$ in $\Sigma^\nu$, that is,
\[
\delta X(\Sigma^\nu)=[X,X^c]^\nu\cup E_X^{\,\nu}\cup O_X^{\,\nu},
\]
where $E_X^\nu$ is the set of negative edges of $\Sigma^\nu[X]$. Transforming the edge subset $\delta X(\Sigma^\nu)$ of $\Sigma^\nu$ back to an edge subset of $\Sigma$ by the same $\nu$, denoted $\delta X^\nu$, we obtain and define
\begin{align}\label{defn:eq-DXV}
\delta X^\nu:
&=[X,X^c]\cup E_X\cup O_X.
\end{align}
The notation $\delta X^\nu$ is justified by the fact that $\delta X$ is traditionally used for cut in graph theory, and that the switching $\nu$ is unique up to sign in the sense of Theorem \ref{thm:uni-cut}(a). Moreover, the edge subset $\delta X^\nu$ is the support of the tension $\delta(\nu\cdot{\bf 1}_X)$ in
${\rm T}(\Sigma,{\Bbb Z})$.

\begin{defn}[Signed-graph cut and its direction]\label{defn:cut}
A {\em signed-graph cut} or just {\em cut} of a signed graph $\Sigma$ is a nonempty edge subset $U$ of the form $\delta X^\nu$, where $\nu$ is a switching on a nonempty vertex subset $X$, such that $\Sigma[X]\smallsetminus U$ is balanced, edges of $\Sigma^\nu[X]\cap U^\nu$ are negative, edges of $\Sigma^\nu[X]\smallsetminus U^\nu$ are positive, and $(\Sigma[X]\smallsetminus U)\cup e$ is unbalanced for each edge $e$ of $\Sigma[X]\cap U$.

A cut $U$ is a {\em uni-cut} if $\Sigma[X]$ is connected. A {\em bond} of $\Sigma$ is a cut $U$ that is minimal in the sense that there is no cut of $\Sigma$ contained properly in $U$ as edge subset.

A {\em direction} of a cut $\,U=\delta X^\nu$ is an orientation $\omega(U)$ on $U$ such that all arcs of $\omega^\nu(U)$ point to $X$ or all point away from $X$.
\end{defn}

\begin{thm}\label{thm:uni-cut}
Let $\Sigma$ be connected and unbalanced. Let $X$ be a nonempty vertex subset of $\Sigma$ such that $\Sigma[X]$ is connected.
\begin{enumerate}[\hspace{5mm}\rm (a)]
\item
There exists a switching $\nu$ on $X$ such that $U:=\delta X^\nu$ is a uni-cut of $\Sigma$. Moreover, $\Sigma[X]\setminus U$ is connected.

\item
If $\mu$ is a switching on $X$ such that $\delta X^\nu=\delta X^\mu$, then $\mu=\pm\nu$ on $X$; that is, $\nu$ is unique up to a sign on $X$.
\end{enumerate}
\end{thm}
\begin{proof}
(a) Choose a spanning tree $T$ of $\Sigma[X]$. It is trivial that there exists a switching on $X$ such that all edges of $T^\nu$ are positive. Let $E_X^\nu$ denote the set of negative edges of $\Sigma^\nu[X]$. It is easy to see that $U:=\delta X^\nu=[X,X^c]\cup E_X\cup O_X$ is a uni-cut of $\Sigma$.

Suppose $\Sigma[X]\smallsetminus U$ contains two or more components $\Sigma_i$. Since $\Sigma[X]$ is connected, there are two components $\Sigma_1,\Sigma_2$ connected by an edge $e_0$ of $\Sigma[X]\cap U$. Since $\Sigma[X]\smallsetminus U$ is balanced, so does each of its components. Clearly, the union $\Sigma_1\cup\Sigma_2\cup e_0$ is balanced. It follows that $(\Sigma[X]\smallsetminus U)\cup e_0$ is
a union of its balanced component $\Sigma_1\cup\Sigma_2\cup e_0$ and other balanced components; consequently, itself is balanced. However, by definition of cut, $(\Sigma[X]\smallsetminus U)\cup e$ is unbalanced for each member $e$ of $U$, contradictory to that $(\Sigma[X]\smallsetminus U)\cup e_0$ is balanced.

(b) By the choice of $\nu$, it is clear that $U=\supp \delta(\nu\cdot{\bm 1}_X)$. Let $\mu$ be any switching on $X$ such that $\delta X^\nu=\delta X^\mu$. Suppose $\nu\neq\pm\mu$. There exist two distinct vertices $u_1,u_2$ of $X$ such that $\nu(u_1)=\mu(u_1)$ and $\nu(u_2)=-\mu(u_2)$. Set
\[
X_1:=\{v\in X\:|\:\nu(v)=\mu(v)\}, \quad X_2:=\{v\in X\:|\:\nu(v)=-\mu(v)\}.
\]
Then both $X_1$ and $X_2$ are nonempty.

Since $\Sigma[X]$ is connected, there exists an edge $e_0$ connecting $\Sigma[X_1]$ and $\Sigma[X_2]$, with endpoints $v_1\in X_1$ and $v_2\in X_2$. There are four possible cases:
\begin{enumerate}[\hspace{5mm} (1)]
\item
$\nu(v_1)=\nu(v_2)$, $\mu(v_1)=\mu(v_2)$,

\item
$\nu(v_1)=\nu(v_2)$, $\mu(v_1)=-\mu(v_2)$,

\item
$\nu(v_1)=-\nu(v_2)$, $\mu(v_1)=\mu(v_2)$,

\item
$\nu(v_1)=-\nu(v_2)$, $\mu(v_1)=-\mu(v_2)$.
\end{enumerate}
Case (1) implies the contradiction $\nu(v_1)=\mu(v_1)=\mu(v_2)=-\nu(v_2)=-\nu(v_1)$. Case (4) implies the contradiction $\nu(v_1)=\mu(v_1)=-\mu(v_2)=\nu(v_2)=-\nu(v_1)$.

Note that switching does not change balance. For each edge $e$ of $\Sigma[X]$ with $u,v$ endpoints, recall that $e\in U$ if and only if $\sigma^\nu(e)=\nu(u)\sigma(e)\nu(v)=-1$. Consider the case (2). If $e_0$ is positive in $\Sigma$, then $e_0\not\in U$ by $\nu(v_1)=\nu(v_2)$, and $e_0\in U$ by $\mu(v_1)=-\mu(v_2)$, which is a contradiction. If $e_0$ is negative in $\Sigma$, then $e_0\in U$ by $\nu(v_1)=\nu(v_2)$, and $e_0\not\in U$ by $\mu(v_1)=-\mu(v_2)$, which is also a contradiction. Finally, consider the case (3). If $e_0$ is positive in $\Sigma$, then $e_0\in U$ by $\nu(v_1)=-\nu(v_2)$, and $e_0\not\in U$ by $\mu(v_1)=\mu(v_2)$, which is a contradiction. If $e_0$ is negative in $\Sigma$, then $e_0\not\in U$ by $\nu(v_1)=-\nu(v_2)$, and $e_0\in U$ by
$\mu(v_1)=\mu(v_2)$, which is also a contradiction.
\end{proof}

The following corollary concerning signed graph cuts follows straightforwardly from Theorem \ref{thm:uni-cut}.

\begin{cor}\label{cor:cut}
Let $\Sigma$ be connected and unbalanced. Let $X$ be a nonempty vertex subset and $\Sigma[X]$ be decomposed into connected components $\Sigma[X_i]$. Then there exists a switching $\nu$ on $X$ such that $U:=\delta X^\nu$ is a cut of $\Sigma$. Moreover, each $\Sigma[X_i]\smallsetminus U$ is connected and $\nu$ is unique up to sign on each component of $\Sigma[X]$.
\end{cor}

\noindent
{\bf Remark 2.}
The concept of cut is not well-defined for signed graph in the literature. Chen-Wang \cite{Chen-Wang-EJC} gave a correct definition of cut for signed graph with no outer-edges. Zaslavsky \cite[p.\,54]{Zaslavsky-signed-graphs-DAM} proposed a so-called ``improving set" as a cut, meaning a nonempty edge subset whose removal increases the number of balanced components, and claiming that it is analogous to the cut of ordinary graph. However, even for an ordinary graph, an improving set is not necessarily a cut. For instance, for $K_3$ the complete graph of three vertices, its edge set is an improving set, but is not a cut in the sense of graph theory.\vspace{1ex}

Unlike circuit, the notion of cut is not intrinsic as it cannot exist without an ambient signed graph. The current notation of cut, $U=\delta X^\nu$, is related to the old notation
\begin{equation}
U=[X,X^c]\cup E_X\cup O_X,
\end{equation}
where $E_X$ is required to be a minimal edge subset to be removed from $\Sigma[X]$ so that $\Sigma[X]\smallsetminus E_X$ is balanced. This required minimality is not necessarily satisfied in \eqref{defn:eq-DXV}. This old notation was adopted by Chen and Wang \cite{Chen-Wang-EJC} and by Zaslavsky \cite{Chen-Wang-Zaslavsky-DM}. In particular, $[X,X^c]\cup E_X$ is commonly used in the literature when $\Sigma$ contains no outer-edges.

\begin{thm}[Characterization of signed-graph bond]\label{thm:bond}
Let $\Sigma$ be connected and unbalanced. Let $U=\delta X^\nu=[X,X^c]\cup E_X\cup O_X$ be a cut of $\Sigma$. Then
\begin{enumerate}[\hspace{5mm}\rm (a)]
\item
The removal of $U$ increases the number of balanced components.

\item
The cut $U$ is a bond if and only if $\Sigma[X]$ is connected and each component of $\Sigma[X^c)$ is unbalanced.

\item
An edge subset $S$ of $\Sigma$ is a bond of $\Sigma$ if and only if $S$ is a minimal edge subset whose removal increases exactly one balanced component, that is,
\begin{equation}\label{Thm:balance-component+1}
b(\Sigma\setminus S)=b(\Sigma)+1,\quad b(\Sigma)=b((\Sigma\setminus S)\cup e)\;\;\mbox{for}\;\; e\in S.
\end{equation}
\end{enumerate}
\end{thm}
\begin{figure}[h]
\centering {\includegraphics[width=80mm]{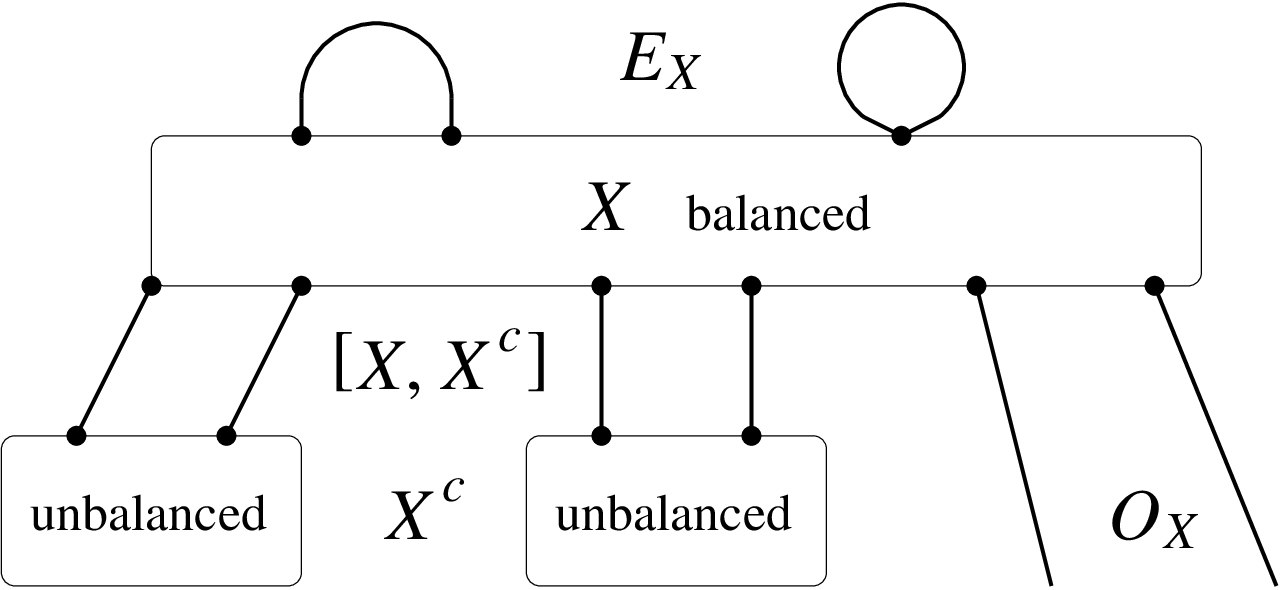}}
\caption{A typical bond $\delta X^\nu=[X,X^c]\cup E_X\cup O_X$.}
\label{Fig:Circuit-patterns}
\end{figure}
\begin{proof}
(a) When $U$ is removed from $\Sigma$, the signed subgraph $\Sigma[X]\setminus U$ is balanced, and is disjoint from $\Sigma[X^c)$. If $X$ is properly contained in $V$, the number of balanced components increased by the removal of $U$ is at least the number of components of $\Sigma[X]\setminus U$. If $X=V$, then $U=E_X\cup O_X$; the number of components of $\Sigma[X]\setminus U$ equals the number of balanced components increased by the removal of $U$.

(b) ``$\Rightarrow$" Let $U$ be a bond, that is, $U$ is a minimal cut. The minimality of $U$ implies that $U$ must be a uni-cut. If $X=V$, then $\Sigma[V)=\Sigma$ and $U=E_V\cup O_V$. Obviously, $\Sigma[V)$ is connected and unbalanced, $\Sigma[V^c)=\varnothing$; nothing is to be proved.

If $X$ is properly contained in $V$, then $\Sigma[X]$ is connected, for $U$ is a uni-cut. Suppose that $\Sigma[X^c)$ contains a balanced component $\Sigma[Y_j]$. Then $[Y_j,{Y_j}^c]$ is a uni-cut contained in $U$. The minimality of $U$ implies $U=[Y_j,{Y_j}^c]$. Thus $Y_j^c=X$ and $E_X=O_X=\varnothing$. We obtain $U=[X,X^c]$; all its edges are positive or all are negative; otherwise, both $[X,X^c]^+$ and $[X,X^c]^-$ are nonempty; then $[X,X^c]^-$ is a cut contained properly in $U$, contradictory to the minimality of $U$. We see that $\Sigma$ itself is balanced, contradictory to the unbalance of $\Sigma$.

``$\Leftarrow$" Let $U$ be a cut satisfying the given conditions. Since $U$ is a uni-cut, we see that $\Sigma[X]\setminus U$ is the only balanced component of $\Sigma\setminus U$ by Theorem \ref{thm:bond}(a). We have $b(\Sigma\setminus U)=1$. Since $(\Sigma\setminus U)\cup e$ is unbalanced for each member $e$ of $U$ by definition of cut, we have
\[
b((\Sigma\setminus U)\cup e)=b(\Sigma)=0.
\]
Suppose $U$ is not a bond, that is, $U$ is not minimal. There exists a cut $U'$ contained properly in $U$. Then $U'$, being a cut, implies that $b(\Sigma\setminus U')\geq 1$. Choose an edge $e\in U\setminus U'$. Clearly, $U'$ is contained in $U\setminus e$; then $b(\Sigma\setminus U')\leq b(\Sigma\setminus(U\setminus e))$. We arrive at the contradiction
\begin{align*}
1 \leq b(\Sigma\setminus U') & \leq b(\Sigma\setminus(U\setminus e))\\
&=b((\Sigma\setminus U)\cup e) =b(\Sigma)=0.
\end{align*}

(c) ``$\Rightarrow$''
Let $S$ be a bond of $\Sigma$, written as the cut form $S=[Y, Y^c]\cup E_Y\cup O_Y$, satisfying the conditions of part (b). Then $b(\Sigma\setminus S)=1$ and $b((\Sigma\setminus S)\cup e)=b(\Sigma)=0$. Clearly, the removal of $S$ results exactly one balanced component $\Sigma[Y]\setminus S$.

``$\Leftarrow$''
Let $b(\Sigma\setminus S)=1$ and $b((\Sigma\setminus S)\cup e)=b(\Sigma)=0$ for each $e\in S$. Consider the vertex set $Y$ of the unique balanced component of $\Sigma\setminus S$. Then $S=[Y,Y^c]\cup E_Y\cup O_Y$, where $E_Y$ is an edge subset of $\Sigma[Y]$, and $\Sigma[Y]\setminus E_Y$ is the only balanced component of $\Sigma\setminus S$. We see that $S$ is a cut of $\Sigma$ and satisfies the conditions of part (b). It follows from (b) that $S$ is a bond of $\Sigma$.
\end{proof}

\begin{cor}\label{cor:bond}
Let $U=\delta X^\nu$ be a cut of $\Sigma$, where $\Sigma$ is not necessarily connected. Then $U$ is a bond of $\Sigma$ if and only if $\Sigma[X)$ is connected and the component $\Sigma[Y)$ of $\Sigma$ that contains $\Sigma[X)$ is either unbalanced or balanced. Moreover, if $\Sigma[Y)$ is balanced, then $X\subsetneq Y$ and $U=[X,Y\smallsetminus X]$.
\end{cor}
\begin{proof}
Trivial by Theorem \ref{thm:bond}.
\end{proof}

A {\em cut chain} of $\Sigma$ is an integer-valued $1$-chain ${\bf I}_{\omega(U)}$, where $U=\delta X^\nu$ is a cut of $\Sigma$ and $\omega(U)$ is a direction of $U$, defined by
\begin{align}\label{eq:cut-chain}
{\bf I}_{\omega(U)}&=\delta({\bf 1}_X^\nu)=\delta(\nu\cdot{\bf 1}_X).
\end{align}
More explicitly,
\begin{align*}
{\bf I}_{\omega(U)}(\vec{e}\,)&=\left\{
\begin{array}{rl}
2 & \mbox{if\,$\vec{e}\in\omega(U)$, $e\in\Sigma[X]$,}\\
1 & \mbox{if\,$\vec{e}\in\omega(U)$, $e\not\in\Sigma[X]$,}\\
0 & \mbox{if\,$e\not\in U$.}
\end{array}\right.
\end{align*}
In the case that $U$ is an edge subset of $\Sigma[X]$, the chain ${\bf I}_{\omega(U)}=2\cdot{\bf 1}_{\omega(U)}$ is not a principal chain. We modify each cut chain to its {\em reduced form}
\begin{align}\label{eq:reduced-cut-chain}
\tilde{\bf I}_{\omega(U)}&=
\left\{\begin{array}{rl}
{\bf 1}_{\omega(U)} & \mbox{if\; $U\subset \Sigma[X]$,}\\
{\bf I}_{\omega(U)} & \mbox{if\; $U\not\subset \Sigma[X]$.}
\end{array}\right.
\end{align}

\begin{prop}[Characterization of tensions on uni-cut]\label{Lem:Tension-support-cut-lemma}
Let $U=\delta X^\nu$ be a uni-cut of $\Sigma$ with direction $\omega(U)$ such that all arcs of $\omega^\nu(U)$ point to $X$. Write $U=[X,X^c]\cup E_X\cup O_X$ and decompose $\Sigma[X^c)$ into connected components $\Sigma[Y_j)$. Let ${\bf g}$ be a tension of $\Sigma$ supported on $U$.
\begin{enumerate}[\hspace{5mm}\rm (a)]
\item
If $E_X\neq\varnothing$, then ${\bf g}=b$ is constant on the arc set $\omega(E_X)$.

\item
If $O_X\neq\varnothing$, then ${\bf g}=c$ is constant on the arc set $\omega(O_X)$. If in addition $E_X\neq\varnothing$, then $b=2c$.

\item
If $(\Sigma[X]\setminus E_X)\cup[X,Y_j]\cup\Sigma[Y_j)$ is balanced, then ${\bf g}=c_j$ is constant on the arc set $\omega([X,Y_j])$.

\item
If $\Sigma[Y_j)$ is unbalanced, then ${\bf g}=c'_j$ is constant on the arc set $\omega([X,Y_j])$. In addition, if $E_X\neq\varnothing$ then $b=2c'_j$. If $O_X\neq\varnothing$ and $O_{Y_j}=\varnothing$ then $2c=2c'_j$. If $O_{Y_i}=\varnothing$ or $O_{Y_j}=\varnothing$ for two unbalanced components of $\Sigma[X^c)$, then $2c'_i=2c'_j$. Moreover, if $O_X\neq\varnothing$ and $O_{Y_j}\neq\varnothing$ then $c=c'_j$. If $O_{Y_i}\neq\varnothing$ and $O_{Y_j}\neq\varnothing$ then $c'_i=c'_j$.

\item
If $\Sigma[Y_j)$ is balanced and $(\Sigma[X]\setminus E_X)\cup[X,Y_j]\cup\Sigma[Y_j)$ is unbalanced, then ${\bf g}=c_j^+$ is constant on $\omega^+([X,Y_j]):=\omega(([X,Y_j]^{\nu\nu_j+})^{\nu\nu_j})$, and ${\bf g}=c_j^-$ is constant on $\omega^-([X,Y_j]):=\omega(([X,Y_j]^{\nu\nu_j-})^{\nu\nu_j})$, where $\nu_j$ is a switching on $Y_j$ such that all edges of a spanning tree $T_j$ of $\Sigma[Y_j]$ are positive. In addition, if $E_X\neq\varnothing$ then $b=c_j^++c_j^-$.
\end{enumerate}
\end{prop}
\begin{proof}
Note that uni-cut is invariant under switching, that is, the edge subset $U$ of $\Sigma$ is a uni-cut if and only if $U^\nu$ is a uni-cut of $\Sigma^\nu$ for any switching $\nu$ on $X$. Since $\Sigma[X]$ is connected, so is $\Sigma[X]\smallsetminus E_X$ by definition of uni-cut. We may assume that all edges of $\Sigma[X]\smallsetminus U$ are positive and all edges of $\Sigma[X]\cap U$ are negative. We further assume that all arcs of $\omega(U)$ point to $X$.

(a) Given two edges $e_1,e_2$ of $E_{X}$ with arcs $\vec{e}_1,\vec{e}_2\in\omega(U)$ pointing to $X$, and the endpoints $u_1,v_1$ of $e_1$ and the endpoints $u_2,v_2$ of $e_2$. Choose a directed path $\omega(P_1)$ from $u_1$ to $u_2$ and a directed path $\omega(P_2)$ from $v_2$ to $v_1$ in $\Sigma[X]\smallsetminus U$. Then $\omega(W):=v_1\vec{e}_1\omega(P_1)\vec{e}_2^{\;-1}\omega(P_2)$ is a closed directed positive walk and ${\bf I}_{\omega(W)}$ is an integer-valued flow of $\Sigma$. See the left of Figure~\ref{Fig:b-cases}. We obtain
\[
{\bf g}(\vec e_2)-{\bf g}(\vec e_1)=\langle{\bf I}_{\omega(W)},{\bf g}\rangle=0,
 \]
that is, ${\bf g}(\vec e_1)={\bf g}(\vec e_2)$. So ${\bf g}=b$ is constant on the arc subset $\omega(E_X)$.

\begin{figure}[h]
\centering
\includegraphics[height=30mm]{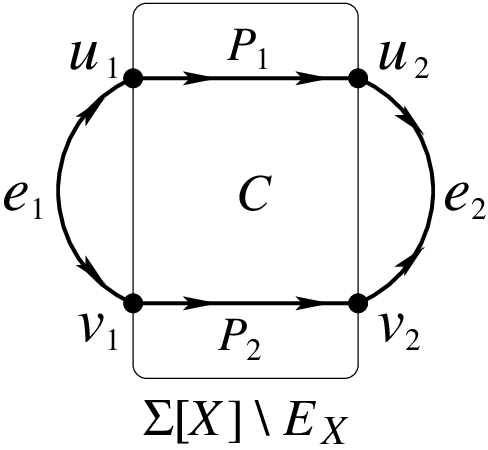}\hspace{10mm}
\includegraphics[height=30mm]{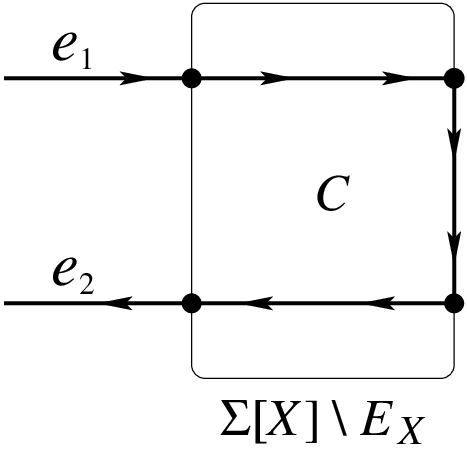}\hspace{10mm}
\includegraphics[height=30mm]{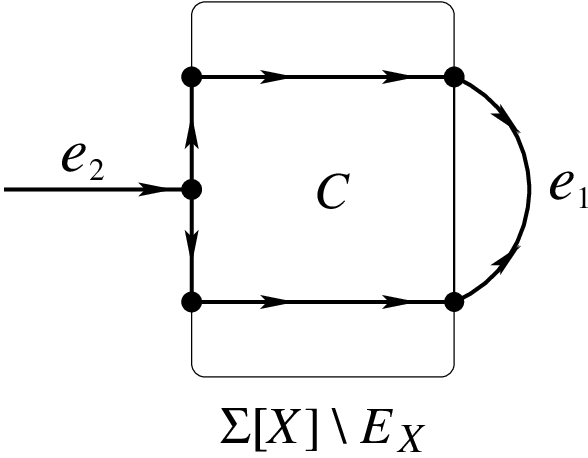}
\caption{Case (a) and (b)}\label{Fig:b-cases}
\end{figure}

(b) Given two edges $e_1,e_2$ of $O_X$ with arcs $\vec{e}_1,\vec{e}_2\in\omega(U)$. Let $C$ be a circuit of $\Sigma[X]\cup O_X$ in line pattern (C1) of Theorem~\ref{thm:circuits}, with a direction $\omega(C)$, such that $e_1,e_2\in C$. See the middle of Figure~\ref{Fig:b-cases}.
If $\vec{e}_1\in\omega(C)$, then $\vec{e}_2^{\;-1}\in\omega(C)$. We obtain
\[
{\bf g}(\vec{e}_1)-{\bf g}(\vec{e}_2)=\langle{\bf I}_{\omega(C)}, {\bf g}\rangle=0,
\]
that is, ${\bf g}(\vec{e}_1)={\bf g}(\vec{e}_2)$. So ${\bf g}=c$ is constant on the arc set $\omega(O_X)$. If in addition $E_X\neq\varnothing$, choose edges $e_1\in E_X,e_2\in O_X$ and a circuit $C$ of $\Sigma[X]\cup O_X$ in pattern (C4) of Theorem~\ref{thm:circuits}, with a direction $\omega(C)$ and $e_1,e_2\in C$. See the right of Figure~\ref{Fig:b-cases}. If $\vec{e}_1\in\omega(C)$, then $\vec{e}_2^{\;-1}\in\omega(C)$. We obtain
\[
{\bf g}(\vec{e}_1)-2{\bf g}(\vec{e}_2)=\langle{\bf I}_{\omega(C)}, {\bf g}\rangle=0,
\]
that is, ${\bf g}(\vec{e}_1)=2{\bf g}(\vec{e}_2)$. So $b=2c$.

(c) Note that $\Sigma[Y_j)$ contains no outer-edges, no negative circles, and $\Sigma[Y_j)=\Sigma[Y_j]$. There exists a switching $\nu_j$ on $Y_j$ such that all edges of $\Sigma^{\nu_j}[Y_j]$ are positive. We may simply assume that all edges of $\Sigma[Y_j]$ are already positive. Then edges of $[X,Y_j]$ are either all positive or all negative. Given two edges $e_1,e_2$ of $[X,Y_j]$, with the arcs $\vec{e}_1,\vec{e}_2\in\omega(U)$, and with $e_1,e_2$ both positive or both negative. Let $C$ be a circuit of $(\Sigma[X]\smallsetminus E_X)\cup[X,Y_j]\cup\Sigma[Y_j]$ in pattern (C2) of Theorem~\ref{thm:circuits}, with a direction $\omega(C)$, such that $e_1,e_2\in C$. See the left of Figure~\ref{Fig:c-case}, with edges of $[X,Y_j]$ all positive; and the right of Figure~\ref{Fig:c-case}, with edges of $[X,Y_j]$ all negative. If $\vec{e}_2\in\omega(C)$, then $\vec{e}_1^{\;-1}\in\omega(C)$. We obtain
\[
{\bf g}(\vec{e}_2)-{\bf g}(\vec{e}_1)
=\langle{\bf I}_{\omega(C)},{\bf g}\rangle=0,
\]
that is, ${\bf g}(\vec{e}_2)={\bf g}(\vec{e}_1)$. So ${\bf g}=c_j$ is constant on the arc set $\omega([X,Y_j])$ of $\omega(U)$.
\begin{figure}[h]
\centering
\includegraphics[height=30mm]{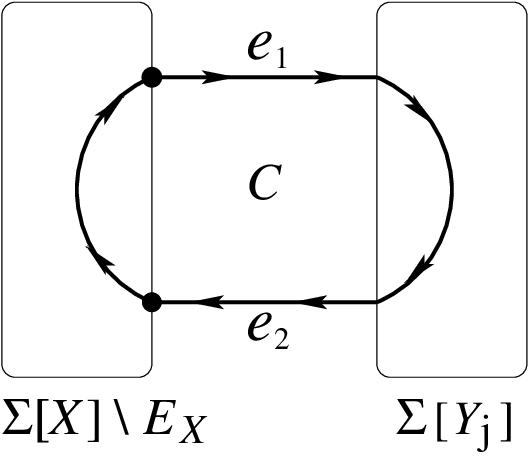}\hspace{20mm}
\includegraphics[height=30mm]{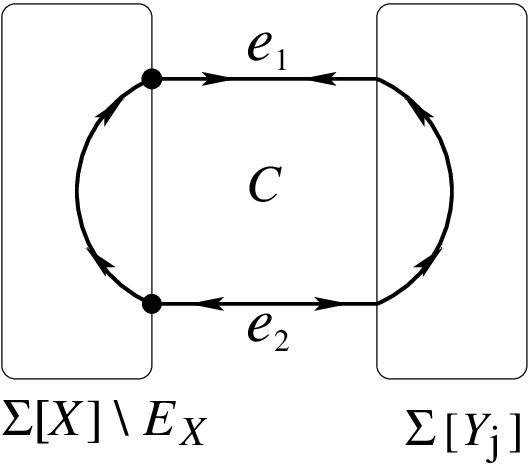}
\caption{Cases (c)}\label{Fig:c-case}
\end{figure}

(d) We first show that $\bf g$ is constant on $\omega([X,Y_j])$. Choose a spanning tree $T_j$ of $\Sigma[Y_j]$ and a switching $\nu_j$ such that all edges of ${T_j}^{\nu_j}$ are positive. We may assume that all edges of $T_j$ are already positive. Given arcs $\vec{e}_1,\vec{e}_2\in\omega([X,Y_j])$. Let $e_1,e_2\in[X,Y_j]^+$ or $e_1,e_2\in[X,Y_j]^-$. Choose a circuit $C$ of $(\Sigma[X]\setminus E_X)\cup[X,Y_j]\cup T_j$ in pattern (C2) of Theorem~\ref{thm:circuits}, with a direction $\omega(C)$ and $e_1,e_2\in C$. See the left for the former case and the middle for the latter case in Figure~\ref{Fig:c-case}.
If $\vec{e}_1\in\omega(C)$, then $\vec{e}_2^{\;-1}\in\omega(C)$. We obtain ${\bf g}(\vec{e}_1)-{\bf g}(\vec{e}_2)=\langle{\bf I}_{\omega(C)},{\bf g}\rangle=0$; that is, ${\bf g}(\vec{e}_1)={\bf g}(\vec{e}_2)$. We see that ${\bf g}$ is constant on $\omega([X,Y_j]^+)$, and constant on $\omega([X,Y_j]^-)$.

Let $\vec{e}_1\in\omega([X,Y_j]^+)$ and $\vec{e}_2\in\omega([X,Y_j]^-)$ pointing to $X$, written $e_1=u_1v_1$ and $e_2=u_2v_2$ with endpoints $u_1,u_2\in X,v_1,v_2\in Y_j$. Choose a directed path $\omega(P_0)$ whose initial arc points away from $u_1$ and whose terminal arc points to $u_2$ in $\Sigma[X]\smallsetminus E_X$. Choose directed paths $\omega(P_i)$ whose initial arcs point away from a vertex $v_0$ and whose terminal arcs point to $v_i$ in $T_j$, $i=1,2$. In the case that $\Sigma[Y_j)$ contains an outer-edge $e_0$, we may select $v_0$ as the endpoint of $e_0$, and choose $\vec{e}_0$ pointing to $v_0$. Then
\[
\omega(W):=\vec{e}_0\omega(P_1)\vec{e}_1 \omega(P_0)\vec{e\;}_2^{-1}\omega(P_2^{-1})\vec{e}_0
\]
is a directed closed positive walk; see the left of Figure~\ref{Fig:d1-cases}. In the case that $\Sigma[Y_j)$ contains no outer-edges, then $\Sigma[Y_j)=\Sigma[Y_j]$ and contains a negative circle $C_0$; we may assume that $C_0$ contains the vertex $v_0$. Arrange a directed closed negative path $\omega(P_0)$ on $C_0$, having its initial and terminal arcs pointing to $v_0$. Then
\[
\omega(W):=\omega(P_0)\omega(P_1)\vec{e}_1 \omega(P_0)\vec{e\;}_2^{-1}\omega(P_2^{-1})\omega(P_0^{-1})
\]
is a directed closed positive walk; see the right of Figure~\ref{Fig:d1-cases}. In both cases, the cut $U$ and the closed walk $W$ have the only common edges $e_1$ and $e_2$. We obtain ${\bf g}(\vec{e}_1)-{\bf g}(\vec{e}_2)=\langle{\bf I}_{\omega(W)},{\bf g}\rangle=0$, that is, ${\bf g}(\vec{e}_1)={\bf g}(\vec{e}_2)$. We see that ${\bf g}=c'_j$ is constant on $\omega([X,Y_j])$.
\begin{figure}[h]
\centering
\includegraphics[height=30mm]{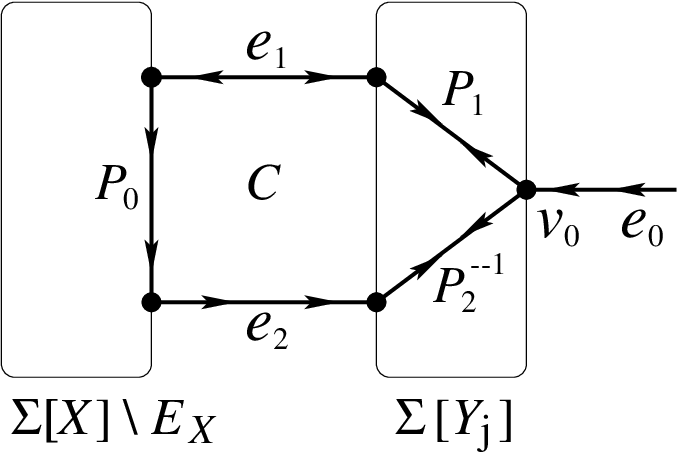}\hspace{10mm}
\includegraphics[height=30mm]{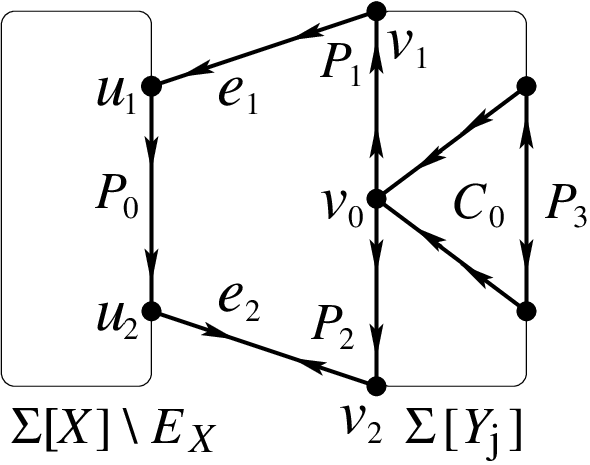}
\caption{Case (d1)}\label{Fig:d1-cases}
\end{figure}

Let $E_X\neq\varnothing$. Choose edges $e_0\in E_X,e_1\in[X,Y_j]$. If $O_{Y_j}\neq\varnothing$, choose an outer-edge $e_2\in O_{Y_j}$, there exists a circuit $C:=C_1H_1$ of pattern (C4) of Theorem~\ref{thm:circuits} with $e_0\in C_1$ and $e_1,e_2\in H_1$, contained in $\Sigma[X]\cup[X,Y_j]\cup\Sigma[Y_j)$. See the left of Figure~\ref{Fig:d2-cases}. If $O_{Y_j}=\varnothing$, then If $\Sigma[Y_j]=\Sigma[Y_j)$ contains negative circles. There exists a circuit $C:=C_1PC_2$ of pattern (C5) of Theorem~\ref{thm:circuits} with $e_0\in C_1$ and $e_1\in P$, contained in $\Sigma[X]\cup[X,Y_j]\cup\Sigma[Y_j]$. See the right of Figure~\ref{Fig:d2-cases}. In both cases, $e_1$ is a double edge of the circuit $C$.
Choose a direction $\omega(C)$ of the circuit $C$. Then $\vec{e}_0\in\omega(C)$ implies $\vec{e}_1^{\;-1}\in\omega(C)$. We obtain ${\bf g}(\vec{e}_0)-2{\bf g}(\vec{e}_1)=\langle{\bf I}_{\omega(C)},{\bf g}\rangle=0$,
that is, ${\bf g}(\vec e_0)=2{\bf g}(\vec e_1)$. So $b=2c'_j$.
\begin{figure}[h]
\centering
\includegraphics[height=30mm]{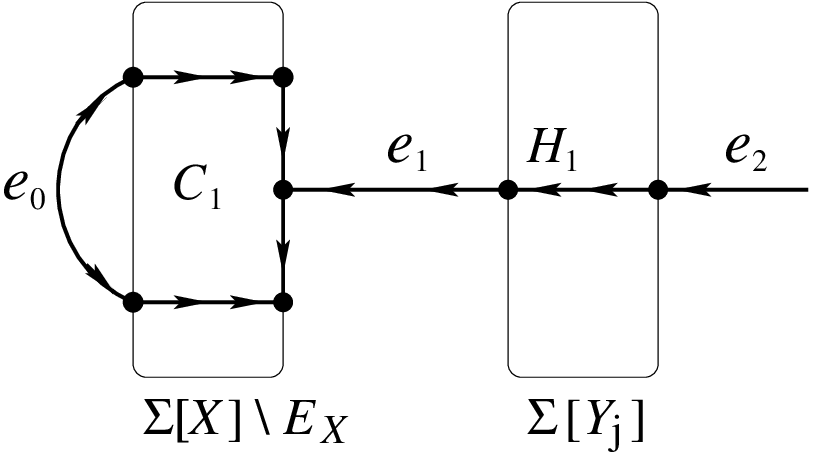}\hspace{10mm}
\includegraphics[height=30mm]{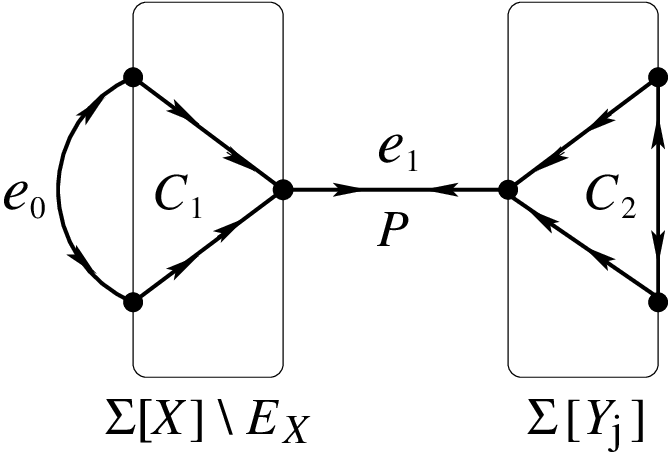}
\caption{Case (d2)}\label{Fig:d2-cases}
\end{figure}

Let $O_X\neq\varnothing$. Fix an edge $e_1\in O_X$. If $O_{Y_j}=\varnothing$, for each edge $e_2\in[X,Y_j]$ there exists a circuit $C:=C_1H_1$ of pattern (C5) of Theorem~\ref{thm:circuits}, contained in $(\Sigma[X)\smallsetminus E_X)\cup[X,Y_j]\cup\Sigma[Y_j]$, where $e_1,e_2\in H_1$. Choose a direction $\omega(C)$ of $C$. Then $\vec{e}_1\in\omega(C)$ implies $\vec{e}_2^{\;-1}\in\omega(C)$. See the left of Figure~\ref{Fig:Cases (d3) and e-cases}.
Thus $2{\bf g}(\vec{e}_1)-2{\bf g}(\vec{e}_2)=\langle{\bf I}_{\omega(C)},{\bf g}\rangle=0$. It follows that $2c=2c'_j$. If $O_{Y_j}\neq\varnothing$, choose an outer-edge $e_0\in O_{Y_j}$. For each edge $e_2\in[X,Y_j]$ there exists a circuit $C:=L$ of line pattern (C1) of Theorem~\ref{thm:circuits}, contained in $(\Sigma[X)\smallsetminus E_X)\cup[X,Y_j]\cup\Sigma[Y_j]$, where $e_0,e_1,e_2\in L$.  Choose a direction $\omega(C)$ of $C$. Then $\vec{e}_1\in\omega(C)$ implies $\vec{e}_2^{\;-1}\in\omega(C)$. Thus ${\bf g}(\vec{e}_1)-{\bf g}(\vec{e}_2)=\langle{\bf I}_{\omega(C)},{\bf g}\rangle=0$. It follows that $c=c'_j$. Likewise, if $O_{X_i}=\varnothing$ or $O_{X_j}=\varnothing$ then $2c'_i=2c'_j$. If $O_{X_i}\neq\varnothing$ or $O_{X_j}\neq\varnothing$ then $c'_i=c'_j$.

(e) Note that $\Sigma[Y_j)=\Sigma[Y_j]$ and contains no outer-edges and no negative circles. Since $(\Sigma[X]\setminus E_X)\cup\Sigma[Y_j]$ is balanced, all edges of $(\Sigma^\nu[X]\setminus E^\nu_X)\cup\Sigma^{\nu_j}[Y_j]$ are positive. The unbalance of $(\Sigma[X]\setminus E_X)\cup[X,Y_j]\cup\Sigma[Y_j]$ implies that $[X,Y_j]^{\nu\nu_j}$ is decomposed into disjoint a nonempty set $[X,Y_j]^{\nu\nu_j+}$ of positive edges and a nonempty set $[X,Y_j]^{\nu\nu_j-}$ of negative edges. Given edges $e_1,e_2\in[X,Y_j]$ such that either $e^{\nu\nu_j}_1,e^{\nu\nu_j}_2\in[X,Y_j]^{\nu\nu_j+}$ or $e^{\nu\nu_j}_1,e^{\nu\nu_j}_2\in[X,Y_j]^{\nu\nu_j-}$. Choose a circuit $C$ of pattern (C2) of Theorem~\ref{thm:circuits} with a direction $\omega(C)$ and $e_1,e_2\in C$, contained in $(\Sigma[X]\setminus E_X)\cup[X,Y_j]\cup\Sigma[Y_j]$, such that $C^{\nu\nu_j}$ is contained in
\[
\Sigma^\nu[X]\setminus E^\nu_X)\cup[X,Y_j]^{\nu\nu_j\epsilon}\cup\Sigma^{\nu_j}[Y_j], \quad\mbox{where $\epsilon=+$ or $-$.}
\]
See the left of Figure~\ref{Fig:c-case} for the former case, and the right of Figure~\ref{Fig:c-case} for the latter case. Moreover,  $\vec{e}_1^{\,\nu\nu_j}\in\omega^{\nu\nu_j}(C^{\,\nu\nu_j})$ implies $(\vec{e}_2^{\,\nu\nu_j})^{-1}\in\omega^{\nu\nu_j}(C^{\,\nu\nu_j})$. It is equivalent to saying that $\vec{e}_1\in\omega(C)$ implies $\vec{e}_2^{\;-1}\in\omega(C)$. We obtain ${\bf g}(\vec{e}_1)-{\bf g}(\vec{e}_2) =\langle{\bf I}_{\omega(C)},{\bf g}\rangle=0$. Thus ${\bf g}=c_j^+$ is constant on
\[
\omega^{\nu_j+}([X,Y_j]^\nu):=\omega(([X,Y_j]^{\nu\nu_j+})^{\nu\nu_j}),
\]
and ${\bf g}=c_j^-$ is constant on
\[
\omega^{\nu_j-}([X,Y_j]^\nu):=\omega(([X,Y_j]^{\nu\nu_j-})^{\nu\nu_j}).
\]
\begin{figure}[h]
\centering
\includegraphics[height=30mm]{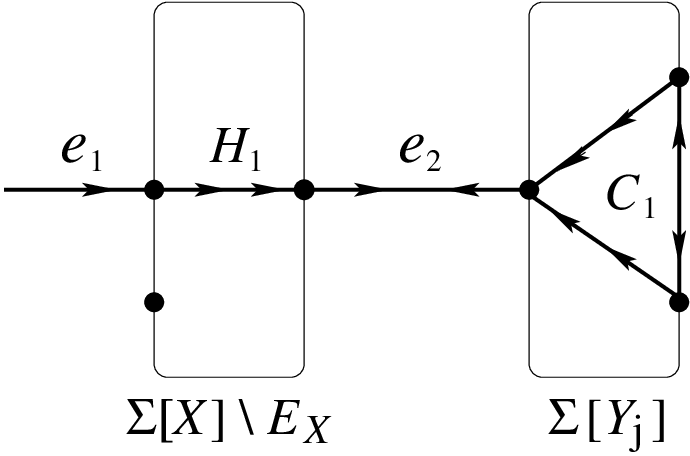}\hspace{10mm}
\includegraphics[height=30mm]{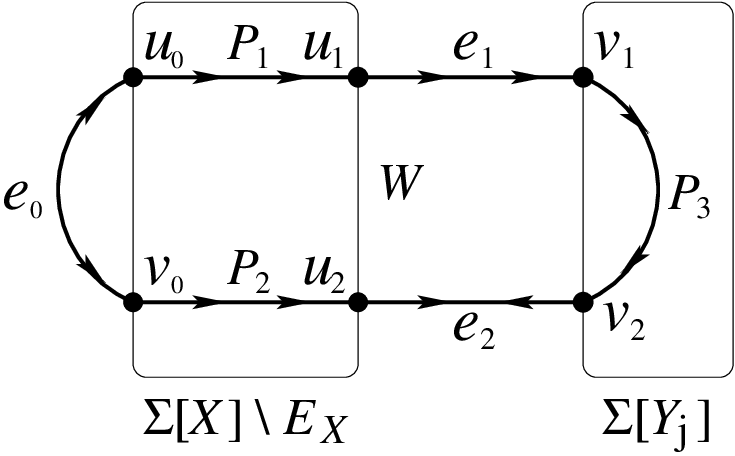}
\caption{Cases (d3) and (e)}\label{Fig:Cases (d3) and e-cases}
\end{figure}

If in addition, $E_X\neq\varnothing$, choose an edge $e_0\in E_X$ such that $\vec{e}_0^{\,\nu\nu_j}$ points to its endpoints $u_0,v_0$ in $X$; choose edges $e_1,e_2\in[X,Y_j]$ such that $\vec{e}_1^{\,\nu\nu_j},\vec{e}_2^{\,\nu\nu_j}$ points away from their endpoints $u_1,u_2$ respectively in $X$, $e_1^{\,\nu\nu_j}$ points to its point $v_1$ in $Y$ and $e_2^{\,\nu\nu_j}$ points away from its endpoint $v_2$ in $Y$. Let $P_1,P_2$ be directed paths in $\Sigma^{\,\nu\nu_j}[X]\smallsetminus E_X^{\nu\nu_j}$ from $u_0,v_0$ to $u_1,u_2$ respectively. Let $P_3$ be a directed path in the tree $T_j^{\,\nu\nu_j}$ from $v_1$ to $v_2$. Then $W:=P_1\vec{e}_1^{\,\nu\nu_j}P_3\vec{e}_2^{\,\nu\nu_j}P_2^{-1}\vec{e}_3^{\,\nu\nu_j}u_0$
is a directed closed positive walk in $\Sigma^{\nu\nu_j}$. See the right of Figure~\ref{Fig:Cases (d3) and e-cases}. It follows that $W^{\nu\nu_j}$ is a directed closed positive walk in $\Sigma$ and $\vec{e}_0,\vec{e}_1^{\;-1},\vec{e}_2^{\;-1}\in\omega(U)$.
We obtain
\[
{\bf g}(\vec{e}_0)-{\bf g}(\vec{e}_1)-{\bf g}(\vec{e}_2)=\langle{\bf I}_{\omega(W)},{\bf g}\rangle=0,
\]
that is, ${\bf g}(\vec{e}_0)={\bf g}(\vec{e}_1)+{\bf g}(\vec{e}_2)$. So $b=c_j^++c_j^-$.
\end{proof}

\begin{prop}\label{prop:bond-chain-is-elementary}
Given a tension $\bf g$ with coefficients in ${\Bbb A}$, supported on a bond 
\[
U=\delta X^\nu=[X,X^c]\cup E_X\cup O_X 
\] 
of $\Sigma$ with a direction $\omega(U)$. If $U=E_X$, then ${\bf g}=c{\bf 1}_{\omega(E_X)}$. Otherwise, let $\Sigma[Y_j]$ $(j\in J$) be components of $\Sigma[X^c]$ connected to $\Sigma[X]$. Then
\begin{equation}\label{eq:g-bond-formula}
{\bf g}=2c\cdot{\bf 1}_{\omega(E_X)}+
c\cdot{\bf 1}_{\omega(O_X)}+
\sum_{i\in I}c\cdot{\bf 1}_{\omega([X,Y_i])}
+\sum_{j\in J\smallsetminus I} c_j\cdot{\bf 1}_{\omega([X,Y_j])},
\end{equation}
where $I=\{i\in J: O_{Y_i}\neq\varnothing\}$, $c_j=c+a_j$ and $a_j\in{\rm Tor}_2({\Bbb A})$ for $j\in J\smallsetminus I$.
\end{prop}
\begin{proof}
According to Lemma~\eqref{Lem:Tension-support-cut-lemma}, we have that
${\bf g}=b$ is constant on $\omega(E_X)$, ${\bf g}=c$ is constant on $\omega(O_X)$, and ${\bf g}=c_j$ is constant on $\omega([X,Y_j]$, $j\in J$.
Moreover, $c=c_i=c_j$ for $i,j\in I$ and $b=2c=2c_j$ for $j\in J\smallsetminus I$. The condition $2c=2c_j$ for $j\in J\smallsetminus I$ implies that $c_j=c+a_j$ for each member $a_j\in{\rm Tor}_2({\Bbb A})$. The expression \eqref{eq:g-bond-formula} follows immediately.
\end{proof}

\begin{thm}[Characterization of elementary tensions]\label{thm:support-elementary}
Let $\Sigma$ be connected and unbalanced, and $U$ be an edge
subset of $\Sigma$. Then $U$ is a bond of $\Sigma$ if and only if $U$ is the
support of an elementary integer-valued tension of $\Sigma$.
\end{thm}
\begin{proof}
Let $U$ be a bond of $\Sigma$ with a direction $\omega(U)$. The cut chain ${\bf I}_{\omega(U)}$, defined by \eqref{eq:cut-chain}, is an integer-valued tension of $\Sigma$. For each $\bf h$ nonzero tension of ${\rm T}(\Sigma,{\Bbb Z})$ such that $\supp{\bf h}\subset U$, we have ${\bf h}=c\cdot{\bf I}_{\omega(U)}$ by Proposition~\ref{prop:bond-chain-is-elementary}. By definition of elementary chain, the cut chain ${\bf I}_{\omega(U)}$ is an elementary tension.

Conversely, let $U=\supp {\bf g}$, where $\bf g$ is an elementary integer-valued tension of $\Sigma$. We first claim that no edges of $U$ connect two (possibly identical) unbalanced components of $\Sigma\setminus U$. Suppose it is not true, that is, $\Sigma\setminus U$ contains two unbalanced components $\Sigma_1,\Sigma_2$ (possibly identical) connected by an edge $e_0\in U$ with endpoints $v_1\in\Sigma_1$ and $v_2\in\Sigma_2$. We claim that $\Sigma_1\cup\Sigma_2\cup e_0$ contains a circuit $C$ with $e_0\in C$.

In fact, if $\Sigma_i$ $(i=1,2)$ contains an outer-edge $e_i$ with endpoint $u_i$, choose a shortest path $P_i$ from $u_i$ to $v_i$ in $\Sigma_i$, then the circuit $C$ is the open line $e_1P_1e_0P_2^{-1}e_2$. If $\Sigma_i$ $(i=1,2)$ contains negative circles $C_i$, choose a shortest path $P_i$ from $C_i$ to the vertex $v_i$ in $\Sigma_i$, then $C$ is the circuit $C_1P_1e_0P_2^{-1}C_2$ in pattern (C5) of Theorem \ref{thm:circuits}. If $\Sigma_1$ contains a negative circle $C_1$ and $\Sigma_2$ contains an outer-edge $e_2$ with endpoint $u_2$, choose a shortest path $P_1$ from $C_1$ to $v_1$ and a shortest path $P_2$ from $v_2$ to $u_2$, then $C$ is the circuit $C_1P_1e_0{P_2}e_2$ in pattern (C4) of Theorem \ref{thm:circuits}. Now choose a direction $\omega(C)$ of $C$ and orientation $\vec{e}_0\in\omega(C)$. Then
\[
0=\langle{\bf I}_{\omega(C)},{\bf g}\rangle=\left\{
\begin{array}{rl}
{\bf g}(\vec{e}_0)  & \mbox{if $C$ is in line pattern (C1),}\\
2{\bf g}(\vec{e}_0) & \mbox{if $C$ is in pattern of (C4),(C5).}
\end{array}\right.
\]
Being zero for $\bf g$ at the edge $e_0\in U$ is contradictory to that $\bf g$ is nowhere-zero on $U$.

Let $\Sigma_1=\Sigma_2$ and the edge $e_0\in U$ be a loop or link with endpoints $v_1,v_2$ in $\Sigma_1$. Choose a spanning tree $T_1$ of $\Sigma_1$. If $\Sigma_1$ contains an outer-edge $e_1$ with endpoint $u_1$, then $B_1:=T_1\cup e_1$ is a base of $\Sigma_1$, and $B_1\cup e_0$ contains a unique circuit $C$ in pattern (C4) of Theorem \ref{thm:circuits}, with $e_0\in C$. If $\Sigma_1$ contains no outer-edges, then its unbalance forces that there exists an edge $e_1$ of $\Sigma_1$ such that $T_1\cup e_1$ contains a unique negative circle $C_1$. Then $B_1:=T_1\cup e_1$ is base of $\Sigma_1$, and $B_1\cup e_0$ contains a unique circuit $C$ with $e_0\in C$, and in pattern (C2), (C3) and (C5) of Theorem \ref{thm:circuits}. Note that $e_0$ is a single edge of the circuit $C$. Choose a direction $\omega(C)$ of $C$ with $\vec{e}_0\in\omega(C)$. We obtain ${\bf g}(\vec{e}\,)=\langle{\bf I}_{\omega(C)},{\bf g}\rangle=0$, contradictory to that $\bf g$ is nowhere-zero on $U$.

We have shown that there are no edges of $U$, connecting one unbalanced component of $\Sigma\setminus U$ to itself or to another unbalanced component of $\Sigma\setminus U$. Next we claim that $\Sigma\setminus U$ contains exactly one balanced component.

Suppose $\Sigma\setminus U$ contains no balanced components. Since $\Sigma$ is connected, the unbalanced components of $\Sigma\setminus U$ must be connected by edges of $U$. It follows that $\Sigma\setminus U$ is connected and unbalanced. Choose a spanning tree $T$ of $\Sigma\setminus U$. Choose an outer-edge $e_1$ of $\Sigma\setminus U$ if outer-edges exist; otherwise, choose an edge $e_1$ of $\Sigma\setminus U$ such that $T\cup e_1$ contains a negative circle $C_1$. Then $B_1:=T\cup e_1$ is a base of $\Sigma\setminus U$. Thus  for each edge $e\in U$, $B_1\cup e$ contains a unique circuit $C$. Choose a direction $\omega(C)$ of $C$ and consider the arc $\vec{e}\in\omega(C)$. We obtain ${\bf g}(\vec{e}\,)=\langle{\bf I}_{\omega(C)},{\bf g}\rangle=0$, contradictory to that ${\bf g}$ is nowhere-zero on $U$.

Suppose $\Sigma\smallsetminus U$ contains two or more balanced components. Given $\Sigma_1,\Sigma_2$ two distinct balanced components of $\Sigma\smallsetminus U$. Set $X_i=V(\Sigma_i)$; let $E_{X_i}$ denote the smallest subset of $\Sigma[X_i]\cap U$ such that $\Sigma[X_i]\smallsetminus E_{X_i}$ is balanced; and define
\[
U_i=[X_i,{X_i}^c]\cup E_{X_i}\cup O_{X_i},\quad i=1,2.
\]
Then $U_i$ is a cut of $\Sigma$, and $U_i=\supp{\bf I}_{\omega(U_i)}$, where $\omega(U_i)$ is a direction of the cut $U_i$, and ${\bf I}_{\omega(U_i)}$ is the cut chain defined by \eqref{eq:cut-chain}. Note that $U_i\subset U$ by definition of $U_i$. It forces that $U_i=U$, for $U$ is a minimal support of nonzero tension, $i=1,2$. Since $U_1=U_2$, it forces that $U=[X_1,X_2]$ and $\Sigma=\Sigma_1\cup\Sigma_2\cup U$. We may assume that all edges of $\Sigma_1,\Sigma_2$ are positive.
Since $\Sigma$ is unbalanced, it forces that $[X_1,X_2]$ contains both positive edges and negative edges. Then $U^+:=[X_1,X_2]^+$ and $U^-:=[X_1,X_2]^-$ are bonds of $\Sigma$, and both are contained in $U$. Thus $U^+=\supp{\rm I}_{\omega(U^+)}$ and $U^-=\supp{\rm I}_{\omega(U^-)}$, where $\omega(U^+)$ and $\omega(U^-)$ are directions of bonds $U^+$ and $U^-$ respectively. Since $U^+\subset U$ and $U^-\subset U$, by minimality of $U$ it forces that $U=U^+=U^-$, which is a contradiction.

Now we have shown that $\Sigma\setminus U$ contains exactly one balanced component $\Sigma_0$ with vertex set $X$, and all other components are unbalanced. Denote by $E_X$ the smallest edge subset of $\Sigma[X]\cap U$ such that $\Sigma[X]\setminus E_X$ is balanced. Then $U_0:=[X,X^c]\cup E_X\cup O_X$ is a bond of $\Sigma$, and $U_0\subseteq U$. Since $U=\supp{\bf g}$ with $\bf g$ elementary tension and $U_0$ is the support of the bond chain ${\bf I}_{\omega(U_0)}$, the minimality of $U$ implies $U=U_0$, which is a bond of $\Sigma$.
\end{proof}

\section{Flow group, boundary group, and homology group}

\subsection{Canonical bases of signed graphs}

Let $\Sigma$ be decomposed into connected components $\Sigma_i=(V_i,E_i,\sigma_i)$. For each $\Sigma_i$, choose $T_i$ a spanning tree of the connected compact signed subgraph
\[
\Sigma[V_i]=\Sigma_i\smallsetminus\{\mbox{outer-edges}\}.
\]
If $\Sigma_i$ contains outer-edges, choose an outer-edge $e_i$. If $\Sigma_i$ is unbalanced and contains no outer-edges, choose an edge $e_i$ such that the unique circle contained in $T_i\cup e_i$ is negative. We construct a base $B_i$ for each component $\Sigma_i$ as follows
\begin{align}\label{canonical-base}
B_i:&=\left\{\begin{array}{ll}
T_i   & \mbox{if $\Sigma_i$ is balanced,}\\
T_i\cup e_i & \mbox{if $\Sigma_i$ is unbalanced.}
\end{array}\right.
\end{align}
The base $B:\,=\bigcup_iB_i$ of $\Sigma$ is called a {\em canonical base} of $\Sigma$.

Given $B$ a canonical base of $\Sigma$. The components of $B$ are in one-to-one correspondence with the components of $\Sigma$. More precisely, balanced components correspond to balanced components; components with outer-edges correspond to components with outer-edges; and unbalanced components without outer-edges correspond to unbalanced components without outer-edges.

Let $B^c:=E\setminus B$ denote the complement of $B$. For each edge $e\in B^c$, there exists a unique signed graph circuit $C(B,e)$ contained in $B\cup e$, and $e$ is an edge of $C(B,e)$ and $C(B,e)\cap B^c=\{e\}$. Choose a direction $\omega(C(B,e))$ of the circuit $C(B,e)$ and an orientation $\vec{e}\in\omega(B(C,e))$. We obtain an integer-valued circuit flow
\[
{\bf I}_{\omega(C(B,e))},
\]
defined by \eqref{eq:circuit flow}. For any two edges $e,e'\in B^c$, we have
\begin{equation}\label{eq:circuit-flow-Kronecker delta}
{\bf I}_{\omega(C(B,e))}(\vec{e}{\,'})=\delta_{e,e'}
=\left\{\begin{array}{ll}
1 & \mbox{if $e=e'$,}\\
0 & \mbox{if $e\neq e'$.}
\end{array}\right.
\end{equation}
For each $C_{neg}$ negative circle contained in the canonical base $B$, choose an orientation $\omega(C_{neg})$ to generate an indicator $1$-chain ${\bf 1}_{\omega(C_{neg})}$. According to Lemma \ref{lem:negative-circle-flow}, the 1-chain
\[
a\cdot {\bf 1}_{\omega(C_{neg})}
\]
is a flow of $\Sigma$ with coefficients in ${\Bbb A}$ if $a\in{\rm Tor}_2({\Bbb A})$, and is not a flow if $a\not\in{\rm Tor}_2({\Bbb A})$.

\subsection{Flow, boundary, and homology groups}

\begin{thm}[Characterization of the flow group]\label{thm:flow-group-structure}
Let $B$ be a canonical base of $\Sigma$. Then
\begin{align}\label{eq:flow-group-structure}
{\rm F}(\Sigma,{\Bbb A})
&=\Big(\bigoplus_{\rho}{\rm Tor}_2({\Bbb A})
\cdot {\bm 1}_{\omega(C_\rho)}\Big)
\oplus \Big(\bigoplus_{e\in B^c}{\Bbb A}
\cdot {\bf I}_{\omega(C_e)}\Big)\\
&\cong ({\rm Tor}_2{\Bbb A})^{u_c(\Sigma)}\oplus{\Bbb A}^{{\rm cr}(\Sigma)},
\end{align}
where $C_\rho$ ranges over all negative circles contained in $B$ and ${\rm cr}(\Sigma)=|B^c|$.
\end{thm}
\begin{proof}
Since ${\rm F}(\Sigma,{\Bbb A})=\bigoplus_i{\rm F}(\Sigma_i,{\Bbb A})$, where $\Sigma_i$ ranges over the components of $\Sigma$, it is enough to show \eqref{eq:flow-group-structure} for each component of $\Sigma$. We may assume that $\Sigma$ itself is connected. First of all, the right-hand side of \eqref{eq:flow-group-structure} is clearly contained in the left-hand side, since the chains, $a\cdot{\bf 1}_{\omega(C_\sigma)}$ with $a\in{\rm Tor}_2({\Bbb A})$ and $b\cdot{\bf I}_{\omega(C_e)}$ with $b\in{\Bbb A}$, are flows of $\Sigma$ with coefficients in $\Bbb A$, according to Lemma~\ref{lem:negative-circle-flow} and circuit flows \eqref{eq:circuit flow}. It is enough to show that each member of the left-hand side of \eqref{eq:flow-group-structure} is contained in the right-hand side, and to verify the direct sum.

Let $\bf f$ be a flow of $\Sigma$ with coefficient in $\Bbb A$. Consider the flow ${\bf f}':={\bf f}-{\bf f}_0$, where
\begin{equation}\label{eq:flow-difference}
{\bf f}_0:=\sum_{e\in B^c}{\bf f}(\vec{e\,})\cdot{\bf I}_{\omega(C(B,e))}, \quad \vec{e}\in\omega(C(B,e)).
\end{equation}
For each $e'\in B^c$ with $\vec{e}{\,'}\in\omega(C(B,e'))$, we have ${\bf f}_0(\vec{e}{\,'})={\bf f}(\vec{e}{\,'})$ by \eqref{eq:flow-difference}; consequently, ${\bf f}'(\vec{e}{\,'})=0$. We see that ${\bf f}'$ vanishes on $B^c$.

{\em Case} 1: {\em $\Sigma$ contains outer-edges or $\Sigma$ is balanced.}

Notice that the base $B$ has the form $B=T\cup e_1$ for the former case and $B=T$ for the latter case, where $T$ is a spanning tree of $\Sigma$ and $e_1$ is an outer-edge attached to $T$. The latter case can be viewed as the former case without $e_1$. We claim that ${\bf f}'$ vanishes on $B=T\cup e_1$.

If the tree $T$ contains edges, then there exist at least two leaves of $T$, one of the leaves, say $u$, must be distinct from the endpoint $v$ of the outer-edge $e_1$. Let $e_u$ denote the unique edge of $T$ at the leaf $u$. Since ${\bf f}'$ is a flow vanishing on $B^c$, we see that ${\bf f}'$ vanishes on all edges at $u$, except the edge $e_u$. Then ${\bf f}'(\vec{e}_u\,)=\partial{\bf f}'(u)=0$. Now consider $B':=B\smallsetminus ue_u$, which is a tree attached with the outer-edge $e_1$. Note that ${\bf f}'$ vanishes on ${B'}^c$. Again, if $B'$ still contains edges other than $e_1$, that is, $T':=T\smallsetminus ue_u$ contains edges, then there exists at least one leaf $u_1$ of $T'$ other than $v$. Let $e_{u_1}$ denote the unique edge of $T'$ at the leaf $u_1$. Likewise, ${\bf f}'$ vanishes on all edges at $u_1$, except $e_{u_1}$; then ${\bf f}'(\vec{e}_{u_1})=\partial{\bf f}'(u_1)=0$. Continue this procedure; we see that ${\bf f}'$ vanishes on all edges of $T$. Finally, ${\bf f}'$ vanishes on all edges at $v$, except $e_1$; then ${\bf f}'(\vec{e}_1)=\partial{\bf f}'(v)=0$. We have shown that ${\bf f}'$ vanishes on $B$. Thus ${\bf f}'={\bf 0}$; this means that the flow $\bf f$ belongs to the right-hand side of \eqref{eq:flow-group-structure}.

We are left to show that \eqref{eq:flow-group-structure} is a direct sum. Let the zero flow $\bf 0$ be written as
\[
{\bf 0}=\sum_{e\in B^c}a_e\cdot{\bf I}_{\omega(C(B,e))}.
\]
For each edge $e'\in B^c$, we evaluate both sides at $e'$. Recall ${\bf I}_{\omega(C(B,e))}(\vec{e}^{\,\prime})=\delta_{e,e'}$ for $e,e'\in B^c$ with $\vec{e}{\,'}\in\omega(C(B,e'))$. We obtain
\[
0=\sum_{e\in B^c}a_e\cdot{\bf I}_{\omega(C(B,e))}(\vec{e}{\,'})=a_{e'}.
\]
This means that \eqref{eq:flow-group-structure} is indeed a direct sum.

{\em Case} 2: {\em $\Sigma$ contains negative circles but no outer-edges.}

Notice that $B=T\cup e_0$, where $T$ is a spanning tree and $e_0$ is an edge such that $T\cup e_0$ contains a unique negative circle $C_0$. We claim that ${\bf f}'$ vanishes on $B\setminus C_0$. Recall that $B$ is the negative circle $C_0$ attached with some trees. If $B\setminus C_0$ contains edges, then $B$ contains at least one leaf $u$ not on the circle $C_0$, having exactly one edge $e_u$ at $u$ in $B$. Note that ${\bf f}'$ vanishes on all edges at $u$, except the edge $e_u$. Then ${\bf f}'(\vec{e}_u)=(\partial{\bf f}')(u)=0$. Now consider $B':=B\smallsetminus ue_u$. If $B'\smallsetminus C_0$ still contains edges, then there exists at least one leaf $u_1$ not on $C_0$, having exactly one edge $e_{u_1}$ at $u_1$ in $B'$. Note that ${\bf f}'$ vanishes on all edges at $u_1$, except $e_{u_1}$. Then ${\bf f}'(\vec{e}_{u_1})=(\partial{\bf f}')(u_1)=0$. Continue this procedure; we see that ${\bf f}'$ vanishes on $B\setminus C_0$. It turds out that ${\bf f}'$ is a flow supported on the negative circle $C_0$. By Lemma~\ref{lem:negative-circle-flow}, ${\bf f}'=a\cdot{\bf 1}_{\omega(C_0)}$ with $a\in{\rm Tor}_2({\Bbb A})$. We have seen that each flow $\bf f$ can be written as
\[
{\bf f}=a\cdot{\bf 1}_{\omega(C_0)}+\sum_{e\in B^c}{\bf f}(\vec{e\,})\cdot{\bf I}_{\omega(C(B,e))},
\]
which is clearly a member in the right-hand side of \eqref{eq:flow-group-structure}.

To see that the right-hand side of \eqref{eq:flow-group-structure} is a direct sum, consider the zero flow $\bm 0$ written in the form
\[
{\bf 0}=a_0\cdot{\bf 1}_{\omega(C_0)}+
\sum_{e\in B^c}a_e\cdot{\bf I}_{\omega(C(B,e))},
\]
where $a_0\in{\rm Tor}_2({\Bbb A})$ and $a_e\in{\Bbb A}$. For each edge $e'$ of $B^c$ and evaluate both sides at the arc $\vec{e}{\,'}\in\omega(C(B,e'))$. Since ${\bf I}_{\omega(C(B,e))}(\vec{e}{\,'})=\delta_{e,e'}$, we obtain $a_{e'}=0$. It follows that $a_0\cdot{\bm 1}_{\omega(C_0)}=\bf 0$. Of course, $a_0=0$. This means that the right-hand side of \eqref{eq:flow-group-structure} is a direct sum.
\end{proof}

Notice that ${\rm F}(\Sigma,{\Bbb A})$ is not necessarily isomorphic to the tensor product ${\rm F}(\Sigma,{\Bbb Z})\otimes{\Bbb A}$. In fact, if $\Sigma$ contains no negative circles, it is indeed that ${\rm F}(\Sigma,{\Bbb A})\cong {\rm F}(\Sigma,{\Bbb Z})\otimes{\Bbb A}$. However, if $\Sigma$ is connected, unbalanced and contains no outer-edges, we have
\[
{\rm F}(\Sigma,{\Bbb A})\cong{\rm Tor}_2({\Bbb A})\oplus
\big({\rm F}(\Sigma,{\Bbb Z}) \otimes{\Bbb A}\big).
\]

\begin{thm}[Characterization of the boundary group]\label{thm:boundary-group}
\begin{equation}\label{eq:boundary group}
{\rm B}_0(\Sigma,{\Bbb A})\cong (2{\Bbb A})^{u_c(\Sigma)}\oplus{\Bbb A}^{r(\Sigma)-u_c(\Sigma)}.
\end{equation}
More specifically, when $\Sigma$ is connected, we have the following concrete expressions.
\begin{enumerate}[\rm (a)]
\item
If $\Sigma$ is balanced, with a switching $\nu$ such that all edges of $\Sigma^\nu$ are positive, then
\begin{align*}
{\rm B}_0(\Sigma,{\Bbb A})
&= \Big{\{\bf b}\in {\rm C}_0(\Sigma,{\Bbb A}):
\sum_{v\in V}\nu(v){\bf b}(v)=0\Big\}\\
&\cong {\Bbb A}^{|V|-1}.
\end{align*}

\item
If $\Sigma$ is unbalanced and contains no outer-edges, then
\begin{align*}
{\rm B}_0(\Sigma,{\Bbb A})
&=\Big{\{\bf b}\in {\rm C}_0(\Sigma,{\Bbb A}):
\sum_{v\in V}\nu(v){\bf b}(v)\in 2{\Bbb A}\Big\}\\
&\cong 2{\Bbb A}\oplus {\Bbb A}^{|V|-1}.
\end{align*}

\item
If $\Sigma$ contains outer-edges, then ${\rm B}_0(\Sigma,{\Bbb A})={\rm C}_0(\Sigma,{\Bbb A})\cong{\Bbb A}^{|V|}$.
\end{enumerate}
\end{thm}
\begin{proof}
(a) and (b): The signed graph $\Sigma$ contains no outer-edges. Let $T$ be a spanning tree of $\Sigma$ and $\nu$ be a switching such that all edges of $T^\nu$ are positive; the switching $\nu$ is unique up to sign. Then all edges of $\Sigma^\nu$ are positive in the case (a) and at least one edge of $\Sigma^\nu\setminus T^\nu$ is negative in the case (b). After switching $\nu$, each flow $\bf c$ of $\Sigma$ with coefficients in $\Bbb A$ becomes a flow ${\bf c}^\nu$ of $\Sigma^\nu$. Note that $[v,\vec{e}\,]=0$ for each positive loop $e$, and $[v,\vec{e}\,]=\pm2$ for each negative loop $e$ at $v$. For each link edge $e$ with endpoints $u$ and $v$, we have
\[
[u,\vec{e}^{\,\nu}]+[v,\vec{e}^{\,\nu}]=\left\{\begin{array}{rl}
0 & \mbox{if $e^\nu$ is positive,}\\
\pm2 & \mbox{if $e^\nu$ is negative.}\\
\end{array}\right.
\]
It follows that
\begin{align*}
\sum_{v\in V}(\partial{\bf c}^\nu)(v)
&=\sum_{v\in V}\sum_{e\in E}[v,\vec{e}^{\,\nu}]
{\bf c}^\nu(\vec{e}^{\,\nu})\\
&=\sum_{e\in E} {\bf c}^\nu(\vec{e}^{\,\nu})
\sum_{v\in V}[v,\vec{e}^{\,\nu}]\\
&=\left\{\begin{array}{rl}
0  & \mbox{for the case (a),}\\
2a & \mbox{for the case (b), where $a\in{\Bbb A}$.}
\end{array}\right.
\end{align*}
Since $\partial{\bf c}^\nu=(\partial{\bf c})^\nu$ and $(\partial{\bf c})^\nu(v)=\nu(v)(\partial{\bf c})(v)$, we obtain
\begin{align*}
\sum_{v\in V}(\partial{\bf c}^\nu)(v)
&= \sum_{v\in V}\nu(v)(\partial{\bf c})(v).
\end{align*}
We see that the 0-chain $\partial{\bf c}$ satisfies the required conditions for the cases (a) and (b).

Conversely, given a 0-chain $\bf b$ with coefficients in $\Bbb A$, satisfying
\[
\sum_{v\in V}\nu(v){\bf b}(v)\left\{\begin{array}{ll}=0 & \mbox{for the case (a),}\\
\in 2{\Bbb A} & \mbox{for the case (b).}
\end{array}\right.
\]
We aim to construct a 1-chain ${\bf c}$ of $\Sigma$ such that $\partial{\bf c}={\bf b}$. We construct such a 1-chain $\bf c$ supported on $T$ by induction on the number of edges of $T$. We first consider the case that $T$ contains no edges, that is, $T$ consists of a single vertex $v_0$. Then all edges of $\Sigma$ must be loops at $v_0$; all loops are positive for the case (a), and at least one loop is negative for the case (b). By definition of $\bf b$, for the case (a), ${\bf b}(v_0)=0$; for the case (b), let us write ${\bf b}(v_0)=2a$ with $a\in\Bbb A$. Let $\bf c$ be a $1$-chain, vanishing on all loop arcs for the case (a), and for the case (b) having value $a$ at one negative loop arc pointing to $v_0$ and vanishing on all other loop arcs. Clearly, $\partial{\bf c}={\bf b}$.

Let $T$ contain some edges. The tree $T$ contains at least two leaves (vertices of degree $1$). Choose a leaf $v_0$ and an edge $e_0$ of $T$ with endpoints $v_0,v_1$. Consider the signed graph $\Sigma_1:=\Sigma^\nu/e_0^\nu$ with the spanning tree $T_1:=T^\nu\smallsetminus e_0^\nu v_0$. Clearly, $\Sigma_1$ is connected and unbalanced. We identity $V(\Sigma_1)$ as the set $V\smallsetminus v_0$ and $E(\Sigma_1)$ as $E\smallsetminus e_0$. Consider the $0$-chain ${\bf b}_1$ of $\Sigma_1$, defined by
\[
{\bf b}_1(v)=\left\{\begin{array}{ll}
{\bf b}^\nu(v) & \mbox{if}\; v\neq v_1,\\
{\bf b}^\nu(v_1)+{\bf b}^\nu(v_0) & \mbox{if}\; v=v_1,
\end{array}\right.
\]
where ${\bf b}^\nu(v)=\nu(v){\bf b}(v)$ for all $v\in V$. Note that $\nu(v_1)\sigma(e_0)\nu(v_0)=\sigma^\nu(e_0)=1$, that is, $\nu(v_1)\sigma(e_0)=\nu(v_0)$. It follows that for the case (a),
\begin{align*}
\sum_{v\in V(T_1)}{\bf b}_1(v)
&=\sum_{v\in V(T)}{\bf b}^\nu(v)\\
&=\left\{\begin{array}{ll}
=0 & \mbox{for the case (a),}\\
\in 2{\Bbb A} & \mbox{for the case (b).}\\
\end{array}\right.
\end{align*}
By induction there exists a 1-chain ${\bf c}_1$ of $\Sigma_1$ such that $\partial{\bf c}_1={\bf b}_1$. We define a 1-chain ${\bf c}^\nu$ of $\Sigma^\nu$ by
\[
{\bf c}^\nu(\vec{e\,}^\nu)=\left\{\begin{array}{ll}
{\bf c}_1(\vec{e\;}^\nu\!/e_0^\nu) & \mbox{if $e\neq e_0$,}\\
{[v_0,\vec{e}_0^{\;\nu}]}\big({\bf b}^\nu(v_0)-\sum_{e_1\neq e_0}[v_0,\vec{e}_1^{\,\nu}] {\bf c}_1({\vec{e}_1^{\,\nu}}/e_0^\nu)\big) & \mbox{if $e=e_0$.}
\end{array}\right.
\]
We aim to show that $\partial{\bf c}^\nu={\bf b}^\nu$ on $\Sigma^\nu$. Note that $[v,\vec{e}_0^{\,\nu}]=0$ and $[v,\vec{e}^{\,\nu}]=[v,\vec{e}^{\,\nu}/e_0^\nu]$ for vertices $v\neq v_0,v_1$ and edges $e\neq e_0$. It is clear that
\begin{align*}
(\partial{\bf c}^\nu)(v)&=\sum_{e\neq e_0}[v,\vec{e}^{\,\nu}]
{\bf c}^\nu(\vec{e\,}^\nu)\\
&=\sum_{e\neq e_0}[v,\vec{e\,}^\nu/e_0^\nu]{\bf c}_1(\vec{e\,}^\nu/e_0^\nu)\\
&={\bf b}_1(v)={\bf b}^\nu(v) \quad
\mbox{for}\; v\neq v_0,v_1.
\end{align*}
For the vertex $v_0$, we trivially have
\begin{align*}
(\partial{\bf c}^\nu)(v_0)&={\bf c}^\nu(\vec{e}_0)
+\sum_{e\neq e_0}[v_0,\vec{e\,}^\nu]{\bf c}^\nu(\vec{e\,}^\nu)\\
&={\bf b}^\nu(v_0)-\sum_{e_1\neq e_0}[v_0,\vec{e}_1^{\,\nu}]{\bf c}_1(\vec{e\,}^\nu_1/e_0^\nu)+\sum_{e\neq e_0} [v_0,\vec{e\,}^\nu]{\bf c}_1(\vec{e\,}^\nu/e_0^\nu)\\
&={\bf b}^\nu(v_0).
\end{align*}
For the vertex $v_1$, we have
\begin{align*}
(\partial{\bf c}^\nu)(v_1)&=[v_1,\vec{e}_0^{\,\nu}]{\bf c}^\nu(\vec{e}_0^{\,\nu})
+\sum_{e\neq e_0}[v_1,\vec{e}^{\,\nu}]{\bf c}^\nu(\vec{e}^{\,\nu})\\
&=-\Big({\bf b}^\nu(v_0)-\sum_{e_1\neq e_0}[v_0,\vec{e}_1^{\,\nu}]{\bf c}_1(\vec{e}_1^{\,\nu}/e_0)\Big)
+\sum_{e\neq e_0} [v_1,\vec{e}^{\,\nu}]{\bf c}_1(\vec{e}^{\,\nu}/e_0)\\
&=-{\bf b}^\nu(v_0)+(\partial{\bf c}_1)(v_1)=-{\bf b}^\nu(v_0)+{\bf b}_1(v_1)={\bf b}^\nu(v_1).
\end{align*}
We have shown that $\partial {\bf c}^\nu={\bf b}^\nu$. Consequently, $\partial{\bf c}={\bf b}$ by \eqref{eq:switching preversing}.

(c) Let $\bf b$ be a 0-chain of $\Sigma$ with coefficients in $\Bbb A$. Fix a spanning tree $T$ of the compact signed graph $\Sigma\setminus\{\mbox{outer-edges}\}$, choose an outer-edge $e_0$ with endpoint $v_0$ and set $B:=T\cup e_0$. Our aim is to construct a 1-chain $\bf c$ on $B$ such that $\partial{\bf c}={\bf b}$, by induction on the number of edges of $B$. If $B$ contains exactly the edge $e_0$, that is, $T$ consists of a single vertex $v_0$, we define $\bf c$ by setting ${\bf c}(\vec{e}_0):={\bf b}(v_0)$ for the arc $\vec{e}_0$ pointing to $v_0$. Clearly, $\partial{\bf c}(v_0)={\bf b}(v_0)$.

If $B$ contains two or more edges, that is, $T$ contains edges, then $T$ contains at leat two leaves. Choose a leaf $v_1$ and an edge $e_1$ of $T$ with endpoints $v_1$ and $v_2$. Remove the edge $e_1$ and its endpoint $v_1$ to obtain $B_1:=B\setminus e_1v_1$. Consider the 0-chain ${\bf b}_1$ on $B_1$, defined by
\[
{\bf b}_1(v)=\left\{\begin{array}{ll}
{\bf b}(v) & \mbox{if $v\neq v_2$,}\\
{\bf b}(v_2)+\sigma(e_1){\bf b}(v_1) & \mbox{if $v=v_2$.}
\end{array}\right.
\]
By induction there exists a 1-chain ${\bf c}_1$ on $B_1$ such that $\partial{\bf c}_1={\bf b}_1$. We consider ${\bf b}_1,{\bf c}_1$ as chains on $B$, such that ${\bf b}_1(v_1)=0$ and ${\bf c}_1(\vec{e}_1)=0$. We define $1$-chain ${\bf c}$ on $B$ by
\[
{\bf c}(\vec{e}\,)=\left\{\begin{array}{ll}
{\bf c}_1(\vec{e}\,) & \mbox{if $e$ is an edge of $B_1$,}\\
{\bf b}(v_1) & \mbox{if $e=e_1$ and $\vec{e}_1$ points to $v_1$.}
\end{array}\right.
\]
Now for vertices $v\neq v_1,v_2$, by induction hypothesis, we have $\partial{\bf c}(v)=\partial{\bf c}_1(v)={\bf b}_1(v)={\bf b}(v)$. For the vertex $v_1$, we have $\partial{\bf c}(v_1)=\partial[{\bf c}_1+{\bf b}(v_1)\vec{e}_1](v_1) ={\bf b}(v_1)\partial\vec{e}_1(v_1)={\bf b}(v_1)$. For the vertex $v_2$, choose arc $\vec{e}_1$ pointing to $v_1$; we have $\partial\vec{e}_1(v_2)=-\sigma(e_1)$ and
\begin{align*}
\partial{\bf c}(v_2)&=\partial\big({\bf c}_1+{\bf b}(v_1)\vec{e}_1\big)(v_2)\\
&={\bf b}_1(v_2) +{\bf b}(v_1)\partial\vec{e}_1(v_2)={\bf b}(v_2).
\end{align*}
We have shown that $\partial{\bf c}={\bf b}$. The chain $\bf c$ can be viewed as a 1-chain of $\Sigma$, vanishing on $B^c$ and $\partial{\bf c}={\bf b}$.

The three cases can be casted into one formula: ${\rm B}_0(\Sigma,{\Bbb A})\cong ({\rm Tor}_2{\Bbb A})^{u_c(\Sigma)}\oplus{\Bbb A}^{r_0(\Sigma)}$, where $r_0(\Sigma)=r(\Sigma)-u_c(\Sigma)$. Indeed, let $\Sigma$ be connected. If $\Sigma$ is balanced, then $u_c(\Sigma)=0$, $r(\Sigma)=|V|-1$ and $r_0(\Sigma)=|V|-1$. If $\Sigma$ is unbalanced and contains no outer-edges, then $u_c(\Sigma)=1$, $r(\Sigma)=|V|$ and $r_0(\Sigma)=|V|-1$.
If $\Sigma$ contain outer-edges, then $u_c(\Sigma)=0$, $r(\Sigma)=|V|$ and $r_0(\Sigma)=|V|$.
\end{proof}

\begin{thm}[Characterization of the homology group]
\begin{equation}\label{eq:homology group}
{\rm H}_0(\Sigma,{\Bbb A})\cong ({\Bbb A}/2{\Bbb A})^{u_c(\Sigma)}\oplus{\Bbb A}^{b(\Sigma)}.
\end{equation}
If $\Bbb A$ is finite, then ${\rm H}_0(\Sigma,{\Bbb A})\cong ({\rm Tor}_2{\Bbb A})^{u_c(\Sigma)}\oplus{\Bbb A}^{b(\Sigma)}$.
\end{thm}
\begin{proof}
Notice that ${\rm H}_0(\Sigma,{\Bbb A})=\bigoplus_i{\rm H}_0(\Sigma_i,{\Bbb A})$, where $\Sigma_i$ are components of $\Sigma$. Let $\Sigma$ be connected. If $\Sigma$ contains outer-edges, then ${\rm H}_0(\Sigma,{\Bbb A})=0$ by part (c) of Theorem \ref{thm:boundary-group}. Clearly, ${\rm H}_0(\Sigma,{\Bbb A})$ is of the form \eqref{eq:homology group}, with $u_c(\Sigma)=0$ and $b(\Sigma)=0$.

Let us consider the case that $\Sigma$ contains no outer-edges. If $\Sigma$ is balanced, we choose a switching $\nu$ such that all edges of $\Sigma^\nu$ are positive, and consider the linear map
\[
\ell:{\rm C}_0(\Sigma,{\Bbb A})\rightarrow{\Bbb A},\quad
\ell({\bf b})=\sum_{v\in V}\nu(v){\bf b}(v),
\]
which is obviously an epimorphism. Clearly, $\ker\ell={\rm B}_0(\Sigma,{\Bbb A})$. Consequently, ${\rm H}_0(\Sigma,{\Bbb A})={\rm C}_0(\Sigma,{\Bbb A})/\ker\ell\cong {\Bbb A}$, which is the form \eqref{eq:homology group}, with $u_c(\Sigma)=0$ and $b(\Sigma)=1$.

If $\Sigma$ is unbalanced and contains no outer-edges, consider the linear map
\[
\ell:{\rm C}_0(\Sigma,{\Bbb A})\rightarrow{\Bbb A},\quad
\ell({\bf b})=\sum_{v\in V}{\bf b}(v),
\]
which is obviously an epimorphism. According to part (b) of Theorem \ref{thm:boundary-group}, we have ${\rm B}_0(\Sigma,{\Bbb A})=\ell^{-1}(2{\Bbb A})$. Applying the isomorphism theorem of universal algebra or arguing directly, we obtain ${\rm H}_0(\Sigma,{\Bbb A})\cong {\Bbb A}/2{\Bbb A}$, which is the form \eqref{eq:homology group}, with $u_c(\Sigma)=1$ and $b(\Sigma)=0$.

Now for arbitrary $\Sigma$, the invariants $u_c(\Sigma)$ and $b(\Sigma)$ are additive on the components of $\Sigma$, and the homology group ${\rm H}_0(\Sigma,{\Bbb A})$ is additive (in direct sum) on the components of $\Sigma$. We conclude that the homology group formula \eqref{eq:homology group} is valid for arbitrary signed graphs $\Sigma$.
\end{proof}

\noindent
{\bf Remark.} In general ${\Bbb A}/2{\Bbb A}$ is not isomorphic to ${\rm Tor}_2({\Bbb A})$. For instance, ${\Bbb Z}/2{\Bbb Z}\cong\{0,1\}$, ${\rm Tor}_2({\Bbb Z})=0$. However, if $\Bbb A$ is finite, then ${\Bbb A}/2{\Bbb A}\cong{\rm Tor}_2({\Bbb A})$, according to the classification of finite abelian groups.

\section{Bivariate flow polynomial}

Let ${\rm F}_{\rm nz}(\Sigma,{\Bbb A})$ denote the set of nowhere-zero flows of $\Sigma$ with coefficients in $\Bbb A$. When $\Sigma$ contains no outer-edges and ${\Bbb A}$ is finite of odd order, Beck and Zaslavsky \cite[Theorem 4.1]{Beck-Zaslavsky-JCTB} showed that there exits a unique polynomial $\varphi(\Sigma,x)$ such that
\[
\varphi(\Sigma,|{\Bbb A}|)=|{\rm F}_{\rm nz}(\Sigma,{\Bbb A})|,
\]
and asked for the significance when the order $|\Bbb A|$ is even; see \cite[Problem 4.2]{Beck-Zaslavsky-JCTB}. It is well known that the counting $|{\rm F}_{\rm nz}(\Sigma,{\Bbb A})|$ has no polynomial pattern in terms of the order $|\Bbb A|$. More specifically, applying deletion-contraction method, DeVos, Rollov\'{a} and \u{S}\'{a}mal \cite{DeVos-Rollova-Samal} showed that, for $\Sigma$ without outer-edges and finite $\Bbb A$ with torsion-2 subgroup of fixed order $2^d$, the counting pattern of $|{\rm F}_{\rm nz}(\Sigma,{\Bbb A})|$ has a polynomial form in terms of $|\Bbb A|$. However, it was stated in a strange way that for each $d\geq 0$ there exists a polynomial $f_d$ such that $f_d(n)=|{\rm F}_{\rm nz}(\Sigma,{\Bbb A})|$ for every abelian group ${\Bbb A}$ with largest $\epsilon(\Gamma)=d$ and $|{\Bbb A}|=2^dn$, where $\epsilon(\Gamma)$ is the largest integer $k$ for $\Bbb A$ to have a subgroup isomorphic to ${\Bbb Z}_2^k$, where ${\Bbb Z}_2={\Bbb Z}/2{\Bbb Z}$. Applying the inclusion-exclusion principle, Ren and Qian \cite{Ren-Qian} obtained the same result for signed graphs without outer-edges, but also stated in the same strange way.

We treat the counting patterns of $|{\rm F}_{\rm nz}(\Sigma,{\Bbb A})|$ as a bivariate polynomial in terms of the orders of ${\rm Tor}_2(\Bbb A)$ and ${\Bbb A}$. For each $S$ edge subset of $\Sigma$, the subgroup
\[
{\rm F}_S(\Sigma,{\Bbb A})=\{{\bf f}\in{\rm F}(\Sigma,{\Bbb A}): {\bf f}(\vec{e}\,)=0\;\mbox{for}\; e\in S\}
\]
of ${\rm F}(\Sigma,{\Bbb A})$ can be identified as the group of flows of the spanning signed subgraph $\Sigma|S^c=(V,S^c,\sigma|_{S^c})$ over $\Bbb A$. For each edge $e$ we write ${\rm F}_{\{e\}}$ as ${\rm F}_e$. We obtain the following proposition.

\begin{prop}\label{prop:flow group of spanning signed subgraph}
For each $S$ edge subset of $\Sigma$,
\[
{\rm F}_S(\Sigma,{\Bbb A})\cong({\rm Tor}_2{\Bbb A})^{u_c(S^c)}
\oplus{\Bbb A}^{{\rm cr}(S^c)}.
\]
\end{prop}
\begin{proof}
It is trivial according to Theorem \ref{thm:flow-group-structure}.
\end{proof}

The following lemma seems obvious and is a folklore that we have no good reference.

\begin{lem}\label{lem:polynomial-unique}
Let $f(t,x)$ and $g(t,x)$ be two bivariate polynomials of $t$-degree at most $m$ and $x$-degree at most $n$, such that $f=g$ at $(m+1)(n+1)$ distinct places $(p_i,q_j)$, $0\leq i\leq m$, $0\leq j\leq n$, then $f=g$ as polynomials, that is, their coefficients are the same.
\end{lem}
\begin{proof}
It is well known that if two univariate polynomials of degree at most $n$ has same values at $n+1$ distinct places, then the two polynomials are identical, that is, their coefficients are the same. In fact, it can be verified straightforwardly by setting linear equations on the coefficients (viewed as variables), applying Crammer's rule and Vandermonde matrix and determinant.

Let $f(t,x)=\sum a_{ij}t^ix^j$ and $g(t,x)=\sum b_{ij}t^ix^j$. We further write $f$ and $g$ as
\[
f(t,x)=\sum_{j=0}^nf_j(t)x^j, \quad
g(t,x)=\sum_{j=0}^ng_j(t)x^j.
\]
For each fixed $i$, with $0\leq i\leq m$, consider the univariate polynomials $f(p_i,x)$ and $g(p_i,x)$. Since $f(p_i,q_j)=g(p_i,q_j)$, $0\leq j\leq n$, we obtain $f(p_i,x)=g(p_i,x)$ as polynomials, where $0\leq i\leq m$. This means that the coefficients of $f(p_i,x)$ and $g(p_i,x)$ are the same, that is, $f_j(p_i)=g_j(p_i)$, $0\leq i\leq m$ and $0\leq j\leq n$. Now for each fixed $j$, with $0\leq j\leq n$, the polynomials $f_j(t)$ and $g_j(t)$ have degree at most $m$, and have the same values at distinct numbers $p_i$, $0\leq i\leq m$. Again, we  obtain $f_j(t)=g_j(t)$ as polynomials, $0\leq j\leq n$. Clearly, $f(t,x)=g(t,x)$ as polynomials of two variables.
\end{proof}

\begin{thm}[Bivariate flow polynomial]\label{thm:bivariate-flow-poly} For each signed graph $\Sigma$ with edges, there exists a unique bivariate polynomial $\varphi(\Sigma;t,x)$, called the {\em flow polynomial} of $\Sigma$, such that for arbitrary finite abelian groups $\Bbb A$,
\begin{equation}\label{eq:flow-poly-pattern}
|{\rm F}_{\rm nz}(\Sigma,{\Bbb A})|=
\varphi\big(\Sigma;|{\rm Tor}_2{\Bbb A}|,|{\Bbb A}|\big).
\end{equation}
Moreover, $\varphi(\Sigma;t,x)$ admits the expansion
\begin{equation}\label{eq:expansion}
\varphi(\Sigma;t,x)=
\sum_{S\subseteq E}(-1)^{|E\setminus S|} t^{u_c(S)} x^{{\rm cr}(S)},
\end{equation}
and $\varphi(\Sigma;t,x)=0$ if and only if $\Sigma$ contains a co-loop, that is, an edge whose removal
increases exactly one new balanced component.
\end{thm}
\begin{proof}
For each $S$ edge subset of $\Sigma$, let ${\rm F}_S(\Sigma,{\Bbb A})$ denote the subgroup of flows vanishing on $S$, that is, flows supported on $S^c:=E\setminus S$. Then by Proposition~\ref{prop:flow group of spanning signed subgraph},
\[
{\rm F}_{E\setminus S}(\Sigma,{\Bbb A})
\cong {\rm F}(\Sigma|S,{\Bbb A})
\cong ({\rm Tor}_2{\Bbb A})^{u_c(S)}
\oplus {\Bbb A}^{{\rm cr}(S)}.
\]
Applying inclusion-exclusion, we obtain
\begin{align*}
|{\rm F}_{\rm nz}(\Sigma,{\Bbb A})|&=
\Big|{\rm F}(\Sigma,{\Bbb A})\setminus
\bigcup_{e\in E} {\rm F}_e(\Sigma,{\Bbb A})\Big|\\
&=\sum_{S\subseteq E}(-1)^{|S|} |{\rm F}_S(\Sigma,{\Bbb A})|\\
&=\sum_{S\subseteq E}(-1)^{|E\setminus S|}
|{\rm F}_{E\setminus S}(\Sigma,{\Bbb A})|\\
&= \sum_{S\subseteq E}(-1)^{|E\setminus S|}
|{\rm Tor}_2{\Bbb A}|^{u_c(S)} |{\Bbb A}|^{{\rm cr}(S)}.
\end{align*}
It is clear that $\varphi$ satisfies the required condition of the theorem. Since the cardinality $|{\Bbb A}|$ can be any positive integers and $|{\rm Tor}_2{\Bbb A}|$ can be any power $2^k$ of nonnegative integers $k$, applying Lemma \ref{lem:polynomial-unique}, such bivariate polynomial must be unique.
\end{proof}

For any signed graph $(V,\varnothing)$ without edges, its flow polynomial is conventionally defined as $\varphi(\varnothing;t,x)\equiv 1$. Let $\Sigma_i$ denote components of $\Sigma$. Notice that the flow group ${\rm F}(\Sigma,{\Bbb A})$ is a direct sum of flow groups of the (vertex) disjoint signed graphs $\Sigma_i$, and the direct sum preserves the property of nowhere-zero, that is,
\[
{\rm F}_{\rm nz}(\Sigma,{\Bbb A})=\bigoplus_i{\rm F}_{\rm nz}(\Sigma_i,{\Bbb A}).
\]
It follows that
\begin{equation}
\varphi(\Sigma;t,x)=\prod_i\varphi(\Sigma_i;t,x).
\label{eq:product-formula-of-flow-poly}
\end{equation}
We only need to consider connected signed graphs when we study flow polynomials.

The flow polynomial of signed graph $\Sigma$ in the literature is a univariate polynomial $\varphi_{\rm odd}(\Sigma,x)$ such that $\varphi_{\rm odd}(\Sigma,n)$, with $n$ odd positive integers, counts the number of flows of $\Sigma$ valued in the groups ${\Bbb Z}_n$; see Beck and Zaslavsky \cite{Beck-Zaslavsky-JCTB}. Since ${\rm Tor}_2({\Bbb Z}_n)=\{0\}$ for $n$ odd, the polynomial $\varphi_{\rm odd}(\Sigma,x)$ is a specialized case of the bivariate flow polynomial $\varphi(\Sigma;t,x)$ as
\begin{equation}\label{eq:odd-flow polynomial}
\varphi_{\rm odd}(\Sigma,x):=\varphi(\Sigma;1,x).
\end{equation}
However, for $n$ even positive integers, it has been puzzled for a long time in the community of graph theory: whether or not there exists a univariate polynomial $\varphi_{\rm even}(\Sigma,x)$ such that
\[
\varphi_{\rm even}(\Sigma,n)=|{\rm F}_{\rm nz}(\Sigma,{\Bbb Z}_n)|.
\]
See Zaslavsky \cite{Zaslavsky-signed-graphs-DAM, Zaslavsky-signed-graph-coloring-DM} and Problem 4.2 of Beck and Zaslavsky \cite{Beck-Zaslavsky-JCTB}.
In fact, for even $n=2k$, since ${\rm Tor}_2({\Bbb Z}_n)=\{0,k\}\cong{\Bbb Z}_2$, the following theorem answers part of Problem 4.2.

\begin{thm}
For each signed graph $\Sigma$ with edges, there exist unique polynomials $\varphi_{\rm even}(\Sigma,x)$ and $\varphi_{\rm odd}(\Sigma,x)$, such that for all integers $n\geq 2$,
\begin{align}
\varphi_{\rm odd}(\Sigma,n)&=|{\rm F}_{\rm zn}(\Sigma,{\Bbb Z}_n)|, \quad n={\rm odd},\\
\varphi_{\rm even}(\Sigma,n)&=|{\rm F}_{\rm zn}(\Sigma,{\Bbb Z}_n)|, \quad n={\rm even}.
\end{align}
Moreover, $\varphi_{\rm even}(\Sigma,x)$ and $\varphi_{\rm odd}(\Sigma,x)$ are specialized from the flow polynomial $\varphi(\Sigma;t,x)$ as
\begin{align}\label{eq:even-flow polynomial}
\varphi_{\rm odd}(\Sigma,x)
&=\varphi(\Sigma;1,x),\\
\varphi_{\rm even}(\Sigma,x)
&=\varphi(\Sigma;2,x).
\end{align}
\end{thm}

Recall that a {\em bridge} of $\Sigma$ (viewed as a unsigned graph) is an edge whose removal increases exactly one connected component. A {\em co-loop} of $\Sigma$, also known as {\em isthmus}, is an edge that is a bond (co-circuit) of $\Sigma$. According matroid theory, there are three equivalent statements: (1) an edge $e$ is a co-loop, namely, a bond of $\Sigma$; (2) the edge $e$ belongs to every base of $\Sigma$; (3) the edge $e$ does not belong to any circuit of $\Sigma$.

\begin{lem}\label{lem:co-loop}
Let $e$ be a co-loop of $\Sigma$, contained in a component of the form $\Sigma[X)$. Then the edge $e$ is in one and only one of the following situations:
\begin{enumerate}[\hspace{5mm}\rm (a)]
\item
$e$ is an outer-edge of $\Sigma[X)$. More precisely, $\Sigma[X]=\Sigma[X)\setminus e$ and $\Sigma[X]$ is balanced.

\item
$e$ is a bridge of $\Sigma[X)$. More specifically, $e$ is between a balanced component $\Sigma_0$ of $\Sigma[X)\setminus e$
and another component $\Sigma_1$ (either balanced or unbalanced) of $\Sigma[X)\setminus e$.

\item
$e$ is neither a bridge nor an outer-edge of $\Sigma[X)$. More precisely, $\Sigma[X]=\Sigma[X)$, $\Sigma[X]$ is unbalanced, and $\Sigma[X]\setminus e$ is balanced.
\end{enumerate}
\end{lem}
\begin{proof}
It is straightforward by definition of co-loop of $\Sigma$, that is, a bond of $\Sigma$ containing exactly one edge.
\end{proof}

\begin{lem}\label{lem:co-loop characterization}
Let $e$ be an edge of a matroid $\mathcal{M}$ over a nonempty set $E$. Then the following statements are equivalent.
\begin{enumerate}[\hspace{5mm}\rm (a)]
\item
$e$ is a co-loop of $\mathcal{M}$.

\item
$e\in B$ for each base $B$ of $\mathcal{M}$.

\item
There is no circuit $C$ of $\mathcal{M}$ such that $e\in C$.
\end{enumerate}
\end{lem}
\begin{proof}
(a) $\Rightarrow$ (b): Co-loop $e$ by definition means that the singleton $\{e\}$ is a co-circuit of $\mathcal{M}$, a minimal dependent set of $\mathcal{M}^*$. Then $\{e\}$ is not independent in $\mathcal{M}^*$, that is, $e$ is not contained in any base of $\mathcal{M}^*$. Thus $e\not\in B^c$, that is, $e\in B$, for each base $B$ of $\mathcal{M}$.

(b) $\Rightarrow$ (c): Suppose there exists a circuit $C$ of $\mathcal{M}$ such that $e\in C$. Note that $C\setminus e$ is independent in $\mathcal{M}$. Then $C\setminus e$ can be extended to a base $B_1$ of $\mathcal{M}$. Applying (b), we obtain $e\in B_1$. Thus $C$ is contained in $B_1$, so $C$ is independent, a contradiction.

(c) $\Rightarrow$ (a): Suppose that $e$ is not a co-loop of $\mathcal{M}$; that is, $\{e\}$ is not dependent in $\mathcal{M}^*$; in other words, $\{e\}$ is independent in $\mathcal{M}^*$. Let $\{e\}$ be extended to a base $B^c$ of $\mathcal{M}^*$, with $B$ a base of $\mathcal{M}$. Since $e\in B^c$, there exists a unique circuit $C$ contained in $B\cup e$ such that $e\in C$, contradicting (c).
\end{proof}

\begin{lem}\label{lem:deletion-contraction-flow-group}
\begin{enumerate}[\rm (a)]
\item
If $e$ is a positive loop of $\Sigma$, then
\begin{align}
{\rm F}(\Sigma,{\Bbb A})&\cong
{\Bbb A}\oplus{\rm F}(\Sigma\setminus e,{\Bbb A}),\nonumber\\
{\rm F}_{\rm nz}(\Sigma,{\Bbb A})&\cong
({\Bbb A}\setminus 0) \oplus
{\rm F}_{\rm nz}(\Sigma\setminus e,{\Bbb A}).\label{eq:positive loop}
\end{align}

\item
If $e$ is a negative loop and also a co-loop of $\Sigma$, then
\begin{align}
{\rm F}(\Sigma,{\Bbb A})&\cong {\rm Tor}_2({\Bbb A})\oplus {\rm F}(\Sigma\setminus e,{\Bbb A}),\nonumber\\
{\rm F}_{\rm nz}(\Sigma,{\Bbb A})&\cong
({\rm Tor}_2({\Bbb A})\setminus 0) \oplus
{\rm F}_{\rm nz}(\Sigma\setminus e,{\Bbb A}).
\label{eq:both neg-loop and co-loop}
\end{align}

\item
If $e$ is a positive link edge of $\Sigma$, then
\begin{align}
{\rm F}(\Sigma,{\Bbb A})&\cong {\rm F}(\Sigma/e,{\Bbb A}),\nonumber\\
{\rm F}_{e,\rm nz}(\Sigma,{\Bbb A})&\cong
{\rm F}_{\rm nz}(\Sigma\setminus e,{\Bbb A}),\label{eq:positive link edge0}\\
{\rm F}_{\rm nz}(\Sigma/e,{\Bbb A})&\cong
{\rm F}_{\rm nz}(\Sigma,{\Bbb A})\cup {\rm F}_{e,\rm nz}(\Sigma,{\Bbb A})\;
\mbox{\rm (disjoint)}.\label{eq:positive link edge1}
\end{align}

\item
If $e$ is a bridge of $\Sigma$, then ${\rm F}(\Sigma,{\Bbb A})\cong {\rm F}(\Sigma/e,{\Bbb A})$.

\item
If $e$ is a co-loop of $\Sigma$ and is either a bridge or an outer-edge, then
\[
{\rm F}(\Sigma,{\Bbb A})= {\rm F}_e(\Sigma,{\Bbb A})\cong {\rm F}(\Sigma\setminus e,{\Bbb A})\cong {\rm F}(\Sigma/e,{\Bbb A}).
\]
\end{enumerate}
\end{lem}
\begin{proof}
(a) Since a positive loop is a circuit, according to the structure of ${\rm F}(\Sigma,{\Bbb A})$ in Theorem \ref{thm:circuits}, the isomorphism ${\rm F}(\Sigma,{\Bbb A})\cong
{\Bbb A}\oplus{\rm F}(\Sigma\setminus e,{\Bbb A})$ is given by the identity map, and the isomorphism preserves nowhere-zero flows. It follows a desired bijection for \eqref{eq:positive loop}.

(b) Since $e$ is both a negative loop and a co-loop, the edge $e$ is contained in a component of the form $\Sigma[X)$, and $\Sigma[X)\setminus e$ is balanced. It follows that $\Sigma[X)=\Sigma[X]$. According to Theorem \ref{thm:circuits}, ${\rm F}(\Sigma,{\Bbb A})\cong {\rm Tor}_2({\Bbb A})\cdot{\bf 1}_{\vec{e}}\oplus {\rm F}(\Sigma\setminus e,{\Bbb A})$, and the isomorphism preserves the nowhere-zero flows. It follows a desired bijection for \eqref{eq:both neg-loop and co-loop}.

(c) Let $v_1,v_2$ denote the endpoints of the positive link $e$, and be merged along $e$ into a new vertex $w$ of $\Sigma/e$. Denote by $e'/e$ the edge of $\Sigma/e$, obtained from $e'$ ($\neq e$) by contracting $e$. It is straightforward to see that
\[
\psi:{\rm F}(\Sigma,{\Bbb A})\rightarrow{\rm F}(\Sigma/e,{\Bbb A}),
\quad \vec{e}{\,'}\!/e\mapsto {\bf f}(\vec{e}{\,'})
\]
is a well-defined homomorphism. We claim that $\psi$ is a bijection. In fact, if $\psi({\bf f})=0$, then ${\bf f}(\vec{e}{\,'})=0$ for all $e'\neq e$; since $\bf f$ is a flow, it forces that ${\bf f}(\vec{e}\,)=0$; thus ${\bf f}=\bf 0$, so $\psi$ is injective. Given a flow ${\bf f}'$ of $\Sigma/e$. Let $E_i$ denote the set of edges of $\Sigma$ having an endpoint at $v_i$ and set
\[
a_i:=\sum_{e'\in E_i,\,e'\neq e}[v_i,\vec{e}{\,'}]{\bf f}'(\vec{e}{\,'}\!/e), \quad i=1,2.
\]
Then $\partial{\bf f}'(w)=0$ implies $a_1+a_2=0$. Now we define ${\bf f}(\vec{e}{\,'})={\bf f}'(\vec{e}{\,'}\!/e)$ for $e'\neq e$ and ${\bf f}(\vec{e}\,)=a_1$, where $\vec{e}$ points to $v_2$. Clearly, $\partial{\bf f}(v)=0$ for $v\neq v_1,v_2$; and
\begin{align*}
\partial{\bf f}(v_i)&=\sum_{e'\in E_i,\,e'\neq e}[v_i,\vec{e}{\,'}]{\bf f}(\vec{e}{\,'})+(-1)^ia_1\\
&=a_i+(-1)^ia_1=0, \quad i=1,2.
\end{align*}
This means that $\bf f$ is a flow of $\Sigma$. Clearly, $\psi({\bf f})={\bf f}'$ by the construction of $\bf f$. Thus $\psi$ is surjective. We have shown that $\psi$ is an isomorphism of ${\rm F}(\Sigma,{\Bbb A})$ onto ${\rm F}(\Sigma/e,{\Bbb A})$.

 Consider the isomorphism ${\rm F}_e(\Sigma,{\Bbb A})\cong {\rm F}(\Sigma\setminus e,{\Bbb A})$, given by deleting the $e$-coordinate. Its restriction (to nowhere-zero flows on $E\setminus e$) gives rise to a bijection
\[
{\rm F}_{e,\rm nz}(\Sigma,{\Bbb A})\cong {\rm F}_{\rm nz}(\Sigma\setminus e,{\Bbb A}).
\]
Note that a member ${\bf f}'$ of ${\rm F}(\Sigma/e,{\Bbb A})$ is nowhere-zero if and only if its corresponding flow ${\bf f}:=\psi^{-1}({\bf f}')\in{\rm F}(\Sigma,{\Bbb A})$ is nowhere-zero on $E\setminus e$, regardless of the value ${\bf f}(\vec{e}\,)$. It follows that $\bf f$ is either nowhere-zero on $E$, or $\bf f$ is zero on the edge $e$ but nowhere-zero on $E\setminus e$. We thus obtain a bijection of ${\rm F}_{\rm nz}(\Sigma/e,{\Bbb A})$ onto the disjoint union ${\rm F}_{\rm nz}(\Sigma,{\Bbb A})\cup {\rm F}_{e,\rm nz}(\Sigma,{\Bbb A})$.

(d) Let the bridge $e$ of $\Sigma$ connect two components of $\Sigma\smallsetminus e$. Choose a switching $\nu$ on one of the two components such that $e^\nu$ is positive in $\Sigma^\nu$. Then
${\rm F}(\Sigma^\nu,{\Bbb A})\cong {\rm F}(\Sigma^\nu/e^\nu,{\Bbb A})$; consequently, ${\rm F}(\Sigma,{\Bbb A})\cong {\rm F}(\Sigma/e,{\Bbb A})$.

(e) Let $\Sigma[X]$ denote a new balanced component, resulted from the removal of $e$ from $\Sigma$. Let $\nu$ be a switching on $X$ such that $\delta X^\nu=\{e\}$. Since $e$ is either a bridge or an outer-edge of $\Sigma$, we have ${\bf 1}_{\vec{e}}=\pm\delta({\nu\bf 1}_{X})$. Then for each flow ${\bf f}\in{\rm F}(\Sigma,{\Bbb A})$, we have
\[
{\bf f}(\vec{e}\,)=\langle{\bf 1}_{\vec{e}\,},{\bf f}\rangle
=\langle\pm\delta(\nu{\bf 1}_{X}),{\bf f}\rangle
=\langle\pm\nu{\bf 1}_{X},\partial{\bf f}\rangle =0.
\]
It follows that ${\rm F}(\Sigma,{\Bbb A})={\rm F}_{e}(\Sigma,{\Bbb A})$ and ${\rm F}_{e}(\Sigma,{\Bbb A})\cong {\rm F}(\Sigma\setminus e,{\Bbb A})$. If $e$ is a bridge, the isomorphism ${\rm F}(\Sigma,{\Bbb A})\cong {\rm F}(\Sigma/e,{\Bbb A})$ follows from (d). If $e$ is an outer-edge, then $\Sigma/e=\Sigma\setminus e$ and ${\rm F}(\Sigma/e,{\Bbb A})={\rm F}(\Sigma\setminus e,{\Bbb A})\cong {\rm F}(\Sigma,{\Bbb A})$.
\end{proof}

\begin{prop}\label{prop:deletion-contraction formula}
The polynomial $\varphi(\Sigma;t,x)$ satisfies the deletion-contraction formula:
\[
\varphi(\Sigma;t,x)=\left\{\begin{array}{ll}
(x-1)\varphi(\Sigma\setminus e;t,x) & \mbox{\rm if $e$ is a positive loop,}\\
(t-1)\varphi(\Sigma\setminus e;t,x) & \mbox{\rm if $e$ is a negative loop and co-loop,}\\
\varphi(\Sigma/e;t,x)-\varphi(\Sigma\setminus e;t,x) & \mbox{\rm if $e$ is a positive link.}
\end{array}\right.
\]
Moreover, $\varphi=0$ if and only if $\Sigma$ contains a bridge co-loop or an outer-edge co-loop.
\end{prop}
\begin{proof}
The first deletion-contraction formula follows from \eqref{eq:positive loop}; the second follows from \eqref{eq:both neg-loop and co-loop}; and the third follows from \eqref{eq:positive link edge0} and \eqref{eq:positive link edge1}.

``$\Leftarrow$": Let $e$ be a bridge co-loop or an outer-edge co-loop of $\Sigma$. Denote by $\Sigma[X]$ the balanced component of $\Sigma\setminus e$ resulted by the removal of $e$ from $\Sigma$. Choose a switching $\nu$ (unique up to sign) on $X$ such that all edges of $\Sigma^\nu[X]$ are positive. Then ${\bf 1}_{\vec{e}}=\pm\delta(\nu\cdot{\bf 1}_X)$, where the arc $\vec{e}$ points to $X$. Thus for each flow $\bf f$ of $\Sigma$,
\[
{\bf f}(\vec{e}\,)=\langle{\bf 1}_{\vec{e}},{\bf f}\rangle =\langle\pm\delta(\nu\cdot{\bf 1}_X),{\bf f}\rangle
=\langle\pm\nu\cdot{\bf 1}_X,\partial{\bf f}\rangle=0.
\]
It follows that ${\rm F}_{\rm nz}(\Sigma,{\Bbb A})=\varnothing$. Consequently, $\varphi(\Sigma;t,x)=0$.

``$\Rightarrow$": Let $\varphi(\Sigma;t,x)=0$. Then ${\rm F}_{\rm nz}(\Sigma,{\Bbb A})=\varnothing$ for every $\Bbb A$ finite nonzero abelian group. Suppose $\Sigma$ contains no co-loops, that is, every edge is not a co-loop. Then by Proposition \ref{lem:co-loop characterization}, every edge $e$ is contained in a circuit of $\Sigma$. Choose an edge $e_0\in\Sigma$ and a circuit $C_0$, with a direction $\omega(C_0)$ and $e_0\in C_0$. Next choose an edge $e_1\not\in C_1$ and a circuit $C_1$, with a direction $\omega(C_1)$ and $e_1\in C_1$. The choose a edge $e_2\not\in C_1\cup C_2$ and circuit $C_2$, with a direction $\omega(C_2$ and $e_2\in C_2$. Continue this procedure; we obtain circuits $C_i$, with directions $\omega(C_i)$, such that $E=\cup_{i=0}^kC_i$. Then
\[
{\bf f}:=\sum_{i=0}^k 2^i{\bf I}_{\omega(C_i)}
\]
is a nowhere-zero integer-valued flow. So ${\rm F}_{\rm nz}(\Sigma,{\Bbb Z}_n)\neq\varnothing$ for $n$ large enough, contradictory to the emptiness of ${\rm F}_{\rm nz}(\Sigma,{\Bbb Z}_n)$. We have shown that $\Sigma$ contains co-loops.

Let $e^*_1,\ldots,e^*_\ell$ denote the co-loops of $\Sigma$. Denote by $\Sigma[X_j)$ the component of $\Sigma$ that contains $e^*_j$. Suppose every member of $\{e_1^*,\ldots,e^*_\ell\}$ is neither a bridge nor an outer-edge. Then $\Sigma[X_j]=\Sigma[X_j)$ is unbalanced and $\Sigma[X_j]\setminus e^*_j$ is balanced for all $j$. Thus $\Sigma':=\Sigma\setminus\bigcup_je_j^*$ contains no co-loops. Repeating the argument in the previous paragraph, we see that ${\rm F}_{\rm nz}(\Sigma',{\Bbb A})$ is nonempty. Choose a nowhere-zero flow ${\bf f}'$ of $\Sigma'$ with coefficients in $\Bbb A$. Choose a negative circle $C_j$ of each $\Sigma[X_j]$ with $e^*_j\in C_j$. Then $a\cdot {\bf 1}_{\omega(C_j)}$ is a flow with coefficients in ${\Bbb Z}_2\cong\{0,a\}$, where $2a=0$. Thus
\[
{\bf f}:={\bf f}'+\sum_{j=1}^\ell a\cdot {\bf 1}_{\omega(C_j)}
\]
is a nowhere-zero flow of $\Sigma$ with coefficients in the group ${\Bbb Z}_2\times{\Bbb A}$. Hence ${\rm F}_{\rm nz}(\Sigma,{\Bbb Z}_2\times{\Bbb A})\neq\varnothing$. Of course, $\varphi(\Sigma;t,x)\neq 0$, which is a contradiction. We conclude that one of the co-loops $e^*_1,\ldots,e^*_\ell$ is either a bridge or an outer-edge of $\Sigma$
\end{proof}

Let $\Sigma_{m,n}$ denote the signed graph with $m$ outer-edges and $n$ negative loops at a single vertex, where $m,n\geq 0$. Set $\varphi_{m,n}(t,x):=\varphi(\Sigma_{m,n};t,x)$ with convention $\varphi_{0,0}(t,x)\equiv1$.

\begin{prop}\label{prop:Recurrences}
The sequence $\varphi_{m,n}(t,x)$ of polynomials, where $m,n\geq 0$, satisfies the following initial conditions and recurrence relations:
\begin{enumerate}[\hspace{5mm}\rm (a)]
\item
$\varphi_{1,0}=0$, $\varphi_{0,1}=t-1$, $\varphi_{1,1}=x-t$.

\item
$\varphi_{0,n}=t(x-1)^{n-1}-\varphi_{0,n-1}$,\; $n\geq 1$.

\item
$\varphi_{m,n}=(x-1)^{m+n-1}-\varphi_{m-1,n}$,\; $m\geq 1,n\geq 0$.
\end{enumerate}
\end{prop}
\begin{proof}
Choose orientations on edges such that all point to the unique vertex. Let $y_i$ denote the value on the $i$th outer-edge arc, and $z_j$ be the value of the $j$th negative loop arc, where $1\leq i\leq m$, $1\leq j\leq n$. Let ${\rm F}(m,n)$ denote the flow group of $\Sigma_{m,n}$ with coefficients in $\Bbb A$.

(a) Clearly, ${\rm F}_{\rm nz}(1,0)=\varnothing$, ${\rm F}_{\rm nz}(0,1)={\rm Tor}_2({\Bbb A})\setminus\{0\}$. Note that $\Sigma(1,1)$ is the solution set of the equation $y_1+2z_1=0$ over $\Bbb A$. If $z_1\in{\rm Tor}_2({\Bbb A})$, then $y_1=-2z_1=0$, there is no nowhere-zero solution. Let $z_1\in{\Bbb A}\setminus {\rm Tor}_2({\Bbb A})$. Then $y_1=-2z_1\neq 0$. Thus
\[
{\rm F}_{\rm nz}(1,1)=\{(-2z_1,z_1): z_1\in{\Bbb A}\setminus{\rm Tor}_2({\Bbb A})\}\simeq {\Bbb A}\setminus{\rm Tor}_2({\Bbb A}).
\]
It follows that $\varphi_{1,1}=x-t$.

(b) Let $m=0$ and $n\geq 2$. The flow group ${\rm F}(0,n)$ is the solution set of the equation
\begin{equation*}
2z_1+\cdots+2z_n=0
\end{equation*}
over $\Bbb A$. Let $z_1,\ldots,z_{n-1}\in{\Bbb A}$ and set $a_n:=z_1+\cdots+z_{n-1}$. We claim that all solutions of the equation $2z_n+2a_n=0$ for the variable $z_n$ are given by
\begin{equation}\label{eq:znaan}
z_n=a-a_n,\quad \mbox{where}\quad a\in{\rm Tor}_2({\Bbb A}).
\end{equation}
In fact, the solutions given by \eqref{eq:znaan} are clearly solutions of $2z_n+2a_n=0$. Let $z_n=s$ be any solution, that is,
\[
2s+2a_n=2(s+a_n)=0.
\]
This means that $s+a_n\in{\rm Tor}_2({\Bbb A})$. Write $s+a_n=a$ with $a\in{\rm Tor}_2({\Bbb A})$. Then $s=a-a_n$ with $a\in{\rm Tor}_2({\Bbb A})$ has the required form. We thus obtain an isomorphism $\psi:{\rm F}(0,n)\rightarrow{\Bbb A}^{n-1}\times{\rm Tor}_2({\Bbb A})$,
\[
\psi(z_1,\ldots,z_n)=(z_1,\ldots,z_{n-1},a),\; \mbox{where}\; a=z_n+\mbox{$\sum_{i=1}^{n-1}$}z_i.
\]
Let $S=({\Bbb A}\setminus 0)^{n-1}\times{\rm Tor}_2({\Bbb A})$. Then $\psi({\rm F}_{\rm nz}(0,n))$ consists of the members $(z_1,\ldots,z_{n-1},a)\in S$ such that $a\neq\sum_{i=1}^{n-1}z_i$.
On the other hand, ${\rm F}_{\rm nz}(0,n-1)$ can be identified as a subset of ${\rm F}(0,n)$, defined by
\[
{\rm F}_{\rm nz,0}(0,n):={\rm F}(0,n)\cap(({\Bbb A}\setminus 0)^{n-1}\times 0).
\]
Then we have
\begin{align*}
\psi({\rm F}_{\rm nz}(0,n))&=\big\{(a_1,\ldots,a_{n-1},a)\in S: a\neq\mbox{$\sum_{i=1}^{n-1}$}a_i\big\},\\
\psi({\rm F}_{\rm nz,0}(0,n))&=\big\{(a_1,\ldots,a_{n-1},a)\in S: a=\mbox{$\sum_{i=1}^{n-1}$}a_i\big\}.
\end{align*}
It follows that $\psi({\rm F}_{\rm nz}(0,n))=S\setminus \psi({\rm F}_{\rm nz,0}(0,n))$. Hence
\[
|{\rm F}_{\rm nz}(0,n)|=|{\rm Tor}_2{\Bbb A}|\cdot(|{\Bbb A}|-1)^{n-1}-|{\rm F}_{\rm nz}(0,n-1)|.
\]
We obtain the recurrence $\varphi_{0,n}=t(x-1)^{n-1}-\varphi_{0,n-1}$.

(c) The case that $m=1$ and $n=0$ is trivial by (a). Let $m\geq 1$ and $m+n\geq 2$. The flow group ${\rm F}(m,n)$ is the solution set of the equation
\[
y_m+\sum_{i=1}^{m-1}y_i+2\sum_{j=1}^nz_j=0
\]
over $\Bbb A$. Notice that $y_m=-\sum_{i=1}^{m-1}y_i-2\sum_{j=1}^nz_j$ is determined by $y_1,\ldots,y_{m-1},z_1,\ldots,z_n$. Let $y_i,z_j\in{\Bbb A}\setminus\{0\}$, where $1\leq i\leq m-1$ and $1\leq j\leq n$. If
\[
\sum_{i=1}^{m-1}y_i+2\sum_{j=1}^nz_j=0 \quad \mbox{[solution set ${\rm F}(m-1,n)$]},
\]
there is no choice for $y_m$ to be nonzero. However, if  $\sum_{i=1}^{m-1}y_i+2\sum_{j=1}^nz_j\neq 0$, there is exactly one choice for $y_m$ to be nonzero. It follows that
\[
{\rm F}_{\rm nz}(m,n)\simeq ({\Bbb A}\setminus 0)^{m+n-1}\setminus {\rm F}_{\rm nz}(m-1,n).
\]
Thus we conclude the recurrence $\varphi_{m,n}=(x-1)^{m+n-1}-\varphi_{m-1,n}$.
\end{proof}

\begin{cor}\label{cor:varphi-mn}
The sequence $\varphi_{m,n}(t,x)$ of polynomials, where $m,n\geq 0$, has the following explicit expressions
\begin{align*}
\varphi_{m,0}&=\frac{(x-1)^m-(-1)^m}{x}+(-1)^m,\\
\varphi_{0,n}&=t\cdot\frac{(x-1)^n-(-1)^n}{x}+(-1)^n,\\
\varphi_{m,n}&=(x-1)^n\cdot\frac{(x-1)^m-(-1)^m}{x}
+(-1)^mt\cdot\frac{(x-1)^n-(-1)^n}{x} +(-1)^{m+n}.
\end{align*}
\end{cor}
\begin{proof}
Apply the recurrence relations (b) and (c) of Proposition~\ref{prop:Recurrences} to obtain the formulas for  $\varphi_{0,n}$ and $\varphi_{m,n}$ respectively.
\end{proof}

\begin{defn}
The {\em negative cycle-rank} of $\Sigma$, denoted ${\rm cr}^-(\Sigma)$, is the minimal cardinality $|S|$ of an edge subset $S$ such that $\Sigma\setminus S$ contains no negative circles. The {\em positive cycle-rank} of $\Sigma$ is the number
\begin{equation}
{\rm cr}^+(\Sigma):={\rm cr}_{gr}(\Sigma)-{\rm cr}^-(\Sigma),
\end{equation}
where ${\rm cr}_{gr}(\Sigma)$ denotes the cycle-rank of the underlying unsigned graph $(V(\Sigma),E(\Sigma))$.

The {\em $i$th negative cycle-rank} of $\Sigma$ is the minimal cardinality $|S|$ of an edge subset such that $\Sigma\smallsetminus S$ contains exactly $i$ unbalanced components without outer edges.
\end{defn}

The negative cycle-rank ${\rm cr}^-(\Sigma)$ is known as the {\em frustration} of $\Sigma$ in the literature, see \cite{Zaslavsky-signed-graphs-DAM}.

\begin{thm}[Degree Property]
Let $\Sigma$ be connected and $\varphi(\Sigma;t,x)\neq 0$. Then $\varphi(\Sigma;t,x)$ can be written as the form
\begin{equation}\label{eq:flow-poly-t}
\varphi(\Sigma;t,x)=\sum_{i=0}^{d}t^i\varphi_i(\Sigma,x)
\end{equation}
with $\varphi_d\neq 0$ and $d\leq {\rm cr}^-(\Sigma)$ (negative cycle-rank of $\Sigma$). Moreover,
\begin{enumerate}[\hspace{5mm}\rm (a)]
\item
If $\Sigma$ contains no negative circles, then $\varphi(\Sigma;t,x)=\varphi_0(\Sigma,x)$ and is a monic polynomial of degree ${\rm cr}(\Sigma)$.

\item
If $\Sigma$ contains outer-edges, then $\varphi_0(\Sigma,x)$ is monic polynomial of degree ${\rm cr}(\Sigma)$.

\item
If $\Sigma$ contains negative circles and no outer-edges, then the degree of $\varphi_0$ is bounded by ${\rm cr}^+(\Sigma)$, and  $\varphi_0$ can be zero.

\item
If $\Sigma$ contains negative circles and no outer-edges, then $\varphi_1$ is a monic polynomial of degree ${\rm cr}(\Sigma)$.
\end{enumerate}
\end{thm}
\begin{proof}
According to the expansion formula \eqref{eq:expansion}, the power of $t$ is bounded by the maximum $u_c(S)$, the number of unbalanced components without outer-edges of spanning subgraphs $(V,S)$ of $\Sigma$. Each unbalanced component without outer-edges contains at least one negative circle, which needs deleting one of its edges to be destroyed. It follows that the maximum $d$ with $\varphi_d\neq 0$ is bounded by the minimal number $|S|$ of edge subset such that $\Sigma\smallsetminus S$ contains no negative circles.

Next we verify (a)-(d) for signed graphs $\Sigma_{m,n,p}$, which consists of exactly one vertex with $m$ outer-edges, $n$ negative loops, and $p$ positive loops. The flow polynomial of $\Sigma_{m,n,p}$ has the form
\[
\bigg[(x-1)^{n}\cdot\frac{(x-1)^m-(-1)^m}{x}+(-1)^{m+n} +(-1)^mt\cdot\frac{(x-1)^n-(-1)^n}{x}\bigg](x-1)^p.
\]
Then $\varphi=\varphi_0+t\varphi_1$ with
\begin{align}
\varphi_0(\Sigma_{m,n,p},x)&=
(x-1)^{n+p}\cdot\frac{(x-1)^m-(-1)^m}{x}+(-1)^{m+n}(x-1)^p,
\label{eq:Smnp-0}\\
\varphi_1(\Sigma_{m,n,p},x)&= (-1)^m\cdot
\frac{(x-1)^n-(-1)^n}{x}\cdot(x-1)^p. \label{eq:Smnp-1}
\end{align}
It is clear that $\varphi_1\neq 0$ if and only if $n\geq 1$. We obtain the highest mixed degree as
\[
\deg_{\rm mix}\varphi(\Sigma_{m,n,p};t,x)
=n+p={\rm cr}_c(\Sigma)+1, \quad n\geq 1.
\]
The leading coefficient of the highest mixed degree has the sign $(-1)^m$.

Note that $\varphi(\Sigma_{m,n,p};t,x)=0$ if and only if $(m,n)=(1,0)$. It is easy to see that $\varphi_0\neq 0$ if and only $(m,n)\neq(1,0)$. More specifically, for $m=0$, we have $\varphi_0(\Sigma_{0,n,p},x)=(-1)^n(x-1)^p$. For $n=0$ in particular, $\varphi_0(\Sigma_{0,0,p},x)=(x-1)^p$ has leading coefficient $1$ and \begin{align}
\deg\varphi_0(\Sigma_{0,0,p},x)&=p={\rm cr}(\Sigma).
\label{eq:induction-base-00p}
\end{align}
For $m=1,n\geq 1$ or $m\geq 2$, the polynomial $\varphi_0$ has leading coefficient $1$ and its degree is read out from \eqref{eq:Smnp-0} as
\begin{equation}
\deg\varphi_0(\Sigma_{m,n,p},x)=m+n+p-1={\rm cr}(\Sigma).\label{eq:induction-base-mnp}
\end{equation}

The flow polynomial is invariant under switching. To prove (a)-(d), let $\nu$ be a switching such that the number of negative edges of $\Sigma^\nu$ is minimal. For simplicity, we may assume that the number of negative edges of $\Sigma$ is already minimal. Choose a spanning tree $T$ of $\Sigma[V]$ such that all edges of $T$ are positive, which is always possible. For each edge $e$ of $T$, let $B_{gr}(T,e)$ denote the fundamental graph bond of $\Sigma[V]$ with respect to $T$. The edge $e$ is a unique edge of $T$ contained in the graph bond $B_{gr}(T,e)$. Applying deletion-contraction, Theorem \ref{prop:deletion-contraction formula} implies
\begin{equation}\label{eq:deletion-contraction-deg}
\varphi(\Sigma;t,x)=\varphi(\Sigma/e;t,x)
-\varphi(\Sigma\setminus e;t,x).
\end{equation}
Collecting the terms $\pm t^ix^j$ for all $j$ in the expansion of $\varphi$, we obtain
\begin{equation}\label{eq:deletion-contraction0}
\varphi_i(\Sigma,x)=\varphi_i(\Sigma/e,x)
-\varphi_i(\Sigma\setminus e,x), \quad i=0,1,\ldots,d.
\end{equation}
Since $\varphi(\Sigma;t,x)\neq 0$, the signed graph $\Sigma$ contains no co-loops. There are two situations on the signed graph $\Sigma\smallsetminus e$.

(1) $\Sigma\smallsetminus e$ is connected. Then $B_{gr}(T,e)$ contains at least two edges. Thus
\begin{equation}\label{eq:cycle-rank inequality contraction}
{\rm cr}(\Sigma/e)= {\rm cr}(\Sigma), \quad
{\rm cr}(\Sigma\smallsetminus e)={\rm cr}(\Sigma)-1;
\end{equation}
\begin{equation}\label{eq:gr-cycle-rank inequality contraction}
{\rm cr}_{cgr}(\Sigma/e)={\rm cr}_{cgr}(\Sigma),\quad
{\rm cr}_{cgr}(\Sigma\smallsetminus e)={\rm cr}_{cgr}(\Sigma)-1.
\end{equation}

(2) $\Sigma\smallsetminus e$ is disconnected, that is, $e$ is a bridge of $\Sigma[V]$. Then the two components of $\Sigma\smallsetminus e$, denoted $\Sigma_1$ and $\Sigma_2$, must be unbalanced, due to the fact that no co-loops in $\Sigma$. It follows that ${\rm cr}(\Sigma_i)=|E(\Sigma_i)|-|V(\Sigma_i)|$, $i=1,2$. Thus
\begin{equation}\label{eq:cycle-rank inequality deletion}
{\rm cr}(\Sigma_1)+{\rm cr}(\Sigma_2)
=|E(\Sigma)|-|V(\Sigma)|-1={\rm cr}(\Sigma)-1.
\end{equation}

Now we prove (a)-(d) by induction on the number of edges.

(a) and (b): $\Sigma$ contains no negative circles or $\Sigma$ contains outer-edges. Our aim is to show that $\deg\varphi_0={\rm cr}(\Sigma)$. In the former case $\Sigma$ is equivalent to its underlying unsigned graph with or without outer-edges; consequently, $\varphi(\Sigma;t,x)=\varphi_0(\Sigma,x)$. The two cases have the common case that $\Sigma$ contains outer-edges and no negative circles, which are ordinary unsigned graphs with outer-edges. In both case, the induction bases are valid as \eqref{eq:induction-base-00p} and \eqref{eq:induction-base-mnp}.

In the case that $\Sigma\smallsetminus e$ is connected, by induction and \eqref{eq:cycle-rank inequality contraction}, we obtain
\begin{align*}
\deg\varphi_0(\Sigma\smallsetminus e,x)
&={\rm cr}(\Sigma\smallsetminus e)<{\rm cr}(\Sigma)\\
&={\rm cr}(\Sigma/e) =\deg\varphi_0(\Sigma/e,x).
\end{align*}
In the case that $\Sigma\smallsetminus e$ is disconnected with two components $\Sigma_1$ and $\Sigma_2$, by the product formula \eqref{eq:product-formula-of-flow-poly} we have
\[
\varphi(\Sigma\smallsetminus e;t,x) =\varphi(\Sigma_1;t,x)\varphi(\Sigma_2;t,x).
\]
Consequently,
$\varphi_0(\Sigma\smallsetminus e,x) =\varphi_0(\Sigma_1,x)\varphi_0(\Sigma_2,x)$.
According to \eqref{eq:cycle-rank inequality deletion}, we obtain
\begin{align*}
\deg\varphi_0(\Sigma\smallsetminus e,x)&=
\deg\varphi_0(\Sigma_1,x)+\deg\varphi_0(\Sigma_1,x)\\
&={\rm cr}(\Sigma_1)+{\rm cr}(\Sigma_2)<{\rm cr}(\Sigma)\\
&={\rm cr}(\Sigma/e)=\deg\varphi_0(\Sigma/e,x).
\end{align*}
Thus $\deg\varphi_0(\Sigma,x)=\deg\varphi_0(\Sigma/e,x)={\rm cr}(\Sigma)$.

(c) $\Sigma$ contains negative circles and no outer-edges.
Note that $\varphi_0$ admits the expansion
\[
\varphi_0(\Sigma,x)=\sum_{S\subseteq E, u(\Sigma\smallsetminus S)=0} (-1)^{|S|}x^{{\rm cr}(\Sigma\smallsetminus S)}.
\]
For each edge subset $S$ such that $\Sigma\smallsetminus S$ is balanced, let $S_0$ be a minimal subset of $S$ such that $\Sigma\smallsetminus S_0$ is balanced. Then $\Sigma\smallsetminus S_0$ must be connected. Otherwise, one of edges of $S_0$, say $e_0$, is between two components of $\Sigma\smallsetminus S_0$. Then $S_1:=S_0\smallsetminus e_0$ is a proper subset of $S_0$ such that $\Sigma\smallsetminus S_1$ is balanced, contradictory to the minimality of $S_0$. Clearly, ${\rm cr}^-(\Sigma)\leq |S_0|$. Choose a spanning tree $T_0$ of $\Sigma\smallsetminus S_0$. We obtain
\begin{align*}
{\rm cr}(\Sigma\smallsetminus S)&\leq {\rm cr}(\Sigma\smallsetminus S_0)={\rm cr}_{gr}(\Sigma)-|S_0|\\
&\leq {\rm cr}_{gr}(\Sigma)-{\rm cr}^-(\Gamma)
={\rm cr}^+(\Gamma).
\end{align*}
It follows that $\deg \varphi_0\leq {\rm cr}^+(\Gamma)$. The degree of $\varphi_0$ can be strictly less than ${\rm cr}^+(\Gamma)$, and even $\varphi_0=0$. For instance, for $C_1PC_2$, $\varphi=tx-t^2$ with $\varphi_0=0$, $\varphi_1=x$, and $\varphi_2=-1$.

(d) $\Sigma$ contains negative circles and no outer-edges. Our aim is to show $\deg\varphi_1={\rm cr}(\Sigma)$. We claim that $\deg\varphi_1(\Sigma\smallsetminus e,x)
<\deg\varphi_1(\Sigma/e,x)$.

Let $\Sigma\smallsetminus e$ be connected. According to  \eqref{eq:gr-cycle-rank inequality contraction}, we obtain
\begin{align*}
\deg\varphi_1(\Sigma\smallsetminus e,x)&=
{\rm cr}_{gr}(\Sigma\smallsetminus e)<{\rm cr}_{gr}(\Sigma)\\
&={\rm cr}_{gr}(\Sigma/e)=\deg\varphi_1(\Sigma/e,x).
\end{align*}
Let $\Sigma\smallsetminus e$ be disconnected with two components $\Sigma_1$ and $\Sigma_2$. The product formula \eqref{eq:product-formula-of-flow-poly} implies
\[
\varphi_1(\Sigma\smallsetminus e,x) =\varphi_1(\Sigma_1,x)\varphi_0(\Sigma_2,x) +\varphi_0(\Sigma_1,x)\varphi_1(\Sigma_2,x).
\]
It follows that
\begin{align*}
\deg\varphi_1(\Sigma\smallsetminus e,x) & \leq \max\big\{
\deg\varphi_1(\Sigma_1,x) + \deg\varphi_0(\Sigma_2,x), \\
&\hspace{1.7cm} \deg\varphi_0(\Sigma_1,x) + \deg\varphi_1(\Sigma_2,x)\big\}\\
&\leq \max\big\{{\rm cr}_{gr}(\Sigma_1)+{\rm cr}^+(\Sigma_2), {\rm cr}^+(\Sigma_1) + {\rm cr}_{gr}(\Sigma_2)\big\}.
\end{align*}
Since $\Sigma$ contains no co-loops and the edge $e$ is a bridge of $\Sigma[V]$, both $\Sigma_1$ and $\Sigma_2$ must contain negative circles. Recall the graph cycle-rank formula and the positive cycle-rank formula
\[
{\rm cr}_{gr}(\Sigma)=|E(\Sigma)|-|V(\Sigma)|+1, \quad
 {\rm cr}^+(\Sigma)\leq |E(\Sigma)|-|V(\Sigma)|.
\]
for signed graph $\Sigma$ with negative circles and no outer-edges. We obtain
\begin{align*}
\deg\varphi_1(\Sigma\smallsetminus e,x) & \leq
|E(\Sigma_1)|-|V(\Sigma_1)|+|E(\Sigma_2)|-|V(\Sigma_2)|\\
&=|E(\Sigma)|-|V(\Sigma)|-1 \\
&<{\rm cr}_{gr}(\Sigma)={\rm cr}_{gr}(\Sigma/e) =\deg\varphi_1(\Sigma/e,x).
\end{align*}
Now according to the deletion-contraction formula \eqref{eq:deletion-contraction0}, by induction we obtain
\[
\deg\varphi_1(\Sigma,x)
= \deg\varphi_1(\Sigma/e,x)
={\rm cr}(\Sigma/e)
={\rm cr}(\Sigma).
\]
\end{proof}

\begin{thm}
Let $\Sigma$ contain no co-loops, that is, $\varphi(\Sigma;t,x)\neq 0$. Then the counting function $\varphi(\Sigma,n):=|{\rm F}_{\rm nz}(\Sigma,{\Bbb Z}_n)|$, $n\geq 2$, determines a unique quasi-polynomial $\varphi(\Sigma,n)$  of degree ${\rm cr}(\Sigma)$, $n\in{\Bbb Z}$; that is,
\[
\varphi(\Sigma,n)=\sum_{k=0}^{{\rm cr}(\Sigma)}a_k(n)n^k, \quad
n\in{\Bbb Z},
\]
where $a_k:{\Bbb Z}\rightarrow{\Bbb Z}$ are periodic functions of period $2$ if $\Sigma$ contains negative circles, and of period $1$ if $\Sigma$ contains no negative circles.
\end{thm}
\begin{proof}
Let $\varphi_{\rm odd}(\Sigma,n)=\sum_{k=0}^{{\rm cr}(\Sigma)}a'_kn^k$ for odd $n\geq 3$, and $\varphi_{\rm even}(\Sigma,n)=\sum_{k=0}^{{\rm cr}(\Sigma)}a''_kn^k$ for even $n\geq 2$. For integers $k\geq 0$, define functions $a_k:{\Bbb Z}\rightarrow{\Bbb Z}$ by
\[
a_k(n)=\left\{\begin{array}{ll}
a'_k & \mbox{if $n$ is odd},\\
a''_k & \mbox{if $n$ is even}.
\end{array}\right.
\]
Clearly, each $a_k(n)$ is a periodic function of period $2$; its minimal period might be $1$. This means that $\varphi(\Sigma,n)$ is a quasi-polynomial in the sense of Stanley \cite[p.\,201]{Stanley1}

If $\Sigma$ contains no negative circles, it is trivial that $\varphi(\Sigma,n)$ is a polynomial function in terms of integer variable $n\geq 2$, for the (compact-unbalance rank) function $u_c$ is identically zero. It follows that $\varphi(\Sigma,n)$ for $n\in{\Bbb Z}$ is a polynomial function in terms of integer variable $n$, and $a_k(n)$ are periodic functions of period $1$.

Let $\Sigma$ contain some negative circles. Then $\varphi(\Sigma,n)$ is nonzero. Note that both polynomials $\varphi_{\rm odd}$ and $\varphi_{\rm even}$ have degree $d$ equal to ${\rm cr}(\Sigma)$, the cycle-rank of $\Sigma$. For the sake of simplicity we assume that $\Sigma$ is connected. Then $a'_d=1$ and $a''_d=2$. Thus $a_d$ is a periodic function of minimal period $2$.
\end{proof}

\noindent
{\bf Example 1.}
The polynomial flow $\varphi_{m,n}(t,x)$ of the signed graph with $m$ outer-edges and $n$ negative loops at a single vertex are given as follows:
\begin{enumerate}[(1)]
\item
$\varphi_{0,0}=1$, $\varphi_{1,0}=0$, $\varphi_{0,1}=t-1$, $\varphi_{1,1}=x-t$.

\item
$\varphi_{2,0}=x-1$, $\varphi_{1,1}=x-t$, $\varphi_{0,2}=tx-2t+1$.

\item
$\varphi_{3,0}=x^2-3x+2$, $\varphi_{2,1}=x^2-3x+t+1$, $\varphi_{1,2}=x^2-2x-tx+2t$,
$\varphi_{0,3}=tx^2-3tx+3t-1$.

\item
$\varphi_{4,0}=x^3-4x^2+6x-3$, $\varphi_{3,1}=x^3-4x^2+6x-t-2$,
$\varphi_{2,2}=x^3-4x^2+tx+5x-2t-1$,
$\varphi_{1,3}=x^3-tx^2-3x^2+3tx+3x-3t$, $\varphi_{0,4}=t(x-2)(x^2-2x+2)+1$.
\end{enumerate}

\noindent
{\bf Example 2.} Applying the deletion-contraction formula of Proposition~\eqref{prop:deletion-contraction formula} to the following three signed graphs,
\begin{figure}[h]
\centering \includegraphics[width=40mm]{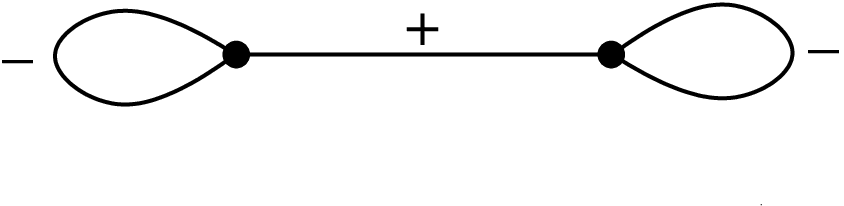}\hspace{5mm}
\includegraphics[width=40mm]{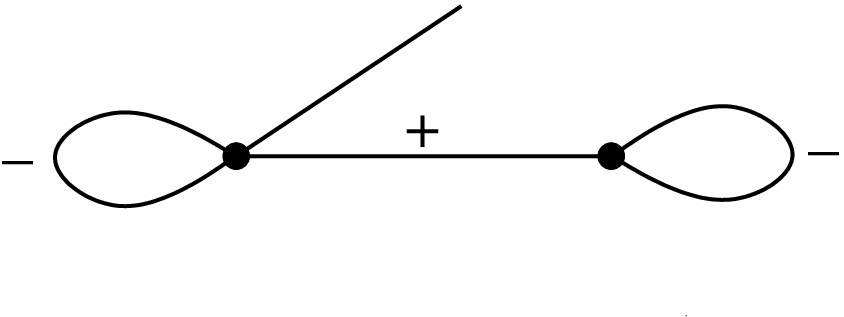}\hspace{5mm}
\includegraphics[width=40mm]{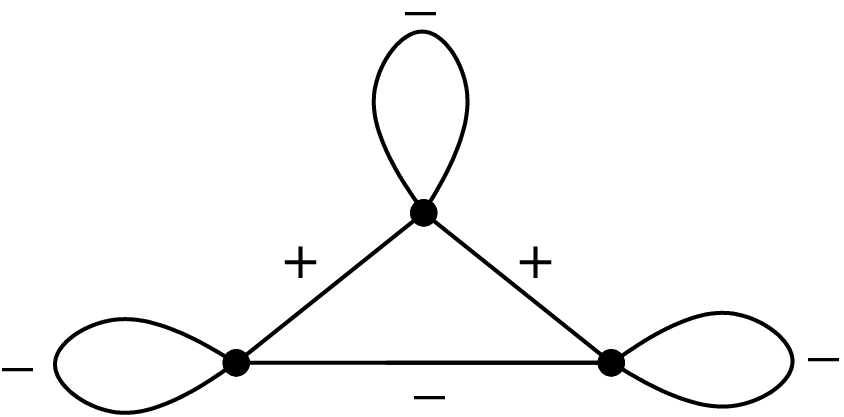} \caption{Three signed graphs.}
\label{Fig:}
\end{figure}
the left one has the flow polynomial $\varphi=t(x-t)=tx-t^2$, the middle one has the flow polynomial
\[
\varphi=(x-t-1)(x-t)=x^2-2tx+t^2-x+t,
\]
and the right one has the flow polynomial
\begin{align*}
\varphi&=\varphi_{0,4}-2\varphi_{0,3}+3\varphi_{0,1}\varphi_{0,2} -\varphi_{0,1}^3\\
&=tx^3-6tx^2+3t^2x-t^3+9tx-3t^2-4t+1.
\end{align*}


\begin{thebibliography}{99}

\bibitem{Beck-Zaslavsky-JCTB}
M.~Beck and T.~Zaslavsky, The number of nowhere-zero flows on graphs
and signed graphs, {\em J. Combin. Theory Ser.\,B\,} {\bf 96} (2006),
901--918.

\bibitem{Berge1}
C. Berge, {\em Graphs,} 2nd ed., North-Holland, Amsterdam, 1985.


\bibitem{Bolker-Zaslavsky} E.D. Bolker and T. Zaslavsky, A simple
algorithm that proves half-integrality of bidirected network
programming, {\em Networks} {\bf 48} (2006), 36--38.


\bibitem{Bouchet}
A. Bouchet, Nowhere-zero integral flows on a bidirected graph.
\emph{J.\ Combin.\ Theory Ser.~B} {\bf 34} (1983), 279--292.







\bibitem{Chen-DCG}
B. Chen, Lattice points, Dedekind sums, and Ehrhart polynomials of lattice polyhedra, {\em Discrete \& Comput. Geom.} {\bf 28} (2002), 175--199.




\bibitem{Chen-SIAM}
B.~Chen, Conformal decompositions of integral tensions and
potentials of signed graphs, {\em SIAM J. Discrete Math.} {\bf 31}
(2018), 2457--2478.

\bibitem{Chen-GC}
B.~Chen, Conformal decomposition of integral flows on signed
graphs with outer-edges. {\em Graphs \& Combin.} {\bf 37}
(2021), 2207--2225.

\bibitem{Chen-Li-Ars-Comb}
B.~Chen and S.~Li, The number of nowhere-zero tensions on graphs and signed graphs, {\em Ars Combin.} {\bf 102} (2011), 47--64.


\bibitem{Chen-Wang-EJC}
B.~Chen and J.~Wang, The flow and tension spaces and lattices of
signed graphs, {\em European J. Combin.} {\bf 30} (2009), 263--279.

\bibitem{Chen-Wang-DAM1}
B.~Chen and J.~Wang, Torsion formulas for signed graphs.
\emph{Discrete Appl. Math.}\; {\bf 158} (2010), 1148--1157.

\bibitem{Chen-Wang-arXiv}
B.~Chen and J. Wang, Classification of indecomposable integral flows
on signed graphs, unpublished manuscript, 2013. arXiv:1112.0642.

\bibitem{Chen-Wang-Zaslavsky-DM}
B.~Chen, J.~Wang and T.~Zaslavsky, Resolution of indecomposable
Integral flows on signed graphs, {\em Discrete Math.} {\bf 340}
(2017), 1271--1286.



\bibitem{DeVos-Rollova-Samal}
M.~DeVos, E.~Rollov\'{a}, and R.~\u{S}\'{a}mal, A note on counting
flows in signed graphs, {\em Electronic J.
Combin.} {\bf 26(2)} (2019), \#P2.38

\bibitem{Edmonds-Johnson}
J.~Edmonds and E.L.~Johnson, Matching: a well-solved class of linear programs, in {\em Combinatorial Structures and Their Applications}, Proceedings of the Calgary Symposium, June 1969, New York: Gordon and Breach.

\bibitem{Geelen-Guenin} J.F. Geelen and B. Guenin,
Packing odd circuits in Eulerian graphs. {\em J. Combin. Theory Ser. B}\; {\bf 86} (2002), 280--295.


\bibitem{Goodall et al}
A. Goodall, B. Litjens, G. Regts, and L. Vena,
Tutte's dichromate for signed graphs, {\em Discrete Appl. Math.} {\bf 289} (2021), 153--184.

\bibitem{Harary}
F. Harary, On the notion of balance of a signed graph, {\em Michigan
Math. J.}\ {\bf 2} (1954), 143--146.


\bibitem{Kochol1}
M. Kochol, Tension polynomials of graphs, {\em J. Graph Theory}\;
{\bf 40} (2002), 137--146.


\bibitem{Kochol3}
M. Kochol, Tension-flow polynomials on graphs, {\em Discrete Math.}
{\bf 274} (2004), 173--185.



\bibitem{Macajova-Skoviera}
E. M\'{a}\v{c}ajov\'{a} and M. \v{S}koviera, Characteristic flows on
signed graphs and short circuit covers, {\em Electronic J. Combin.}
{\bf 23} (3) (2016), \#P3.30

\bibitem{Oxley-matroid}
J.~Oxley, {\em Matroid Theory}, 1991.


\bibitem{Ren-Qian}
X.~Ren and J.~Qian, Flow polynomials of a signed graph, {\em Electronic J.
Combin.} {\bf 26(3)} (2019), \#P3.37

\bibitem{Stanley1}
R.P. Stanley, {\em Enumerative Combinatorics I}, Cambridge University Press, 1979.

\bibitem{Tutte-class-of-abelian-group}
W.T. Tutte, A class of abelian groups, {\em Candian J. Math.} {\bf
8} (1956), 13--28.

\bibitem{Tutte-Homotopy Theorem}
W.T. Tutte, A homotopy theorem for matroids II, {\em Trans.
Amer. Math. Soc.} {\bf 88} (1958), 161--174.

\bibitem{Tutte-Matroid-Theory}
W.T. Tutte, {\em Introduction to the Theory of Matroids,} RAND
Report, 1966.

\bibitem{Tutte-graph-theory}
W.T. Tutte, {\em Graph Theory}, Cambridge University Press, 2001.


\bibitem{Zaslavsky-signed-graphs-DAM}
T. Zaslavsky, Signed graphs, \emph{Discrete Appl. Math.} {\bf 4}
(1982), 47--74. Erratum, \emph{Discrete Appl.\ Math.} {\bf 5}
(1983), 248.

\bibitem{Zaslavsky-signed-graph-coloring-DM}
T. Zaslavsky, Singed graph coloring, {\em Discete Math.}\ {\bf 39}
(1982), 215--228.


\bibitem{Zaslavsky-orientation-EJC}
T. Zaslavsky, Orientation of signed graphs. \emph{European J.
Combin.}\ {\bf 12} (1991), 361--375.
\end{thebibliography}
\end{document}